\documentclass[reqno]{amsart}
\usepackage{amsmath,amssymb}
\usepackage{graphicx,amssymb}
\usepackage[dvipsnames]{xcolor}
\usepackage[mathscr]{euscript}
\usepackage{hyperref}
\usepackage[normalem]{ulem}
\hypersetup{colorlinks,%
            citecolor=gray!5!orange,%
            filecolor=yellow,%
            linkcolor=black!30!blue,%
            urlcolor=blue}      
\usepackage[left=3cm, right=3cm, top=3cm, bottom=3cm]{geometry}
\usepackage{stmaryrd}
\usepackage{enumitem}
\usepackage{graphicx}
\usepackage{cancel}

\newtheorem{theorem}{Theorem}[section]
\newtheorem{lemma}[theorem]{Lemma}
\newtheorem{proposition}[theorem]{Proposition}
\newtheorem{corollary}[theorem]{Corollary}

\newtheorem{definition}[theorem]{Definition}
\newtheorem{remark}[theorem]{Remark}

\usepackage{calc}

\numberwithin{equation}{section}

\newcommand{\mc}[1]{{\mathcal #1}}

\newcommand{\mf}[1]{{\mathfrak #1}}
\newcommand{\mb}[1]{{\mathbf #1}}
\newcommand{\bb}[1]{{\mathbb #1}}

\DeclareMathOperator{\Tr}{Tr}
\newcommand{\ent}{\mathop{\mathrm{Ent}}\nolimits}

\newcommand{\ch}{\mathcal{H}}

\newcommand{\cl}{\mathcal{L}}

\newcommand{\E}{\mathbb{E}}

\newcommand{\N}{\mathbb{N}}
\newcommand{\R}{\mathbb{R}}

\newcommand{\W}{\mathbb{W}}
\newcommand{\X}{\mathbb{X}}

\newcommand{\Z}{\mathbb{Z}}
\newcommand{\LL}{\mathbb{L}}

\newcommand{\PP}{\mathbb{P}}

\newcommand{\wg}{{\widehat G}}

\newcommand{\wrho}{\widehat \rho}
\newcommand{\wm}{{\widehat {\mc M}}_0}
\newcommand{\wb}{\widehat b}
\newcommand{\wgamma}{\widehat \gamma}
\newcommand{\wpi}{{\widehat \pi}}

\newcommand{\vecte}{{e}}

\let\L=\Lambda

\let\s=\sigma

\newcommand{\<}{\langle}
\renewcommand{\>}{\rangle}

\newcommand{\1}{\,\rlap{\small 1}\kern.13em 1}

\newcommand{\sqr}[2]{{\vcenter{\hrule height.#2pt%
                      \hbox{\vrule width.#2pt height#1pt\kern#1pt%
                            \vrule width.#2pt}%
                      \hrule height.#2pt}}}

\renewcommand\footnotemark{}

\usepackage{mathtools}

\newcommand\numberthis{\addtocounter{equation}{1}\tag{\theequation}}

\newcommand \al {\alpha}

\newcommand \ga {\gamma}
\newcommand \Ga {\Gamma}

\newcommand \la {\lambda}
\newcommand \La {\Lambda}

\newcommand \grad {\nabla}

\definecolor{ForestGreen}{RGB}{34,139,34}

\newcommand{\what}[1]{\widehat{#1}}
\newcommand{\scal}[3]{\left\langle #2, #3 \right\rangle_{#1} }

\newcommand{\I}[1]{I_T^{\what E}\left( #1 \middle| \what\gamma\right)}

\newcommand{\norm}[2]{\left\lVert#2\right\rVert_{#1}}
\newcommand{\scall}[3]{\left\langle \left\langle #2, #3 \right\rangle \right\rangle_{#1} }
\newcommand{\DM}{\mathcal{D}([0,T],\mathcal{M}^2)}
\newcommand{\abs}[1]{\left\lvert #1 \right\rvert}
\newcommand{\Q}{\mathcal{Q}}
\renewcommand{\H}{\mathbb H_{0,-}(\Sigma(\what\pi))}
\newcommand{\Itilde}[1]{\tilde I_T^{\what E}(#1|\what \gamma)}
\newcommand{\vare}{\varepsilon}

\newcommand{\Mzero}{\widehat{\mathcal{M}}_0}
\newcommand{\Ea}{\mathfrak{E}_a}
\newcommand{\EaC}{\mathfrak{E}_{aC^{-1}}}
\renewcommand{\ent}[1]{\left\lfloor #1 \right\rfloor}

\makeatletter
\def\l@paragraph{\@tocline{4}{0pt}{1pc}{7pc}{}}
\def\l@subparagraph{\@tocline{5}{0pt}{1pc}{7pc}{}}
\makeatother
\begin{document}
\title[LDP for Boundary driven weakly asymmetric Blume-Capel model]
{Boundary driven weakly asymmetric Blume-Capel model: large deviations for mixed Dirichlet-Neumann boundary conditions}

\author[M. Mourragui and N. Pr\'evost]{Mustapha Mourragui and Nicolas Pr\'evost}
\address{LMRS, UMR 6085, Universit\'e de Rouen,
  \hfill\break\indent Avenue de l'Universit\'e, BP.12, Technop\^ole du
  Madril\-let, \hfill\break\indent
F76801 Saint-\'Etienne-du-Rouvray, France.}
\email{Mustapha.Mourragui@univ-rouen.fr}
\email{nicolas.prevost3@univ-rouen.fr}
\date{\today}


\keywords{Blume-Capel model,  Neumann condition, Dirichlet condition, Empirical density, Large deviations.}

\subjclass[2020]{ 82C22, 60F10, 82C35}

\date{\today}

\thispagestyle{empty}

\begin{abstract}
We consider the Blume-Capel spin model  on a finite cylinder with reservoirs at the boundary. A model with spin variable $\sigma$ taking values in $\{-1,0,1\}$, with the superposition of two dynamics: in the bulk, the spins evolve according to a weakly asymmetric dynamics; and the boundary dynamics follows a mechanism of creation, annihilation and spin flip, its action is accelerated differently on the left and on the right in a way to produce mixed boundary conditions. For the dynamics in the bulk, two quantities are conserved, the magnetization which corresponds to the sum of the spin values, and the concentration which corresponds to the sum of the squared spin values. We first establish, in the diffusive scaling, the hydrodynamic limit for this model which states that the couple of empirical measures (magnetization, concentration) converges to the solution of a system of coupled equations with mixed boundary conditions. Then we prove the associated dynamical large deviations principle.
\end{abstract}

\maketitle

\section{Introduction}\label{sec:intro}

We investigate a d-dimensional interacting particle system describing the evolution of three–state spins according to the boundary driven weakly asymmetric Blume–Capel model (BWABC). The boundary dynamics is modeled through a superposition of creation, annihilation and spin–flip mechanisms. The effect of reservoirs on the symmetric simple exclusion process (SSEP) has been thoroughly studied, see for instance \cite[\dots]{bdgjl3,blm09,fgln,fgn3}. The interacting particle system considered here differs substantially from the SSEP since two conserved quantities are involved, namely the magnetization and the concentration, and the corresponding hydrodynamic limit is described by a system of two coupled equations. Nevertheless, the methods developed in the previous works can be adapted to our framework in order to analyze the effect of such reservoirs.

The Blume-Capel model was originally introduced in \cite{kuh,kuh2,kuh3,b1,c1,beg} to describe the He$^3$-He$^4$ phase transition. A version of this model with long-range Kac interactions has been studied in \cite{mm1}, where the hydrodynamic limit was derived in infinite volume. More recently, several papers in mathematical physics have been devoted to numerical investigations and applications, such as the design of active layers for organic electronics; see for example \cite[\dots]{LMN,lcm,cjlmm,CLMM} and the references therein.

The system consists of spins with nearest-neighbor interactions on a lattice, where each spin variable takes values in $\{-1,0,+1\}$. In the framework of interacting particle systems, the weakly asymmetric Blume-Capel model (WABC) is formally described as follows. The spins evolve in the $d$-dimensional finite cylinder $\Lambda_N  \coloneqq \{-N,\dots,N\}\times \mathbb{T}^{d-1}_N$, 
where $N\geq 1$ is a scaling parameter and $\mathbb{T}^{d-1}_N$ denotes the discrete $(d-1)$-dimensional torus, that is, $\mathbb{T}^{d-1}_N=(\mathbb{Z}/N\mathbb{Z})^{d-1}$. The time evolution of the configuration is given by a continuous-time Markov process $(\sigma_t)_{t\in [0,T]}$ for a fixed $T>0$.

At the microscopic level, the only conserved quantities are the magnetization $\sum_x \sigma(x)$ and the concentration $\sum_x \sigma(x)^2$. The corresponding hydrodynamic equations are defined on the macroscopic domain $\Lambda := (-1,1)\times \mathbb{T}^{d-1}$, where $\mathbb{T}^{d-1}$ denotes the continuous $(d-1)$-dimensional torus, and are expressed in terms of the associated order parameters $m$ and $\phi$.

The WABC is defined through the formal Hamiltonian
\begin{equation}
\label{Ha1}
{\mathscr H}_N^{\what h}(\sigma)\, \coloneqq \, -\sum_{i=1}^2\sum_{x\in\Lambda_N}  \sigma(x)^i h_i\left(\frac{x}{N}\right)
-a_1\sum_{x\in \Lambda_N}\sigma(x)
-a_2\sum_{x\in \Lambda_N}\sigma(x)^2  ,
\end{equation}
 where $\what a \coloneqq (a_1,a_2) \in \R^2$,  $\what h \coloneqq (h_1,h_2)$ with 
\[
h_i(x) \coloneqq E^i\cdot x \coloneqq \sum_{k=1}^d E^i_k x_k, 
\qquad \text{for any } x \in \Lambda \text{ and } i=1,2,
\]
and $E^1, E^2 \in \mathbb{R}^d$ are two $d$-dimensional vectors.

Given the Hamiltonian \eqref{Ha1}, one can construct, in a standard way (see \cite{liggett}), the Markov process describing the WABC model. The transition rates are chosen so that the bulk dynamics satisfies a detailed balance condition with respect to the family of Gibbs measures associated with ${\mathscr H}_N^{\what h}$, parametrized by the chemical potentials $a_1, a_2\in \mathbb R$, for some fixed vectors $E^1, E^2 \in \mathbb R^d$.

The boundary dynamics, parametrized by a pair of functions $\what b\coloneqq(b_1,b_2)$ defined on the boundary of $\Lambda_N$, follows a mechanism of creation, annihilation and spin flip. Its action is accelerated or slowed down, with different intensities at the left and right boundaries, in order to produce mixed boundary conditions. The superposition of the bulk and boundary dynamics (BWABC) defines an \emph{irreducible} Markov jump process on the finite state space $\{-1,0,1\}^{\Lambda_N}$. By the general theory of finite-state Markov processes (see \cite{liggett}), the invariant measure $\mu_N^{SS}$ is unique but non-reversible. It characterizes the long-time behavior of the system. For the simple exclusion process, the stationary measure can be expressed as a product of matrices through the Matrix Product Ansatz (MPA), developed in \cite{derrida_large_2002}. In the case of the Blume-Capel model, however, to the best of our knowledge, such a Matrix Product Ansatz has not been constructed in the literature. The extension of the MPA to this three-state interacting particle system is therefore a natural and challenging open problem. We expect that such a construction should be possible and that it would provide an interesting direction for future work.

Our goal in the present paper is, for two given vectors $(E^1,E^2)$, to investigate the hydrodynamic limit and the large deviations of the BWABC model with mixed Dirichlet-Neumann boundary conditions. The extension to more general boundary conditions, such as Robin boundary conditions, is a natural and relevant direction to pursue. The hydrodynamic behavior can be obtained in a similar way, whereas the large deviation analysis presents additional difficulties due to the complexity of the boundary generator. This extension is currently being investigated and is the subject of ongoing work.

Hydrodynamic limits describe the macroscopic behavior of interacting particle systems. From a probabilistic point of view, they correspond to a law of large numbers for the empirical spatial densities of the conserved quantities. The limit is given by a deterministic trajectory characterized as the unique weak solution of a coupled system of partial differential equations.

Large deviations around the hydrodynamic limit quantify the asymptotic probability of observing atypical macroscopic evolutions. More precisely, they provide an exponential estimate for the probability that the empirical profiles deviate from the hydrodynamic behavior. In some sense, it states that the probability of deviating from the expected hydrodynamic limit goes exponentially fast to zero.

In recent years, many works \cite[\dots]{bdgjl3, blm09, flm, fgln,fgn3} have been devoted to boundary driven systems in bounded domains involving a single conserved quantity. In the present setting, two quantities play a role, namely the magnetization and the concentration. The nonequilibrium systems considered in those works are typically lattice gas models subject to boundary mechanisms of particle creation and annihilation, modeling exchange reservoirs.

A main difference with respect to most of the works mentioned above, lies in the fact that two conserved quantities are involved in the present model, which are further defined through signed measures. For instance, in the exclusion process, where the density is positive, the hydrodynamic equation typically involves the scalar mobility function $\chi(\rho)=\rho(1-\rho)$. In our setting, the mobility is described by a symmetric matrix-valued function, reflecting the coupling between the magnetization and the concentration. This structure, in addition to the non-convexity of the rate functional,  significantly increases the technical complexity of the analysis, in particular in the derivation of the large deviation functional. 

The paper is organized as follows. We begin by introducing our setting and describing the model in detail. In Section 3, we introduce the main concepts and present our main results. Section 4 is devoted to the hydrodynamic limit. Section 5, which is divided into five subsections, deals with dynamical large deviations. Finally, the Appendix is divided into five parts. The first part is devoted to some basic estimates that are used throughout the paper. The next three parts establish the uniqueness of the solution to the hydrodynamic equation and prove some of its properties that are needed to establish the large deviation lower bound. In the fifth and final part of the Appendix, we introduce a functional that allows us to set the large deviation rate functional equal to $+\infty$ for trajectories that are not absolutely continuous with respect to the Lebesgue measure.

\bigskip
\section{Description of the model}\label{sec:2}
Fix a positive integer $d\ge 1$.  Denote by $\Lambda$ the open set $(-1,1)
\times \bb T^{d-1}$  and by $\overline \L = [-1,1]
\times \bb T^{d-1}$ its closure, where $\bb T^{k}$ is the $k$-dimensional torus
$[0,1)^k$, and by $\Gamma \coloneqq \partial\Lambda$ the boundary of $\Lambda$: $\Gamma = \Ga^-\cup\Ga^+$ where $\Ga^\pm \coloneqq  \{(u_1,
\dots , u_d)\in \overline \L  : u_1 = \pm 1\}$. We equip  these sets with their Borel $\sigma$-algebra.

For an integer $N\ge 1$, denote by $\bb T_N^{d-1} \coloneqq \{0,\dots,
N-1\}^{d-1}$, the discrete $(d-1)$-dimensional torus of length $N$. Let us
call $\L_N \coloneqq \{-N,\ldots,N\} \times \bb T_N^{d-1}$, the
cylinder in $\bb Z^d$ of length $2N+1$ and basis $\bb T_N^{d-1}$ and let
$\Gamma_N \coloneqq \Gamma_N^-\cup \Gamma_N^+$ be the boundary of $\L_N$ where $\Gamma_N^\pm \coloneqq \{(x_1, \dots, x_{d}) \in \bb Z\times \bb T_N^{d-1}\,|\, x_1 =\pm
N\}$.  The elements of $\L_N$
are denoted by letters $x,y$ and the elements of $\overline\L$ by the
letters $u, v$.

We introduce a  spin model with reservoirs taking values in
$\{-1,0, 1\}$ on  $\Lambda_N$, that we call the
boundary driven weakly asymmetric Blume-Capel model (WABC).
The spin variable in the site
$x\in \Lambda_N$ is denoted by $\s(x)$  and the phase space is
$\X_N\coloneqq \{-1, 0, 1\}^{\Lambda_N}$. A configuration $\s\in \X_N$ is a function $\s :\Lambda_N \longrightarrow
\{-1,0,+1\}$.
For $x,y\in \Lambda_N$ and any function $f:\mathbb X_N \longrightarrow \R$,
define $\big(\nabla_{x,y}f\big)(\s)$  by
\begin{equation}\label{sauts-flips}
\big(\nabla_{x,y}f\big)(\sigma) \coloneqq f(\sigma^{x,y}) -f(\sigma)  ,  
\end{equation}
where $\s^{x,y}$ is a configuration obtained from $\s$ by
interchanging the value at $x$ and $y$ : 
\begin{equation*}
(\s^{x,y}) (z) := 
  \begin{cases}
        \s (y) & \textrm{ if \ } z=x ,\\
        \s (x) & \textrm{ if \ } z=y ,\\
        \s (z) & \textrm{ if \ } z\neq x,y .
  \end{cases}
\end{equation*}
Fix a couple of $d$-dimensional vectors $\what E \coloneqq (E^1,E^2) \in (\R^d)^2$ and the function $\what h \coloneqq (h_1,h_2): \La \longrightarrow \R^2$, defined for any $u\in\La $  and $i=1,2$ by  $h_i (u) \coloneqq \left( E^i \cdot u\right) = \sum_{k=1}^d E^i_k u_k$, where for $u,v \in \R^d$, $(u\cdot v)$ stands for the usual inner product in $\R^d$.

The total interaction energy among  spins is defined by the Hamiltonian ${\mathscr H}_N^{\hat h}(\sigma)\coloneqq {\mathscr H}_N^{\hat h,a_1,a_2}(\s)$ defined in \eqref{Ha1}.

The boundary driven (WABC) model is the Markov process on $\X_N$ whose generator
$\mf L_{N}\, :=\, \mf L_{\what E,b,N}$
can be decomposed as
\begin{equation}\label{eq:gen}
  \mf L_N \, :=\, N^2\cl_{\what E,N} \;+N^{2-\mathfrak{a}_l}\; L^-_{b,N}\;+N^{2-\mathfrak a_r}\; L^+_{b,N}  
  ,
  \end{equation}
where  $\mathfrak a_l \in (-\infty,1)$ and $\mathfrak a_r\in (1,+\infty)$.

The generator $\cl_{\what E,N},$   describes the bulk dynamics which preserves the magnetization and concentration, its action on functions $f: \X_N\to \R$ is then given by
$$
\left(\cl_{\what E, N}  f\right) (\s)= 
 \sum_{k=1}^d \sum_{x,x+\vecte_k\in\L_N} C_N^{\what E}  ({x,x+\vecte_k};\s)  \left[ f(\s^{x,x+\vecte_k}) -
f(\s) \right] ,
$$
with the  rate of exchange occupancies   $C_N^{\what E } $  given by 
\begin{equation}
\label{rate1}
C_N^{\what E}(x,y;\s)= 
\exp \left\{ - \frac{1}{2}  \big(\nabla_{x,y}{\mathscr H}^{\what h}_N (\s)\big)  \right\}.
\end{equation}

For any $\what E \in (\R^d)^2$, the operator $\cl_{\what E,N}$ is  self-adjoint w.r.t. the family of
Gibbs measures $\mu^{\what E,a_1,a_2}_{N}$ associated to the Hamiltonian \eqref {Ha1} with chemical potentials $(a_1,a_2)\in \R^2$: 
\begin{equation*}
\mu^{\what E,a_1,a_2}_{N}(\s)
= \frac 1{Z^{\what E, a_1,a_2}_{N}} \exp\{- {\mathscr H}^{\what h}_N(\s) \}, \qquad \s \in  \X_N ,
\end{equation*}
where ${Z^{\what E,a_1,a_2}_{N}}$ is the normalization constant. 
 This means that the   rates of the bulk dynamics   
$ 
\{ C_N^{\what E} (x,y;\s), \quad x,\,  y  \in \L_N,\, \ \s 
\in \X_N \}$,  
satisfies the detailed balance conditions:
$$ C_N^{\what E}(x,y;\s)=   e^{-  [{\mathscr H}^{\what h}_N (\s^{x,y}) -{\mathscr H}^{\what h}_N (\s)] } C_N^{\what E}(y,x;\s^{x,y}), \qquad x,y \in \La_N, \, \sigma \in \mathbb X_N. $$

In order to define the boundary dynamics, we need to introduce
a family of \textit{invariant probability measures}   for the symmetric Blume-Capel dynamics with  $\what E=\what 0$ generated by $\mathcal L_{\what 0 ,N}$. 
Since this exchange dynamics conserves magnetization and concentration, its invariant measures 
are parametrized by two chemical potentials and are product: 
for a vector-valued function $\widehat A = (a_1,a_2) \in [-1,1]\times[0,1]$, 
we define ${\bar \nu}_{\widehat A}^N$ as the
product measure on $\X_N$ with chemical potential $\widehat A$ given by
$$
d{\bar \nu}_{\widehat A}^N(\s)=Z_{ \widehat A}^{-1}\exp\Big\{ a_1\sum_{x\in\Lambda_N}
\s(x) + a_2\sum_{x\in\Lambda_N} \sigma(x)^2\Big\} ,
$$
where $Z_{\widehat A}$ is the normalization constant. For 
$(a_1,a_2)\in [-1,1]\times [0,1]$ let
$m=m(a_1,a_2)$ (resp. $\phi=\phi(a_1,a_2)$) be the expectation of
$\s(0)$ (resp. $\sigma(0)^2$) under ${\bar \nu}_{\widehat A}^N$:
$$
 m (a_1,a_2) \; =\;{\E}^{{\bar \nu}_{\widehat A}^N}\big(\s
    (0) \big)  ,\qquad 
    \phi(a_1,a_2) \; =\; {\E}^{{\bar \nu}_{\widehat A}^N}\big(\sigma(0)^2 \big) .  
$$
Observe that the function $\Phi$ defined on $(-1,1) \times  (0,1)$ by $\Phi(a,b)=(m,\phi)$ is a bijection from \mbox{$(-1,1) \times  (0,1)$} to ${\mb I}$, where
\begin{equation}\label{cC0}
{\mb I}=\Big\{ (m,\phi)\ :\ |m|< \phi<1 \Big\}.
\end{equation}
A simple computation shows that 
$\Phi^{-1}(m,\phi):=\Psi(m,\phi)=\displaystyle (\Psi_1(m,\phi),\Psi_2(m,\phi))$, where
$$
\Psi_1(m,\phi) =\frac{1}2\log \Big( \frac{m+\phi}{\phi-m}\Big) ,
\qquad \Psi_2(m,\phi) =\frac{1}2\log \Big( \frac{\phi^2-m^2}{4(1-\phi)^2}\Big) .
$$
For every $\widehat \rho=(m,\phi)\in {\mb I}$, we denote by
$\displaystyle \nu_{\widehat \rho}^N:= {\bar \nu}_{\Psi(\widehat \rho)}^N$ the product measure parametrized by $\Psi$,
such that
$$
m \; =\;{\E}^{\nu_{\widehat \rho}^N}\big[\s (0) \big], \qquad
  \phi \; =\; {\E}^{\nu_{\widehat \rho}^N}\big[\sigma(x)^2 \big]  .
$$


For any integer $0\leq l \leq +\infty$, 
  denote by  $\mathcal C^l (\overline \Lambda)$
(resp. $\mathcal C^l_c (\Lambda)$) the space of
$l$-continuously differentiable functions on $\overline \Lambda$ 
(resp. with compact support in $\Lambda$) with values in $\R$,
and by $\mathcal C_{0,-}^l (\overline \Lambda)$ the subset of $\mathcal C^l (\overline \Lambda)$ of functions vanishing at the left boundary $\Gamma^-$. For $l=0$,
we omit the subscript $l$ and denote simply $\mathcal C=\mathcal C^0$.

In order to define the dynamics at the boundary, we need to introduce a function $\what b$ defined as the trace on $\Ga$ of some  regular function $\widehat\theta=(\theta_1,\theta_2) : \overline\Lambda
\to \mathbf{I}$ in   $ (\mathcal C^2(\overline{\Lambda}))^2$ such that for all $u\in \overline \L$,
\begin{equation}\label{eq:cC}
\abs{\theta_1(u)}< \theta_2(u) <1.
\end{equation}

Notice that, since $\what\theta$ is continuous and takes its values in $\mathbf{I}$ , the previous condition on $\what\theta$ implies that for all $u\in \overline\Lambda$, 
\begin{equation}\label{eq:cC2}
0<   |\theta_1 (u)|+ c^*< \theta_2 (u) \leq  C^*< 1,
\end{equation}
for two positive constants $c^*,C^*$.

We denote by $\nu_{\widehat \theta}^N$, the product measure with varying profile $\widehat \theta$:
\begin{equation}\label{meas}
\nu_{\widehat \theta}^N (\sigma)
=Z_{\widehat \theta}^{-1}\exp\Big\{ \sum_{x\in\Lambda_N}\Psi_1(\theta_1(x/N),\theta_2(x/N))
\s(x) + \sum_{x\in\Lambda_N}\Psi_2(\theta_1(x/N),\theta_2(x/N)) \sigma(x)^2\Big\}  .
\end{equation}\\
Each marginal of $\nu_{\what \theta}^N$ is concentrated on the set $\{-1,0,1\}$, for all $x\in \La_N$ :

\[
    \nu_{\what \theta}^N\left(\sigma(x)=\pm1\right) = \frac{\theta_2\left(\frac{x}{N}\right) \pm \theta_1\left(\frac{x}{N}\right)}{2}, \qquad \nu_{\what \theta}^N\left(\sigma(x)=0\right) = 1 - \theta_2\left(\frac{x}{N}\right).
\]

The generator $L^-_{b,N}$ (resp. $L^+_{b,N}$) models the spins reservoir at the left boundary (resp. right boundary) of $\La_N$,  it is defined by
the infinitesimal generator of a Glauber process acting on $\Gamma_N$ as 
\begin{equation}\label{gen-bord}
\begin{aligned}
(L^\pm_{b,N} f)(\s) &\;=\;  \: \sum_{x \in \Gamma_N^\pm}\Big\{r_x^{(1,0),-}\big(\widehat b,\s\big)\big[ f(\s^{x,-})-f(\s)\big] +
r_x^{(0,1),+}\big(\widehat b,\s\big)\big[ f(\s^{x,+})-f(\s)\big]\Big\}\\
\  & \;\; +\; \: \sum_{x \in \Gamma_N\pm}\Big\{r_x^{(-1,0),+}\big(\widehat b,\s\big)\big[ f(\s^{x,+})-f(\s)\big] +
r_x^{(0,-1),-}\big(\widehat b,\s\big)\big[ f(\s^{x,-})-f(\s)\big]\Big\}\\
\  & \;\; +\; \: \sum_{x \in \Gamma_N\pm} r_x^{(-1,1)}\big(\widehat b,\s\big)\big[ f(T^x\s)-f(\s)\big] ,
\end{aligned}
\end{equation}
where for $x\in \Gamma_N$ and $\widehat\lambda=(m,\phi)\in {\mb I}$ the rates
$r_x\big(\widehat\lambda,\s)$ are given by
\begin{equation}\label{rate-b}
\begin{aligned}
  r_x^{(\pm 1,0),\mp}\big(\widehat\lambda,\s\big) &\;:=\;(1-\phi(x/N))\1_{\{\s(x)=\pm 1\}} ,\\
  r_x^{(0,\pm 1),\pm}\big(\widehat\lambda,\s\big) &\;:=\; \frac{\phi(x/N)\pm m(x/N)}{2}\1_{\{\s(x)=0\}} ,\\
  r_x^{(-1,1)}\big(\widehat \lambda,\s\big) & = \frac{\phi(x/N)- m(x/N)}{2} \1_{\{\s(x)=+ 1\}} +\frac{\phi(x/N)+ m(x/N)}{2}\1_{\{\s(x)=-1\}} ,
\end{aligned}
\end{equation}
and $T^x\s$ is the configuration  obtained from $\s$ by flipping the spin at site $x$:
\begin{equation*}
(T^x\sigma) (z) := 
  \begin{cases}
        -\s (x) & \textrm{ if \ } z=x,\\
        \s (z) & \textrm{ if \ } z\neq x ,
  \end{cases}
\end{equation*}
and $\s^{x,\pm}$ is defined as
\begin{equation*}
(\sigma^{x,\pm} ) (z) := 
  \begin{cases}
        \s (x)\pm 1\ \text{mod}\ 3 & \textrm{ if \ } z=x,\\
        \s (z) & \textrm{ if \ } z\neq x   .
  \end{cases}
\end{equation*}

Notice that in view of the
diffusive scaling limit, the generator has been speeded up by $N^2$.
We denote by $(\sigma_t)$ the Markov process on $\X_N$ with generator
$\mf L_{N}$.

Since the Markov process $(\sigma_t)_{t\in [0,T]}$ is irreducible, for each $N\ge 1$ and $\what E \coloneqq (E^1,E^2) \in (\R^d)^2$, there exists a unique invariant
measure $\mu^{SS}_N=\mu^{SS}_N(\what E ,\what b)$ in which we drop the dependence on $\what E$ and $\what b$ 
from the notation. Moreover, if $\what b$ is not constant, then
the invariant measure $\mu^{SS}_N$ is not-reversible and cannot be written in a simple form.

\begin{remark}
    It would be interesting to study the present dynamics with more general boundary conditions, in particular with non-reversible boundary dynamics. We expect that, provided suitable replacement lemmas at the boundary (in our case Lemma \ref{replacementGauche} and Lemma \ref{replacementDroite}) can be established for the corresponding boundary dynamics, the same results should remain valid for a broader class of boundary rates,  while the bulk part of the argument should remain essentially unchanged. Similar phenomena have been rigorously analyzed for the SSEP with non-reversible slow boundary dynamics  in \cite{erignoux_hydrodynamics_2020} and references therein. 
\end{remark}

\section{The results}\label{results} 
\subsection{Notation}
We fix $T>0$. 
For a metric space $\mathbb S$, denote $\mathcal{D}([0,T],\mathbb S)$ the space of càdlàg functions from $[0,T]$ to $\mathbb S$, endowed with the Skorokhod topology and its corresponding Borel $\sigma$-algebra. Given a probability measure $\mu$ on $\X_N$, 
the probability measure $\bb P^{N,\widehat b}_\mu$ on the path space 
$\mathcal{D}([0,T],  \X_N)$, 
is the law of $(\sigma_t)_{t\in [0,T]}$ with initial 
distribution  $\mu$. The  associated expectation
is  denoted by  $\bb E^{N,\widehat b}_\mu$.\\
We denote by $\mc M=\mc M (\L)$ the space of finite signed measures on $\L$, with total variation bounded by 2
endowed with the weak topology. For a finite signed measure $m(d u)$ and a continuous function $G\in \mc C(\L)$,  
we let $\< m(d u), G\>$ be the integral of $G$ with respect to
$m(d u)$.  

Given a configuration $\sigma$, we define the empirical measure 
$$\widehat \pi^N :=  \widehat \pi^N(\sigma)=\begin{pmatrix}
  \pi^{N,1}(\sigma) \\ \pi^{N,2}(\sigma)
\end{pmatrix} \in \mc M\times\mc M:= \mc M^2 , $$ 
where the empirical measures $\pi^{N,i} \in \mc M$ for $i=1,2$ are defined as
\begin{equation*} 
\pi^{N,1} \, =\, N^{-d}\sum_{x\in\L_N}\sigma(x)\, \delta_{x/N} ,\quad 
\pi^{N,2} \, =\, N^{-d}\sum_{x\in\L_N}\sigma(x)^2\, \delta_{x/N} ,
\end{equation*}
and $\delta_{u}$ is the Dirac measure concentrated on $u$. 

We  also  denote by $\widehat \pi^N$   the map from $\mathcal{D}([0,T],\X_N)$  to 
 $\mathcal{D}([0,T],\mathcal{M}^2)$  defined by 
$\widehat \pi^N (\sigma_{\cdot})_t=  \widehat \pi^N (\sigma_t)$ and denote by  
$\mathbb{Q}_{\mu}^{N,\widehat b}  =  \mathbb P^{N,\widehat b}_{\mu} \circ (\widehat \pi^N)^{-1} $ 
 the law of the process
$\big(\widehat \pi^N (\sigma_t )\big)_{t\in [0,T]}$. 

Let $\widehat{\mathcal{M}}_0$ be the subset of ${\mathcal{M}}^2$ of all couple of signed measures $\widehat \pi=(\pi^1,\pi^2)$ such that
$\pi^1(du)=m(u)du$ and $\pi^2(du)=\phi(u)du$ are absolutely continuous  with respect to the Lebesgue measure, with couple of densities $(m,\phi)$ in $\overline{\mathbf{I}}$:
\begin{equation*}
\widehat{\mathcal{M}}_0=\big\{\widehat \pi\in \mathcal {M}^2:\widehat \pi(du)\coloneqq(m(u)du,\phi(u)du) \;\; \hbox{ and } \;\;
(m,\phi)\in \bar {\mb I}\;  \hbox{ a.e.} \big\} ,
\end{equation*}
where 
\begin{equation}\label{cC0b}
 \bar {\mb I}\, =\, \Big\{ (m,\phi)\ :\ |m|\le \phi\le 1 \Big\}.
\end{equation}

For positive  integers $1\leq l,k \leq +\infty$, we denote by
$\mathcal C^{k,l} ([0,T]\times \overline \Lambda)$
(resp. $\mathcal C^{k,l}_{0,-} ([0,T]\times \overline \Lambda)$)
the space  of functions  from $[0,T]\times \overline\Lambda$ to $\R$
that are $k$-continuously differentiable in time and
$l$-continuously differentiable in space (resp.
and vanishing at the left boundary $\Gamma^-$ of $\L$). Similarly, we define
$\mathcal C^{k,l}_c ([0,T]\times \Lambda)$ as the subspace of
$\mathcal C^{k,l} ([0,T]\times \overline \Lambda)$
of functions with compact support in $[0,T]\times \L$.

Let $L^2(\L)$ be the Hilbert space of functions $G:\L
\to \R$ such that $\displaystyle \int_\L | G(u) |^2 du <\infty$ is equipped with
the inner product
\begin{equation*}
\<F,G\> =\int_\L F(u) \,  G (u) \, du .
\end{equation*}
The norm of $L^2(\L)$ is denoted by $\| \cdot \|_{L_2(\L)}$. 

For any  function 
$\widehat G = (G_1,G_2)\in (\mathcal C(\Lambda))^2$,
the integral of $\widehat G$ 
with respect to $\widehat \pi^N$, denoted by $\langle\widehat \pi^N, \widehat G\rangle $, is given by
\begin{equation*}
\langle\widehat \pi^N, \widehat G\rangle  =\, \sum_{i=1}^2 \langle\pi^{N,i}\, ,\, G_i \rangle .
\end{equation*}

For  $\widehat G = (G_1,G_2), \widehat H= (H_1,H_2) \in \big(L^2(\Lambda)\big)^2$,
$\langle \widehat G(\cdot),\widehat H(\cdot) \rangle$, also denotes the inner product:
\begin{equation}
  \langle \widehat G(\cdot), \widehat H(\cdot) \rangle 
  = \sum\limits_{i=1}^2 \langle G_i (\cdot),H_i(\cdot) \rangle = \sum_{i=1}^2 \int_
  \Lambda G_i(u)H_i(u) du  .\label{produitScalaireChapo}    
\end{equation}

For a smooth function $G: [0,T]\times \L \to \R$, $\partial_t G(t,u)$ 
represents the partial derivative with respect to the time variable $t$ and
for $1\le j\le d,k\ge 1$, $\partial_{e_j}^k G(t,u)$ stands for the $k$-th partial derivative 
in the direction $e_j$ with respect to the space variable $u$. 
The discrete gradient $\partial_{N,e_j}$ in the direction $e_j$ is defined for
$x,x+e_j \in \Lambda_N$ and $G:\Lambda \to \mathbb R$, by
$$
\partial_{N,e_j} G(x/N) = N\Big(G\Big(\frac{x+e_j}{N}\Big) - G\Big(\frac{x}{N}\Big)\Big) .
$$ 
The discrete Laplacian $\Delta_N$ and the Laplacian $\Delta$ are respectively defined for 
$G \in \mathcal C^2 (\Lambda )$, if $x,x\pm e_j\in\Lambda_N $ for $1\le j\le d$ and $u\in \Lambda$, by 
\begin{align*}
& \Delta_N G(x/N) = N^2 \sum\limits_{j=1}^d \Big[ G\Big(\frac{x+e_j}{N}\Big)
 + G\Big(\frac{x-e_j}{N}\Big) - 2 G\Big(\frac{x}{N}\Big) \Big] \ \ \mbox{ and } \ \ 
 \Delta G (u)= \sum\limits_{j=1}^d \partial_{e_j}^2 G(u) .
\end{align*}
For a vector valued function $\widehat G=(G_1,G_2)\in \big(\mathcal C^{1,2} ([0,T]\times \overline \Lambda)\big)^2$ and $t\in [0,T]$, we shall denote
$$
\Delta \widehat G_t =(\Delta G_{1,t}, \Delta G_{2,t}) ,\quad \partial_{e_k} \widehat G_t =(\partial_{e_k} G_{1,t}, \partial_{e_k} G_{2,t})
 ,\;\; k=1,\ldots, d\, \; \text{and}\;\; \partial_t \wg_t =(\partial_t G_{1,t}, \partial_t G_{2,t}) ,
$$
where $\partial_t$ stands for the time derivative.

Before stating our results, we need more notation. Let $\scall{}{\cdot}{\cdot}$ be the inner product of 
$L^2([0,T]\times\Lambda)$: for any $F,G\in L^2([0,T]\times\Lambda)$,
\begin{equation*}
 \scall{}{F}{G}  =\int_0^T \big<F_s,G_s\big>\, ds.
\end{equation*}
By abuse of notation, for $\what F=(F_1,F_2),\what G=(G_1,G_2)\in (L^2([0,T]\times \Lambda))^2$, we also denote
\begin{equation}\label{doubleProduitScall}
 \scall{}{\what F}{\what G}  =\sum_{i=1}^{2}\scall{}{ F_i}{ G_i}.
\end{equation}

Let $H^1(\L)$ be the Sobolev space of functions $F$ with
generalized derivatives $\nabla F=\big(\partial_{e_1} F,\cdots,\partial_{e_d} F \big)$
in $\big(L^2(\L)\big)^d$ endowed with the inner product
$\<\cdot, \cdot\>_{H^1}$, defined by
\begin{equation*}
\<F,G\>_{H^1} \coloneqq \< F, G \> +
\sum_{k=1}^d \<\partial_{e_k} F \, , \, \partial_{e_k} G \>.
\end{equation*}
The corresponding norm of the Hilbert space $H^1(\L)$  is denoted by
$\|\cdot\|_{H^1}$.
Denote by $\Tr$ the trace operator on such Sobolev spaces, that allows us
to define the value of an element $F$ in ${H}^1(\L)$ at the boundary.
The operator $\Tr$  is defined as a linear and continuous map, $\Tr : {H}^1(\L)\to L^2(\Gamma)$ such that $\Tr$  extends the classical trace, that is
$\Tr(F)=F_{|_\Gamma}$, for any $F\in {H}^1(\L)\cap {\mathcal C}(\bar \L)$.

Finally, for a Banach space $(\bb
B,\Vert\cdot\Vert_{\bb B})$ we denote by $L^2([0,T],\bb B)$
the Banach space of measurable functions $U:[0,T]\to\bb B$ for which
\begin{equation*}
\Vert U\Vert^2_{L^2([0,T],\bb B)} \;=\;
\int_0^T\Vert U_t\Vert_{\bb B}^2\, dt \;<\; \infty
\end{equation*}
holds.

\smallskip
\noindent 
\textbf{Warning:} \begin{itemize}
    \item[{\bf-}] By abuse of notation, for $\what a \coloneqq (a_1,a_2),\what c \coloneqq (c_1,c_2)\in \R^2$, $\what a \;\what c = \sum_{i=1}^2 a_i c_i$.
    \item[{\bf-}] By abuse of notation, if $\what \pi \coloneqq  (\pi^1,\pi^2) \in \mathcal D ([0,T],  \wm)$ is such that for almost all $t\in [0,T]$, $\pi_t^1$ and $\pi_t^2$ are absolutely continuous with respect to the Lebesgue 
measure with densities $\wpi_t (du)=(\rho_t^1(u)du,\rho^2_t(u)du)$, we shall write $\wpi_t =(\rho_t^1 ,\rho^2_t)\eqqcolon \what \rho_t$.
    \item[{\bf-}] For any function $F \in H^1(\La)$ and $r \in \Ga$, we denote $F(r)\coloneqq \Tr(F)(r)$. 
\end{itemize}

\subsection{Hydrodynamic limit}\label{SectionHydro}
We first describe \textit{the hydrodynamic equations}. Let $\widehat \gamma = (\gamma_1,\gamma_2) : \overline \Lambda \to {\mb I}$
be an initial profile, and denote by 
$\widehat \rho = (m,\phi) : [0,T] \times \overline \Lambda \to \bar {\mb I}$ 
a typical macroscopic trajectory.  We shall prove in Theorem 
 \ref{th-hy} below that the macroscopic evolution of  the  local
particle  density $\widehat \pi^N$  is described by  the following system of 
non-linear coupled equations
\begin{equation}\label{eq:1}
\left\{ \begin{array}{rl}
  \displaystyle \partial_t \what\rho &= \displaystyle \grad \cdot \left(\grad \what\rho -\frac{1}{2}\Sigma(\what\rho)\what E  \right ),\\  
  \displaystyle \what\rho|_{\Ga^-} &= \displaystyle \what b ,  \\
  \displaystyle \partial_{e_1}\what\rho|_{\Ga^+} &= \displaystyle \frac{1}{2} \Sigma(\what\rho) \what E_1  , \\[0.2cm]
  \what\rho_0 &= \what \gamma,
\end{array} \right.
\end{equation}
where $\Sigma(m,\phi):=\Sigma\, =\, \big( \Sigma_{i,j}\big)_{1\le i,j\le 2}$ is the mobility matrix:
\begin{equation}\label{mobility-matrix}
\Sigma(m,\phi)\, =2\,
\begin{pmatrix}
   \phi -m^2 & m(1-\phi) \\
   m(1-\phi) & \phi(1-\phi)
\end{pmatrix}
 .
\end{equation}
A simple computation shows that for $(m,\phi)\in \mathbf{I}$, the matrix $\Sigma(m,\phi)$ is positive-definite, with the inverse
\begin{equation*}\label{mobility-matrix-1}
\Sigma^{-1}(m,\phi)\, =\frac{1}{2(\phi^2-m^2)}\,
\begin{pmatrix}
   \phi & -m \\
  - m & \phi +\frac{\phi^2 -m^2}{1-\phi}
\end{pmatrix}
 .
\end{equation*}
In the sequel, we shall denote by $S^{+}_2$ (resp. $S^{++}_2$) the set of all positive semi-definite matrices (resp. positive-definite  matrices) of size $2\times 2$. We equip these sets with the following partial order : \begin{equation}\label{orderMatrices}
    A,B \in S^+_2, \qquad A\leq B \, \iff \, \scal{}{x}{Ax} \leq \scal{}{x}{Bx}, \; \forall x \in \R^2.
\end{equation}

For   $\widehat G \in (\mathcal C^{1,2}_{0,-}([0,T]\times \overline \L))^2$, $\what  \rho =(\rho_1,\rho_2) \in  \mc D([0,T], \what{\mc M}^0)$ and $\what \gamma : \overline\La \rightarrow \mathbf I$ denote 
\begin{equation}\label{weakAlternative}
  \begin{split}
   &\ell_{\wg}^{\what E}(\wrho \, | \, \what \gamma) := \big\langle \wrho_T, \wg_T \big\rangle 
  - \langle {\what \gamma}, \what G_0 \rangle
  - \int_0^{T} \!dt\, \big\langle \wrho_t, \partial_t \wg_t \big\rangle\\
  &\qquad\qquad\quad
  \, -\int_0^{T} \!dt\, \big\langle  \wrho_t , \Delta \wg_t \big\rangle
  \;-\; \sum_{i=1}^2  \int_0^T 
  \int_{\Gamma^-}  b_i(r)(\partial_{e_1}\what G_{i,t})(r)\, dS(r) \, dt  \\
  &\qquad\qquad\quad
  +\; \sum_{i=1}^2 \int_0^T 
  \int_{\Gamma^+}  \what \rho_{i,t}(r)(\partial_{e_1}\what G_{i,t})(r)\, dS(r) \, dt \\ 
  &\qquad\qquad\quad 
  - \frac{1}{2} \sum_{k=1}^d \int_0^T
  \scal{}{\partial_{e_k} \what G_t}{\Sigma(\what\rho_t) \what E_k } dt ,
  \end{split}
  \end{equation}
    where for $r$ in $\Ga$, $\what \rho(r) \coloneqq \Tr(\what \rho)(r) \coloneqq (\Tr(\rho_1)(r), \Tr(\rho_2)(r))$ and $dS(r)$ is an element of surface on $\Ga$.

 \begin{definition}Denote by  $\mc A _{[0, T]}   $ the set of all weak  solutions of the  
 hydrodynamic limit:
\begin{equation*}
\mc A _{[0, T]} = \Big\{  \wrho \in  \big(L^2\big([0,T], H^1(\La)\big)\big)^2 \; :\; 
\quad \forall  \wg \in \big(  \mathcal C^{1,2}_{0,-} ( [0,T]
 \times \L)\big)^2 \, ,\; {\ell}_{\wg}^{\what E}(\wrho)=0
   \Big\} ,
\end{equation*}
where we denote ${\ell}_{\wg}^{\what E}(\wrho) \coloneqq \ell_{\wg}^{\what E}(\wrho\, | \, \what \rho_0)$. \\
A weak solution to the  boundary value problem \eqref{eq:1} is a function 
$\widehat \rho :[0,T]\times \Lambda\to [-1,1]\times [0,1] $ 
satisfying (IB1) and (IB2) below:
\begin{itemize} 
\item[(IB1)] $\displaystyle \widehat \rho\in \mc A _{[0, T]}$,
\item[(IB2)]  $\what \rho_0 = \widehat \gamma $ a.e..
\end{itemize}    
\end{definition}
 
\begin{proposition}
    \label{th-hy}
    For any sequence of initial probability measures $(\mu_N)_{N\ge 1}$, the sequence of probability measures 
    $(\mathbb{Q}_{\mu_N}^{N,\wb})_{N\geq 1}$ 
    is weakly relatively compact and  all its converging
    subsequences converge to   some limit  $\mathbb{Q}^{*,\what b}$ that is concentrated on the set of paths $\what\pi = (\rho_1\, du,\rho_2\, du)\in \mathcal{D}([0,T],\what{\mathcal{M}}_0)$ such that $\what \rho\coloneqq (\rho_1,\rho_2)$ is a weak solution to the hydrodynamic equation in the sense that $\what \rho \in \mc A_{[0,T]}$. Furthermore, $\what \rho$ satisfies the following inequality  
    $$ \sum_{k=1}^d \scall{}{\partial_{e_k} \what \rho}{\Sigma(\what\rho)^{-1}\partial_{e_k} \what \rho} < + \infty.$$  Moreover, if the sequence of initial measures $(\mu_N)_{N\geq 1}$
is associated to some initial profile $\widehat\gamma=(\gamma_1,\gamma_2) : \overline \Lambda \to {\mb I}$, in the following sense:
    for any $\delta>0$ and for  any  function  $\wg\in \big(\mc C(\L)\big)^2$ 
    \begin{equation}\label{pfl}
    \lim_{N\to \infty}\mu_N\Big\{ \Big| \langle  \widehat\pi_N, \wg\rangle\, -\, \langle\widehat\gamma ,\wg\rangle  \Big| \ge \delta \Big\}=0 ,
    \end{equation}
    then the sequence of probability measures $(\mathbb{Q}_{\mu_N}^{N,\what b})_{N\geq 1}$ 
    converges to the Dirac measure  concentrated on the unique   weak  solution $\wrho(\cdot,\cdot)$ of the boundary value 
    problem \eqref{eq:1}. Accordingly, for any $t\in [0,T]$, 
    any $\delta>0$ and any function $\wg\in  \left(\mc C(\Lambda)\right)^2$
    $$
    \lim_{N\to \infty}\PP_{\mu_N}^{N,\wb}\Big\{ \Big| \langle  \widehat\pi_N(\sigma_t), \wg\rangle \, -\, 
    \langle\wrho(t,\cdot),\wg \rangle  
     \Big| \ge \delta \Big\}=0 .
    $$
    \end{proposition}

    \subsection{Dynamical large deviations}
Fix an initial profile $\wgamma=(\gamma^1,\gamma^2) : \overline\La \rightarrow \mathbf I$.
We are interested on large deviations of the empirical measure $(\wpi^N(\sigma_t))_{t\in [0,T]}$ during the interval time $[0,T]$ 
and starting from the profile $\wgamma$.
Recall from \eqref{mobility-matrix} the definition of the mobility matrix. Define the energy functional 
${\mc E} ={\mc E}^{T,\what E} :\mathcal{D}([0,T],{{\mc M}}^2)\to [0,\infty]$ by
\begin{equation}
\label{1:Q}
{\mc E}(\pi^1,\pi^2)=
\begin{cases}
 {\mc Q} (\pi^1,\pi^2) &  \hbox{  if  } \ \ (\pi^1,\pi^2) \in \mathcal{D}([0,T], \wm),\\ 
+ \infty & \hbox{ otherwise,}
\end{cases}
\end{equation}
where the functional ${\mc Q}:\mathcal{D}([0,T],\wm)\to[0,\infty]$ is given for a trajectory $\wpi=(\pi^1,\pi^2) \in \mathcal{D}([0,T], \wm)$ with
$\wpi_t(du) =(m_t(u)du,\phi_t(u)du)$ for $t\in[0,T]$  by the formula 
\begin{align}
{\mc Q}(\what \pi) \coloneqq & \sum_{k=1}^d \mathcal{Q}_k(\what \pi) \nonumber \\
\coloneqq & \sum_{k=1}^d \sup
\Big\{  \int_0^T  \scal{}{\what\pi_t}{\partial_{e_k}  \what G_t} \, dt  
- \frac12\int_0^T \scal{}{\what G_t }{\Sigma(\what\pi_t) \what G_t} \, dt  \Big\},  \label{defQenergy}
\end{align}
in which the supremum is carried over all   $\what G \in
  ( \mathcal C^{\infty}_c([0,T]\times \L))^2$.  
We shall prove in Lemma \ref{l-energy1} that ${\mc Q}(\wpi)$ is finite if and only if 
$\wrho:=(m,\phi)\in \big(L^2\big([0,T], H^1(\La)\big)\big)^2$ and 
\[
  \sum_{k=1}^d \scall{}{\partial_{e_k}\what \rho}{\Sigma(\what \rho)^{-1} \partial_{e_k} \what \rho } < + \infty .
\]
In that case 
\begin{equation}\label {tm4}
\mc Q(\what\pi) = \frac{1}{2} \sum_{k=1}^d \scall{}{\partial_{e_k}\what \rho}{\Sigma(\what \rho)^{-1} \partial_{e_k} \what \rho }. 
\end{equation}
We shall prove in Lemma \ref{l-energy1}  that $ {\mc Q}$ is convex and lower semicontinuous.

Denote by $I_2=\begin{pmatrix}
    1 & 0 \\ 0 & 1
\end{pmatrix}$ the $2\times 2$ identity matrix and define the new energy functional $\Q_{I_2}:\DM \to[0,\infty]$ given by 
\[
    \Q_{I_2}(\what \pi) \coloneqq  \sum_{k=1}^d \sup
\Big\{  \int_0^T  \scal{}{\what\pi_t}{\partial_{e_k}  \what G_t} \, dt  
- \frac12\int_0^T \scal{}{\what G_t }{\what G_t} \, dt  \Big\} .  \label{defQI2energy}
\]
Notice that, for all $\what \pi \in \mathcal D ([0,T],\wm)$, $\Q_{I_2}(\what \pi)$ is finite if and only if the density $\what \rho$ of $\what \pi$ belongs to $\left( L^2([0,T],H^1(\La))\right)^2$. In this case 
\begin{equation}\label {energyI2}
\mc Q_{I_2}(\what\pi) = \frac{1}{2} \sum_{k=1}^d \scall{}{\partial_{e_k}\what \rho}{ \partial_{e_k} \what \rho }. 
\end{equation}

For each $\what G \in ( \mathcal C^{1,2}_{0,-}([0,T]\times \overline\La))^2$, let $J^{\what E}_{\what G}\coloneqq J_{\what G,\what \gamma,T}^{\what E}: \DM \longrightarrow \R\cup\{+\infty\}$ be the functional given by 
\[
  J_{\what G}^{\what E}(\what \pi) \coloneqq \left\{ \begin{array}{ll}
    \displaystyle\ell_{\what G}^{\what E}(\what \pi \, | \, \what \gamma) - \frac{1}{2} \sum_{k=1}^d \int_0^T \scal{}{\partial_{e_k} \what G_t}{\Sigma(\what \pi_t)  \partial_{e_k} \what  G_t } \, dt , \qquad & \text{if $\what\pi \in \mathcal{D}([0,T], \wm)$,} \\
    + \infty , \qquad  & \text{else,}
  \end{array} \right.
\]
where for $\what \pi \in \mathcal D ([0,T],\wm)$, $\ell_{\what G}^{\what E}(\what \pi \, | \, \what \gamma) \coloneqq \ell_{\what G}^{\what E}(\what \rho \, | \, \what \gamma) $. 
We then define the functional $\tilde I_T^{\what E}(\cdot|\what \gamma):\DM \longrightarrow [0,+\infty]$ by 
\[
  \tilde I_T^{\what E}(\what \pi|\what \gamma) \coloneqq \sup_{\what G\in ( \mathcal C^{1,2}_{0,-}([0,T]\times \La))^2} J_{\what G}^{\what E}(\what \pi).  
\] 
We can now define the rate function $I_T^{\what E}(\what \pi|\what \gamma):\DM\longrightarrow[0,+\infty] $ 
\[
  I_T^{\what E}(\what \pi|\what \gamma) \coloneqq \left\{ \begin{array}{ll}
    \tilde I_T^{\what E}(\what \pi|\what \gamma), \qquad &\text{if $\mathcal Q_{I_2}(\what \pi)<+\infty$,}\\
    + \infty, \qquad &\text{else.}
  \end{array} \right.  
\]

\begin{remark}
    Due to the non linearity of the last term of \eqref{weakAlternative}, the functional $\I{\cdot}$ is not convex, however we shall prove that the functional $\I{\cdot}$ is lower semicontinuous and has compact  level sets.
\end{remark}

\begin{theorem}
  Fix $T>0$ and an initial  profile $\what \ga : \La \longrightarrow \mathbf I$. Consider a sequence $(\sigma^N)_{N \geq 1}$ of configurations such that $(\delta_{\sigma^N})_{N \geq 1}$ is associated to $\what \ga$, in the sense \eqref{pfl}.  
  Then, the sequence of probability measures $\left(\mathbb Q_{\delta_{\sigma^N}}^{N,\what b}\right)_{N \geq 1}$ on $\DM$ satisfies a large deviation  principle with speed $N^d$ and rate function $I_T^{\what E}(\cdot |\what \ga)$,
  \begin{align*}
    \varlimsup_{N\to + \infty } \frac{1}{N^d} \log \mathbb Q_{\delta_{\sigma^N}}^{N,\what b}\left(\what \pi^N \in \mathcal C\right) &\leq - \inf_{\what \pi\in \mathcal C} I_T^{\what E}(\what \pi|\what \ga), \\
    \varliminf_{N\to + \infty } \frac{1}{N^d} \log\mathbb Q_{\delta_{\sigma^N}}^{N,\what b}\left(\what \pi^N \in \mathcal O\right) &\geq - \inf_{\what \pi\in \mathcal O} I_T^{\what E}(\what \pi|\what \ga), \\
  \end{align*}
  for any closed set $\mathcal C\subset \DM$ and open set $\mathcal O \subset \DM$. Moreover, the functional $\I{\cdot}$ is lower semicontinuous and has compact level sets.
\end{theorem}

    \section{Steps to prove hydrodynamic limit}

Following \cite{gpv} we divide the proof of Proposition \ref{th-hy}
in three steps: tightness of the sequence of measures
$(\Q_{\mu_N}^{N,\what b})_{N\ge 1}$, an
energy estimate to provide the needed regularity for functions in the support of any limit point of the sequence $(\Q_{\mu_N}^{N,\what b})_{N\ge 1}$, and identification of the support of limit point as weak
solution of the hydrodynamic equation \eqref{eq:1}. We then refer to \cite{kl} (Chapter 4  and Chapter 5), and \cite{fgn1}, which present arguments, by now standard, to deduce the hydrodynamic behavior of
the empirical measures from the preceding results and the uniqueness of the weak
solution to equation \eqref{eq:1}.

  We consider for $\what G \in \left( \mathcal C^{1,2}_{0,-}([0,T]\times \La)\right)^2$, the martingale $(M^{\what G,N}_t)_{t\in [0,T]}$
   with respect to the natural filtration associated with $(\eta_t)_{t\in [0,T]}$, defined for any $t\in [0,T]$ by
       \begin{align*}
      M^{\what G,N}_t \coloneqq \scal{}{\what \pi^N(\sigma_t)}{\what G_t} - \scal{}{\what \pi^N(\sigma_0)}{\what G_0} - \int_0^t \left(\partial_s + \mathfrak L_N \scal{}{\what \pi^N(\sigma_s)}{\what G_s}\right) \, ds .
    \end{align*}
     By Doob decomposition Theorem, $\big(M^{\what G,N}_t\big)^2 = N^{\what G,N}_t + \big< M^{\what G,N}_t\big>$, where 
   $N^{\what G,N}_t $ is a martingale and the associated quadratic variation $(\big<M^{\what G,N}_t\big>)_{t\in [0,T]}$ is given by 
    \begin{equation*}
   \big<M^{\what G,N}_t\big>   =  
      \int_0^t \left\{ \mathfrak L_N 
\Big(\scal{}{\what \pi^N(\sigma_s)}{\what G_s}\Big)^2 - 2 \scal{}{\what \pi^N(\sigma_s)}{\what G_s}\mathfrak L_N 
\scal{}{\what \pi^N(\sigma_s)}{\what G_s} \right\} d s .
    \end{equation*} 
     Computing the integral term of the martingale, we find
  \begin{equation}\label{martingale-0}
    \begin{aligned}
        M^{\what G,N}_t =& \scal{}{\what \pi^N(\sigma_t)}{\what G_t} - \scal{}{\what \pi^N(\sigma_0)}{\what G_0}  - \int_0^t \scal{}{\what \pi^N(\sigma_s)}{\partial_s \what G_s} \, ds \\
      & -   \int_0^t \scal{}{\what \pi^N(\sigma_s)}{\Delta \what G_s} \, ds - \int_0^T \frac{1}{N^{d-1}} \sum_{x\in \Ga_N^- } \sum_{i=1}^2 \sigma_s(x)^i \partial_{e_1} G_{i,s}\left(\frac{x}{N}\right) \, ds \\
      & + \int_0^T \frac{1}{N^{d-1}} \sum_{x\in \Ga_N^+ } \sum_{i=1}^2 \sigma_s(x)^i \partial_{e_1} G_{i,s}\left(\frac{x}{N}\right) \, ds \\
      & - \frac{1}{2} \sum_{k=1}^d \int_0^t \scal{}{\partial_{e_k} \what G_s}{\Sigma^{{\mathfrak{h}}}(\sigma_s) \what E_k} \, ds  + O_{N}(1),
    \end{aligned} 
    \end{equation}
    where $O_N(1)$ converges to $0$ when $N\to + \infty$.
    On the other hand,
    a simple computation shows that  the quadratic variation of $\big<M^{\what G,N}_t\big> $ is given by
      \begin{align*}
        \big<M^{\what G,N}_t\big> = &\int_0^t \frac{1}{N^{2d}}\sum_{x,x+e_k\in \La_N} \sum_{k=1}^d C_N^{\what E}(x,x+e_k;\sigma_s) \left( \sum_{i=1}^2 \partial_{N,e_k} G_{i,s} \left( \frac{x}{N} \right) \left[ \sigma_s(x)^i - \sigma_s(x+e_k)^i \right]\right)^2 \, ds  \\
        & + \int_0^t \frac{1}{N^{2(d-1)+\mathfrak a_r}} \sum_{x\in\Ga_N^+} \left( r_x^{(1,0),-}  (\what b, \sigma_s) + r_x^{(0,1),+}(\what b, \sigma_s) \right) \left[  G_{1,s} \left( \frac{x}{N} \right)  + G_{2,s} \left( \frac{x}{N} \right) \right]^2 \, dt  \\
        & + \int_0^t \frac{1}{N^{2(d-1)+\mathfrak a_r}} \sum_{x\in\Ga_N^+} \left( r_x^{(-1,0),+}  (\what b, \sigma_s) + r_x^{(0,-1),-}(\what b, \sigma_s) \right) \left[  G_{1,s} \left( \frac{x}{N} \right)  - G_{2,s} \left( \frac{x}{N} \right) \right]^2 \, dt  \\
        & + \int_0^t \frac{1}{N^{2(d-1)+\mathfrak a_r}} \sum_{x\in\Ga_N^+}  r_x^{(-1,1)}  (\what b, \sigma_s) \left[ 2 G_{1,s} \left( \frac{x}{N} \right)  \right]^2 \, dt  , 
    \end{align*}
    and thus its expectation vanishes as $N\to\infty$.  Therefore, by Doob’s maximal inequality, for any $\delta >0$,
\begin{equation}\label{doob-ineq}
    \lim_{N\to \infty}\PP_{\mu_N}^{N,\wb}\left(\sup_{0\le t\le T} \left|M^{\what G,N}_t \right|\ge \delta\right)=0.
\end{equation}   

Using the fact that $\what \pi^N$ has a total mass bounded by 2, \eqref{martingale-0} and \eqref{doob-ineq},
we get the following result. 
\begin{proposition} {\bf (Tightness)}
    The sequence of probability measures $(\mathbb Q^{N,\what b}_{\mu_N})_{N \geq 1}$ is tight.
\end{proposition}
Next step consists on showing that all limit points of the sequence  $(\mathbb Q^{N,\what b}_{\mu_N})_{N \geq 1}$ are concentrated  in  $\mathcal{D}([0,T],\what{\mathcal{M}}_0)$ and $\left(L^2([0,T],H^1(\L))\right)^2$, this can be shown using results of Appendix \ref{energyEstimates} and adapting similar arguments used in Chapter 5, Section 7 of \cite{kl} : 
  \begin{proposition}{\bf (Regularity of macroscopic trajectories)}
      Let $\mathbb Q^{*,\what b}$ be a limit point of the sequence of probability measures $(\mathbb Q^{N,\what b}_{\mu_N})_{N\geq 1}$, then 
      \[
        {\mathbb Q^{*,\what b}}\left(\big\{ \what \pi \in \mathcal{D}([0,T],\what{\mathcal{M}}_0) \,:\,  \mathcal{Q}(\what \pi) < + \infty\big\} \right) = 1,
      \]
      in particular from Lemma \ref{l-energy1},
      \[
{\mathbb Q^{*,\what b}}\left( \mathcal{D}([0,T],\what{\mathcal{M}}_0) \cap \left(L^2([0,T], H^1(\La))\right)^2  \right) = 1.
      \]
    \end{proposition}

Now, as a consequence, of the expressions \eqref{martingale-0}, \eqref{doob-ineq}, using the replacement lemmata stated in propositions  \ref{replacementBulk}, \ref{replacementGauche}, \ref{replacementDroite}, and following steps by now standard (see for example Chapter 2 Section 4 of \cite{kl}, Section 6 of \cite{fgn1} or Proposition 6 of \cite{msv}),  we prove that all macroscopic trajectories are weak solutions of hydrodynamic equation \eqref{eq:1}.
 \begin{proposition}{\bf (Identification of the limit equation)}
      If $\mathbb Q^{*,\what b}$ is a limit point of the sequence of probability measures $(\mathbb Q^{N,\what b}_{\mu_N})_{N\geq 1}$, then 
      \[
        \mathbb Q^{*,\what b}\left[ \what \pi  \in \mathcal{D}([0,T],\what{\mathcal{M}}_0) \quad \middle|\quad \forall \what G \in \left( \mathcal C^{1,2}_{0,-}([0,T]\times \La)\right)^2, \quad \ell^{\what E}_{\what G}(\what \rho) = 0   \right] = 1.
      \]  
    \end{proposition}

    Finally, the proof of Theorem \ref{th-hy} is achieved by showing the uniqueness of a solution to the hydrodynamic equation \eqref{eq:1} (cf. Proposition \ref{Uniqueness}).

    \section{Large deviations principle}
   
    The goal of this section is to present the main steps for obtaining large-deviations results. We start by proving some properties of the rate functional.
  
    \subsection{Rate function}
In this subsection, we analyze  some topological properties of the set of finite rate functional. 
Let us remember that the definitions and properties of classical and weighted Sobolev spaces are given in Section \ref{sobolev}. 
The following result provides an explicit representation for the rate function when it is finite.
\begin{proposition}\label{lemma:problemePerturbe}
  Take $\what \pi\in\mathcal{D}([0,T],\wm)$ such that  $\I{\what \pi}$ is finite, then $\what \pi$ is uniquely determined by a function $\what H^{\what \pi}\in \H$ such that $\what \pi$ is a weak solution of the following boundary value problem: 
  \begin{equation} \label{Eq:perturbeR+Dlemma}
    \left\{
    \begin{array}{rl}
      \displaystyle \partial_t \what\rho &= \displaystyle \grad \cdot \left[\grad \what\rho -\frac{1}{2} \Sigma(\what\rho) \left(\what E + 2 \grad \what H^{\what\pi} \right) \right],\\  
      \displaystyle \what\rho|_{\Ga^-} &= \displaystyle \what b ,  \\
      \displaystyle \partial_{e_1}\what\rho|_{\Ga^+} &= \displaystyle \frac{1}{2} \Sigma(\what\rho) \left(\what E_1 + 2 \partial_{e_1} \what H^{\what \pi} \right)  , \\[0.2cm]
      \what\rho_0 &= \what \gamma.
    \end{array}
  \right.
  \end{equation}
 In this case, 
\[
    \I{\what \pi} = \frac12 \norm{1,\Sigma(\what \pi)}{\what  H^{\what \pi}}^2,
  \]
 and 
  $\what \pi\in \mathcal{B}_{\what \gamma}^{\what b}$, where 
  \[
  \mathcal{B}_{\what\gamma}^{\what b} \coloneqq \left\{ \what \pi \in \mathcal{D}([0,T],\wm)  ~\middle|~ \what \pi_0=\what \gamma; \quad \what \pi_t\vert_{\Ga^-} = \what b\vert_{\Ga^-} , \quad \forall t\in(0,T] \right\} .
\]       

In this statement, a macroscopic trajectory $({\what \pi}_t)_{t\in[0,T]}$ is said to be a weak solution  of equation  \eqref{Eq:perturbeR+Dlemma}, meaning that  ${\what \pi}_t = {\what \rho}_t\, du$, for almost all $t\in[0,T]$, with $\what \rho \in (L^2([0,T], H^1(\La)))^2$ and for all $\what G\in ( \mathcal C^{1,2}_{0,-}([0,T]\times \overline\La))^2$,  $\ell_{\what G}^{\what E, \what H^{\what \pi}}(\what\pi\, | \, \what \gamma)=0$, where 
\begin{equation}\label{linearPerturbe}
  \begin{split}
  &\ell_{\what G}^{\what E, \what H^{\what \pi}}(\what\pi\, | \, \what \gamma) := \ell^{\what E}_{\what G}(\what \pi \, | \, \what \gamma) - \sum_{k=1}^d \int_0^T \scal{}{\partial_{e_k}\what G_t}{\Sigma(\what \rho_t) \partial_{e_k}\what H_t} \, dt, 
  \end{split}
  \end{equation}
  where $\ell^{\what E}$ has been defined in \eqref{weakAlternative}.
  \end{proposition}

\begin{proof}
  With the definitions introduced in Section \ref{sobolev}, for $\what \pi \in \mathcal{D}([0,T],\wm)$ and $\what G \in ( \mathcal C^{1,2}_{0,-}([0,T]\times \overline\La))^2$, we can simply write $J_{\what G}^{\what E}(\what \pi)$ as 
\[
  J_{\what G}^{\what E}(\what \pi) = \ell_{\what G}^{\what E}(\what\pi\, | \, \what \gamma) - \frac12 \norm{1,\Sigma(\what \pi)}{\what G}^2.
\]
Let us consider $\what\pi\in\DM$  such that $\I{\what\pi}< + \infty$. Then by definition of $\I{\cdot}$, there exists a constant $C>0$ such that for all $\what G \in ( \mathcal C^{1,2}_{0,-}([0,T]\times \overline\La))^2$,
\[
  \ell_{\what G}^{\what E}(\what\pi\, | \, \what \gamma) - \frac{1}{2} \norm{1,\Sigma(\what \pi)}{\what G}^2 \leq C,
\]
with $\ell_{\cdot}^{\what E}(\what\pi\, | \, \what \gamma) : ( \mathcal C^{1,2}_{0,-}([0,T]\times \overline\La))^2 \longrightarrow \R $  being a linear functional. We can easily check that this functional is continuous for the topology of $\H$ and by density of $( \mathcal C^{1,2}_{0,-}([0,T]\times \overline\La))^2$ in $\H$, the functional $\ell_{\cdot}^{\what E}(\what\pi\, | \, \what \gamma)$ can be continuously extended  to $\H$. 
Now from the Riesz representation Theorem, there exists a unique $\what H^{\what \pi}\in \H$ such that for all $\what G \in \H$, 
\[
  \ell_{\what G}^{\what E}(\what\pi\, | \, \what \gamma) = \scall{1,\Sigma(\what \pi)}{\what H^{\what \pi}}{\what G},
\]
where $\scall{1,\Sigma(\what \pi)}{\cdot}{\cdot}$ is defined in \eqref{norme1S}.
This implies that the densities $({\what \rho}_t)_{t\in [0,T]}$ of $({\what \pi}_t)_{t\in [0,T]}$ is a weak solution of equation \eqref{Eq:perturbeR+Dlemma}.

Next, let us consider this particular $\what H^{\what \pi}$, we want to prove that $\I{\what \pi} = \frac12 \norm{1,\Sigma(\what \pi)}{\what  H^{\what \pi}}^2$.
  By approximating $\what H^{\what \pi}$ in $( \mathcal C^{1,2}_{0,-}([0,T]\times \overline\La))^2$, we get the lower bound (see Lemma 10.5.3 of \cite{kl}). For the reverse inequality, we simply use the inequality \eqref{inegaliteCarreScalaire} with $S=I_2$.
Finally, the proof of $\what \pi\in \mathcal{B}_{\what \gamma}^{\what b}$ when $\I{\what \pi}< + \infty$ is similar to the one of Lemma 3.5 in \cite{bdgjl3} and is therefore omitted.
\end{proof} 

\subsubsection{An alternative rate function}

Denote $\La_T^+\coloneqq (0,T) \times \left(\La\cup \Ga^+\right)$. We define for  $\what \pi \in \mathcal D([0,T], \wm)$ and  $\what G \in ( \mathcal C_c^\infty(\La_T^+))^2$  the functional 
\begin{equation} \label{defJJ}
  \mathbb J_{\what G}^{\what E}(\what \pi) \coloneqq - \scall{}{\what \pi}{\partial_t \what G} +   \sum_{k=1}^d \scall{}{\partial_{e_k} \what \pi - \frac{1}{2} \Sigma(\what \pi)\what E_k}{\partial_{e_k} \what G} - \frac{1}{2} \norm{1,\Sigma(\what \pi)}{\what G}^2 ,
\end{equation}
and we define the alternative rate function   $\mathbb I^{\what E}_T(\cdot|\what \ga) :\mathcal D([0,T], \wm) \cap \{ \Q_{I_2}(\what \pi) < + \infty\} \longrightarrow [0,+\infty]$ by 
\[
  \mathbb I^{\what E}_T(\what \pi |\what \gamma) \coloneqq \sup_{\what G\in ( \mathcal C_c^\infty(\La^+_T))^2} \mathbb J_{\what G}^{\what E}(\what \pi) . 
\]

The proofs of the two next lemmata are similar to the proofs of Lemma 4.4. and Lemma 4.5. in \cite{blm09} or arguments before equation (4.6) in \cite{flm} and are therefore omitted.

\begin{lemma}\label{lemma:continuitePi}
  Let $A>0$ and $\what J\in ( \mathcal C^2_{0,-}(\La))^2$. For each $\vare > 0$, there exists $\delta>0$, such that 
  \[
    \sup_{\substack{\what \pi \in \DM,\\ \Itilde{\what \pi}\leq A}} \sup_{\abs{s-r}\leq  \delta } \abs{\scal{}{\what \pi_s}{\what J} - \scal{}{\what \pi_r}{\what J}} \leq \vare .
  \]
  In particular, if $\Itilde{\what \pi}< + \infty$, then $\what \pi \in  \mathcal C([0,T],\wm)$. 
\end{lemma} 

\begin{lemma}\label{lemma:Itilde}
  If $\what \pi \in \mathcal{B}^{\what b}_{\what \ga}$ and $\Q_{I_2}(\what \pi)< + \infty$, then 
  \[
    \I{\what \pi} = \mathbb{I}^{\what E}_T(\what \pi|\what \ga). 
  \]
\end{lemma}

\subsubsection{\texorpdfstring{$I^{\what E}_T(\cdot|\what\ga)$}{I(.|gamma)} lower semi-continuity}
    
     Let us define the linear functional $\partial_t \what \pi \in \left(L^2([0,T],\mathcal H_{0,-}^*(\La))\right)^2$ whose value at $\what G \in ( \mathcal C^\infty_c(\La_T^+))^2$, denoted by $\scall{}{\partial_t \what \pi}{\what G}$ is $- \scall{}{\what \pi}{\partial_t \what G}$.

    Remind from Subsection \ref{sobolev} the definition of $\mathcal H_{0,-}^*(\La)$, which is the dual space of $\mathcal{H}_{0,-}(\Lambda)$; from the density of $\left( \mathcal C_c^\infty(\La_T^+)\right)^2 $ in $\mathcal H_{0,-}(\La)$,  the norm $\norm{1,*}{\cdot}$ can be defined as in \eqref{norme1S} with $S=I_2$,
    \[
      \norm{1,*}{\what L}^2 \coloneqq 2 \sup_{\what G\in ( \mathcal C^{\infty}_{c}(\La_T^+))^2}\left\{ \scall{}{\what L}{\what G}  - \frac{1}{2} \norm{1,I_2}{\what G}^2 \right\}.
    \]
    \begin{lemma}\label{lemma:deriveetemporelle}
      For $\what \pi\in\mathcal D([0,T],\wm)$, we have 
      \[
        \norm{1,*}{\partial_t \what \pi}^2 \leq   C_\Sigma \I{\what \pi} + 4 \norm{\left(L^2([0,T],H^1(\La))\right)^2}{\what \pi}^2 + T C_{\what E,\Sigma},
      \]
      where $C_{\what E,\Sigma}$ is a constant depending on $\what E$ and $\Sigma$.
    \end{lemma}
    \begin{proof}
      Recalling the definition, 
      \[
        \norm{1,*}{\partial_t \what \pi}^2 = 2\sup_{\what G\in  ( \mathcal C^{\infty}_{c}(\La_T^+))^2}\left\{ \scall{}{\partial_t \what \pi}{\what G} - \frac{1}{2} \scall{1,I_2}{\what G}{\what G} \right\} .
      \]
      Without loss of generality, we can assume that $\I{\what \pi}$ and $\Q_{I_2}(\what \pi)$ are finite. Let $\what G\in  ( \mathcal C^{\infty}_{c}(\La_T^+))^2$, from two integration by parts, and after adding and substracting the corresponding terms, we have  
      \begin{align*}
        \scall{}{\partial_t \what \pi}{\what G} = \ell_{\what G}^{\what E}(\what \pi\, | \, \what \gamma) -   \sum_{k=1}^d \int_0^T \scal{}{\partial_{e_k} \what \pi_s}{\partial_{e_k} \what G_s} \, ds  + \frac12  \sum_{k=1}^d \int_0^T \scal{}{\Sigma(\what \pi_s) \what E_k}{\partial_{e_k} \what G_s} \, ds ,
      \end{align*}
      because $\I{\what \pi}$ is finite, thus $\what \pi\in \mathcal{B}_{\what \gamma}^{\what b}$ from Proposition \ref{lemma:problemePerturbe}. 
      Using the inequality \eqref{inegaliteCarreScalaire} with $S=4I_2$ (resp. $S=2I_2$) to bound the second (resp. third) term of the right hand side, we directly get
  \begin{align*}
    \frac12 \norm{1,*}{\partial_t \what \pi}^2 \leq & \sup_{\what G\in  ( \mathcal C^{\infty}_{c}(\La_T^+))^2}  \left\{ \ell_{\what G}^{\what E}(\what \pi\, | \, \what \gamma)- \frac14  \sum_{k=1}^d \int_0^T \scal{}{\partial_{e_k} \what G_s}{\partial_{e_k} \what G_s} \, ds \right\} \\
    &+ 2 \norm{\left(L^2([0,T],H^1(\La))\right)^2}{\what \pi}^2 + \frac{1}{2}  \sum_{k=1}^d \int_0^T \scal{}{\Sigma(\what \pi_s)\what E_k}{\Sigma(\what \pi_s)\what E_k} \, ds,
  \end{align*}
  where the last sum is bounded by $T C_{\what E,\Sigma}$ where $C_{\what E,\Sigma}$ a constant depending only on $\what E$ and $\Sigma$. There exists a constant $C_\Sigma$ such that $\Sigma \leq C_\Sigma I_2$, thus taking $\what G\rightarrow 2 C_\Sigma \what G$, we have from Lemma \ref{lemma:Itilde}
  \begin{align*}
    & \sup_{\what G\in  ( \mathcal C^{\infty}_{c}(\La_T^+))^2}  \left\{ \ell_{\what G}^{\what E}(\what \pi\, | \, \what \gamma)- \frac14  \sum_{k=1}^d \int_0^T \scal{}{\partial_{e_k} \what G_s}{\partial_{e_k} \what G_s} \, ds \right\} \\
    \leq & 2 C_\Sigma \sup_{\what G\in  ( \mathcal C^{\infty}_{c}(\La_T^+))^2}  \left\{ \ell_{\what G}^{\what E}(\what \pi\, | \, \what \gamma)-  \frac12  \sum_{k=1}^d \int_0^T \scal{}{\partial_{e_k} \what G_s}{\Sigma(\what \pi_s) \partial_{e_k} \what G_s} \, ds \right\} 
    =  2 C_\Sigma \I{\what \pi}.
  \end{align*}
  Summing up all together, we get the desired result.  \qedhere 
\end{proof}

    \begin{lemma}\label{lemma:majorerNormeParI}
      There exists a constant $C_1$ depending only on $\what E, \La, T$, such that 
      \[
         \norm{\left(L^2([0,T],H^1(\La))\right)^2}{\what\pi}^2 \leq C_1 \left( \I{\what \pi} + 1 \right),
      \]
      for all $\what \pi \in \mathcal D([0,T],\wm)$. 
    \end{lemma}

    \begin{proof}
      First, recall that from \eqref{defQI2energy}, if $\what \pi \in \mathcal D([0,T], \wm)$, $\Q_{I_2}(\what \pi) = \frac{1}{2} \norm{\left(L^2([0,T],H^1(\La))\right)^2}{\what\pi}^2$.
      If $\I{\what \pi}$ is not finite, the inequality is clear. Thus, let us consider $\what \pi\in \mathcal{D}([0,T],\wm)$ such that $\I{\what \pi}$ is finite, then by definition $\Q_{I_2}(\what \pi)$ is also finite. Recall the definition of the smooth function $\what \theta$ in \eqref{eq:cC}. From Proposition \ref{lemma:problemePerturbe}, for all $t\in [0,T]$, $\what \pi_t\vert_{\Ga^-} = \what b$, therefore from the identity $\Q_{I_2}(\what \pi)< + \infty$ and Lemma \ref{lemma:deriveetemporelle} 
\[
      \what \pi - \what \theta \in \left(\W_{[0,T],0,-}^{1,2}\right)^2,
\] 
where $\W_{[0,T],0,-}^{1,2}$ is defined in \eqref{W1,2}.
We now claim that, for all $\what G \in \left(L^2([0,T],\mathcal{H}_{0,-}(\La))\right)^2$ and $\al >0$, 
\begin{equation}\label{dtpiG}
  \scall{}{\partial_t \what \pi}{\what G} \leq \frac{1}{\al} \I{\what \pi} + \frac{\al}{2}\norm{1,\Sigma(\what \pi)}{\what G} - \sum_{k=1}^d \scall{}{\partial_{e_k} \what \pi - \frac{1}{2} \Sigma(\what \pi) \what  E_k}{\partial_{e_k} \what G} .
\end{equation}
  Indeed, we have $\what \pi \in \mathcal{B}_{\what \gamma}^{\what b}$ (cf. Proposition \ref{lemma:problemePerturbe}) and $\Q_{I_2}(\what \pi) < + \infty$, then from Lemma \ref{lemma:Itilde},  and by definition of $\scall{}{\partial_t \what \pi}{\cdot}$,
  \begin{align*}
    \I{\what \pi} =& \sup_{\what G\in \left(L^2([0,T],\mathcal{H}_{0,-}(\La))\right)^2} \bigg\{  \scall{}{\partial_t \what \pi}{ \what G} +  \sum_{k=1}^d \scall{}{\partial_{e_k} \what \pi - \frac{1}{2}\Sigma(\what \pi)\what E_k}{\partial_{e_k} \what G} - \frac{1}{2} \norm{1,\Sigma(\what \pi)}{\what G}^2  \bigg\} ,
  \end{align*}
  where we have used the density of $\left( \mathcal C^\infty_c(\La_T^+)\right)^2$ in $\left(L^2([0,T],\mathcal{H}_{0,-}(\La))\right)^2$ and the continuity of the function $\mathbb J_\cdot(\what \pi)$ defined in \eqref{defJJ}.
   Then, for all $\what G \in \left(L^2([0,T],\mathcal H_{0,-}(\La))\right)^2$, 
  \[
    \scall{}{ \partial_t \what \pi}{ \what G} +  \sum_{k=1}^d \scall{}{\partial_{e_k} \what \pi - \frac{1}{2}\Sigma(\what \pi)\what E_k}{\partial_{e_k} \what G} - \frac{1}{2} \norm{1,\Sigma(\what \pi)}{\what G}^2 \leq \I{\what \pi} .
  \]
  Now, for all $\al>0$, taking $\al \what G$ in place of $\what G$ and dividing the expression by $\al$ we get the claim.
Moreover, rearranging the terms in \eqref{dtpiG}, from the inequality \eqref{inegaliteCarreScalaire}, with $S = \al I_2$,
\[
  \sum_{k=1}^d \scall{}{\partial_{e_k} \what \pi}{\partial_{e_k} \what G} \leq \frac{1}{\al} \I{\what \pi} + \al \norm{1,\Sigma(\what \pi)}{\what G}^2 - \scall{}{\partial_t \what \pi}{\what G} + \frac{C}{\al} ,  
\]
where $C$ is a constant depending on $\what E,\La,T$.
We now apply the last inequality to the particular function $\what G \coloneqq \what \pi - \what \theta$ which belongs to $\left(L^2([0,T], \mathcal H_{0,-}(\La))\right)^2$ as we previously mentioned. 
The left hand side is equal to 
\[
  \sum_{k=1}^d \scall{}{\partial_{e_k}\what \pi}{\partial_{e_k} \what G} = \sum_{k=1}^d\scall{}{\partial_{e_k} \what \pi}{  \partial_{e_k} \what \pi} -\sum_{k=1}^d \scall{}{\partial_{e_k} \what \pi}{ \partial_{e_k} \what \theta} ,
\]
while using Lemma \ref{lemma:inegrationByPartsRoubicek} on the third term of the right hand side, 
\[
      \abs{\scall{}{\partial_t \what \pi}{\what \pi - \what \theta}} = \abs{\frac12\scal{}{\what\pi_T - \what \theta_T}{\what\pi_T - \what \theta_T} + \frac12\scal{}{\what\pi_0 - \what \theta_0}{\what\pi_0 - \what \theta_0}}\leq 16. 
\]
In addition, since there exists $C_\Sigma>0$ such that $\Sigma\leq C_{\Sigma}I_2$, we have $\norm{1,\Sigma(\what \pi)}{\what G}^2 \leq C_\Sigma \scall{}{\grad \what G}{\grad \what G}$. Collecting the above inequalities, we get 
\[
  \sum_{k=1}^d \scall{}{\partial_{e_k} \what \pi}{\partial_{e_k} \what \pi} \leq\sum_{k=1}^d \scall{}{\partial_{e_k} \what \pi}{ \partial_{e_k} \what \theta } + \frac1\alpha \I{\what \pi} + \alpha C_{\Sigma} \norm{\left(L^2([0,T],L^2(\La))\right)^2}{\grad \what G}^2+ 16 + \frac C\alpha .
\]
Using inequality \eqref{inegaliteCarreScalaire} again with $S=2 \al C_\Sigma I_2$ for the first term of the right hand side, and using inequality \eqref{BasicSommeCarre} for the $\norm{\left(L^2([0,T],L^2(\La))\right)^2}{\grad \what G}^2$ term, we have 
\[
  \norm{\left(L^2([0,T],L^2(\La))\right)^2}{\grad \what \pi}^2 \leq  \al C_\Sigma \norm{\left(L^2([0,T],L^2(\La))\right)^2}{\grad \what \pi}^2 + \frac{C_{\Sigma, \what \theta } + C}{\al} +  \frac1\alpha \I{\what \pi} + \alpha C_{\Sigma,\what \theta} + 16 ,
\] 
where the constants may have different values than before. Finally taking $\al= \frac{1}{2 C_\Sigma}$ we get the desired result. \qedhere
\end{proof}

    \begin{theorem}\label{thm:sciI}
      The function $\I{\cdot} : \DM \longrightarrow [0,+\infty]$ is lower semicontinuous with compact sublevel sets.
    \end{theorem}
    \begin{proof}
      To show that $\I{\cdot}$ is lower semicontinuous, we have to show that for all $\la\geq0$, 
      \[
        E_\la \coloneq \left\{ \I{\what \pi} \leq \la \right\}
      \]
      is a closed set in $\DM$. Let $\la \geq 0$ and $\left(\what \pi^n\right)_{n\in\N} \in (E_\la)^\N$ such that $\what\pi^n \xrightarrow[n\to + \infty]{\DM} \what \pi \in \DM$. Let us show that $\what \pi$ belongs to $E_\la$. From the Skorokhod topology of $\DM$ we directly have that $\what \pi^n$ converges weakly to $\what \pi$ in  $\left(L^2([0,T]\times\La)\right)^2$.  From Lemma \ref{lemma:deriveetemporelle} and Lemma \ref{lemma:majorerNormeParI}, there exists a constant $C_\la \coloneqq C(\la,\Sigma,\La)$ such that for all $n\in\N$ 
      \[
        \norm{1,*}{\partial_t \what\pi^n}^2 +  \norm{\left(L^2([0,T],H^1(\La))\right)^2}{\what \pi^n}^2  \leq C_\la ,
      \]
      thus, from Lemma \ref{lemma:convergenceFaibleINTOforte}, $\what\pi^n$ converges strongly to $\what\pi$ in $\left(L^2([0,T]\times \La)\right)^2$. 

     To show that $\I{\what\pi} \leq \la$, it is enough to show that $\what \pi $ belongs to $\mathcal{B}^{\what b}_{\what \gamma}$,  that $\Q_{I_2}(\what \pi)< + \infty$ and that for all $G \in  \left( \mathcal C^\infty_c (\La_T^+)\right)^2$, $\mathbb{J}_{\what G}^{\what E}(\what \pi) \leq \la$, where $\mathbb{J}_{\what G}^{\what E}(\what \pi)$ is defined in \eqref{defJJ}, so that   
      \begin{align*}
        \I{\what \pi} =  \sup_{\what G\in \left( \mathcal C^\infty_c (\La_T^+)\right)^2} \underbrace{\bigg\{ - \scall{}{\what\pi}{\partial_t \what G} +  \sum_{k=1}^d \scall{ }{\partial_{e_k} \what \pi - \frac{1}{2} \Sigma(\what \pi  )\what E_k  }{\partial_{e_k} \what G} - \frac12 \norm{1,\Sigma(\what\pi)}{\what G}^2 \bigg\}}_{= \mathbb{J}_{\what G}^{\what E}(\what \pi)}\leq \la ,
      \end{align*}
      using Lemma \ref{lemma:Itilde}.

    We already mentioned that $\sup_{n} \norm{\left(L^2([0,T],H^1(\La))\right)^2}{\what \pi^n}$ is finite,  therefore, from the reflexivity of the Hilbert space $\left(L^2([0,T],H^1(\La))\right)^2$, there exists a subsequence $(\what \pi^{n_k})_{k\in\N}$ of $(\what \pi^n)$ such that $ \what \pi^{n_k}$ converges weakly in $\left(L^2([0,T],H^1(\La))\right)^2$ to some $\tilde \pi \in \left(L^2([0,T],H^1(\La))\right)^2$. But the weak convergence in $\left(L^2([0,T],H^1(\La))\right)^2$ implies, in particular, the weak  convergence in $\left(L^2([0,T]\times\La)\right)^2$, thus, by uniqueness of the weak limit we have that $\tilde \pi = \what \pi$, this implies that by uniqueness of a subsequential limit,  $\what\pi^n$ converges weakly to $\what \pi$ in $\left(L^2([0,T],H^1(\La))\right)^2$ and that $\Q_{I_2}(\what \pi) < + \infty$.

     We now need to prove that $\what \pi$ belongs to $\mathcal{B}^{\what b }_{\what \gamma}$. From the Skorokhod topology, for all $\what G\in \left( \mathcal C^\infty(\La)\right)^2$,  the function $\DM  \ni \tilde \pi \mapsto \scal{}{\tilde \pi_0}{G}$ is continuous. Therefore 
      \[
        \scal{}{\what \pi_0^n}{G} = \scal{}{\what \gamma}{\what G} \xrightarrow[n\to + \infty]{}  \scal{}{\what \gamma}{\what G} = \scal{}{\what \pi_0}{\what G},
      \]
      thus $\what \pi_0 = \what \gamma$. Furthermore, from the weak convergence of $\what \pi^n$ to $\what \pi$ in $\left(L^2([0,T],H^1(\La))\right)^2$ and the fact that $\what \pi^{n}\vert_{\Ga^-}=\what b$ from Proposition \ref{lemma:problemePerturbe}, applying Lemma \ref{Lemma:ConvergenceFaibleBord}, we get $\what \pi|_{\Ga^-} = \what b$ and $\what \pi \in \mathcal{B}^{\what b} _{\what \gamma}$. Hence, from Lemma \ref{lemma:Itilde},
      \[
        \I{\what \pi} = \mathbb{I}^{\what E}_T(\what \pi |\what \gamma)  = \sup_{\what G \in \left( \mathcal C^\infty_c (\La_T^+)\right)^2} \mathbb J_{\what G}^{\what E}(\what \pi).
      \]

Let's fix a function $\what G \in \left( \mathcal C^\infty_c (\La_T^+)\right)^2$, recall the definition of $\mathbb J^{\what E}_{\what G}$ in \eqref{defJJ}  and treat the three terms of $\mathbb{J}_{\what G}^{\what E}({\what {\pi}}^n)$ separately. From the strong convergence of $\what \pi^n$  to $\what \pi$ in $\left(L^2([0,T]\times\La)\right)^2$, we have 
      \begin{equation}\label{sci01}
        \scall{}{\what \pi^n}{\partial_t \what G} \xrightarrow[n\to + \infty]{} \scall{}{\what \pi}{\partial_t \what G}, 
      \end{equation}
    \begin{align}\label{sci1}
      \sum_{k=1}^d    \abs{\scall{}{\left[\Sigma(\what\pi^n) - \Sigma(\what \pi) \right]\what E_k }{\partial_{e_k}\what G}} \leq C_{\Sigma,\what G, \what E} \norm{\left(L^2([0,T]\times \La)\right)^2}{\what \pi^n - \what \pi}^2 \xrightarrow[n\to + \infty]{} 0 , 
      \end{align}
      and
      \begin{align}
        \abs{\norm{1,\Sigma(\what \pi^n)}{\what G}^2 - \norm{1,\Sigma(\what\pi)}{\what G }^2}& = \abs{ \sum_{k=1}^d\scall{}{\left[\Sigma(\what\pi^n) - \Sigma(\what \pi) \right] \partial_{e_k} \what G }{\partial_{e_k} \what G}}\\
        \ &\leq C_{\what G} \norm{\left(L^2([0,T]\times \La)\right)^2}{\what \pi ^n - \what \pi}^2 \xrightarrow[n\to + \infty]{} 0, \label{sci2}
      \end{align}
      where we used the Lipschitz continuity of $\Sigma$ and Cauchy-Schwarz inequality. For the remaining term  
      \begin{align}\label{sci3}
        \sum_{k=1}^d\scall{}{\partial_{e_k} \what \pi^n}{\partial_{e_k} \what G} \xrightarrow[n\to + \infty]{}\sum_{k=1}^d\scall{}{\partial_{e_k} \what \pi}{\partial_{e_k} \what G},
      \end{align}  
      using the weak convergence of $\what \pi^n$ to $\what \pi$ in $\left(L^2([0,T],H^1(\La))\right)^2$, implying the weak convergence of $\partial_{e_k} \what \pi^n$ to $\partial_{e_k} \what \pi $ in $\left(L^2([0,T]\times \La)\right)^2$ for all $k\in\{1,\dots,d\}$.      
      Since for all $n\in\N$,  $\mathbb{J}_{\what G}^{\what E}(\what \pi^n) \leq \la$,
      the limits \eqref{sci01}, \eqref{sci1}, \eqref{sci2}, and \eqref{sci3} yield 
      \[
      \lim_{n\to\infty}  \mathbb{J}_{\what G}^{\what E} (\what \pi^n)=\mathbb{J}_{\what G}^{\what E}(\what \pi) \leq \la,
      \]
for all $\what G \in \left(  \mathcal C_c^\infty(\La_T^+)\right)^2$. Thus the set $E_\la$ is closed and this completes the proof of the lower semi-continuity of the functional $\I{\cdot}$  in $\DM$.

      To prove that the sublevel sets are compact sets of $\DM$, we have to show that $E_\la$ is compact. Using the characterization of compact sets of $\DM$ stated in Section 4.1 of \cite{kl}, it suffices to use Lemma \ref{lemma:continuitePi} to get the desired result (see the proof of Theorem 4.2 from \cite{blm09}).
    \end{proof}

    \subsubsection{Comparison between \texorpdfstring{$I^{\what 0}_T(\cdot|\what\gamma)$}{I \hat{0}} and \texorpdfstring{$\I{\cdot}$}{I \hat{E} }}

    Setting $\what E = \left(0_{\R^d},0_{\R^d}\right)$ in the problem \eqref{eq:1}, one gets the following boundary value problem for the heat equation 
\begin{equation} \label{Eq:heat}
  \left\{
  \begin{array}{rl}
      \displaystyle\partial_{t}\what \rho & = \Delta \what \rho ,\\[0.2cm]
      \displaystyle\what \rho|_{\Gamma^-} & =\what b,\\[0.2cm]
      \displaystyle\partial_{e_{1}}\what \rho|_{\Gamma^{+}} & = \what  0, \\[0.2cm]
      \displaystyle\what \rho_0 & =\what \gamma.
  \end{array}
\right.
\end{equation}

\begin{lemma}\label{lemma:comparison}
  For $\what \pi\in\DM$, we have 
  \begin{align}
    \I{\what \pi} &\leq 2 I_T^{\what 0}(\what \pi|\what \gamma) + \frac14 \norm{\Sigma(\what \pi)}{\what E}^2 , \\
    I_T^{\what 0}(\what \pi|\what \gamma) &\leq 2 \I{\what \pi}  + \frac14 \norm{\Sigma(\what \pi)}{\what E}^2 . \label{eq:itzero}
  \end{align}
\end{lemma}
\begin{proof}
  If $\what \pi \notin \mathcal{D}([0,T]),\wm)$ or $Q_{I_2}(\what \pi) = + \infty$, the inequalities are clear, hence let us consider $\what \pi \in \mathcal{D}([0,T],\Mzero)$ with $\Q_{I_2}(\what \pi) < + \infty$. For $\what G \in \left( \mathcal C^{1,2}_{0,-}([0,T]\times \La)\right)^2$, we have 
  \[
    \abs{ \ell^{\what 0}_{\what G}(\what \pi\, | \, \what \gamma) - \ell^{\what E}_{\what G}(\what \pi\, | \, \what \gamma)} = \frac12 \abs{ \sum_{k=1}^d \scall{}{\partial_{e_k}\what G}{\Sigma(\what \pi)\what E_k}} ,
  \] 
  thus, using inequality \eqref{inegaliteCarreScalaire} with $S=I_2$, we get 
  \[
    J_{\what G}^{\what E}(\what \pi) \leq \sup_{\tilde G \in \left( \mathcal C^{1,2}_{0,-}([0,T]\times \La)\right)^2} \left\{ \ell^{\what 0}_{\tilde G}(\what \pi\, | \, \what \gamma) - \frac{1}{4}\norm{1,\Sigma(\what \pi)}{\tilde G}^2 \right\} + \frac14 \norm{\Sigma(\what\pi)}{\what E}^2.
  \] 
  It remains to take $\tilde G \rightarrow 2 \tilde G$ to obtain the first inequality of the Lemma. For the second inequality, we follow the same steps with $J^{\what 0}_{\what G}$ to get the result the same way.
\end{proof}

    \subsubsection{Perturbed proccess}\label{Subsection:PerturbedProcess}
    
    In this section, we consider a perturbation of the original process essential for the proof of the lower bound. 
    Let $\what H \in \left(  \mathcal C^{1,2}_{0,-}([0,T]\times \overline\La)\right)^2$, and introduce
    a new  Hamiltonian ${\mathscr H}_N^{\what H}$  defined through $\what H$ as
\begin{equation*}
\label{HaH}
{\mathscr H}_N^{\hat H}(\s)\, =\, -\sum_{i=1}^2\sum_{x\in\Lambda_N}  \sigma(x)^i H_i\left(\frac{x}{N}\right) .
\end{equation*}
    Define  the generator of a time inhomogeneous Markov process  
    \[
      \mathfrak L^{\what H}_N = N^2 \mathcal L_{\what E, N}^{\what H} + N^{2-\mathfrak{a}_l} L_{b,N}^{-,\what H} + N^{2-\mathfrak{a}_r} L_{b,N}^{+,\what H},
    \]
    defined  for any function $f : \mathbb{X}_N \longrightarrow \R$ by 
    \[
      \left(\mathcal L_{\what E, N}^{\what H} f \right)(\sigma) \coloneqq   \sum_{k=1}^d \sum_{x,x+e_k\in\La_N} C_N^{\what E, \what H}(x,x+e_k;\sigma) \left[ f(\sigma^{x,x+e_k}) - f(\sigma) \right], 
    \]
    \begin{equation}\label{gen-bord-perturbe}
      \begin{aligned}
      (L^{\pm,\what H}_{\what b,N} f)(\sigma)  \;\coloneqq\; & \: \sum_{x \in \Gamma_N^\pm} r_x^{(1,0),-}\big(\widehat b,\sigma\big) \exp\left\{ {\mathscr H}_N^{\hat H} \left(\sigma^{x,-}\right) - {\mathscr H}_N^{\hat H} \left(\sigma\right)
      \right\} \big[ f(\sigma^{x,-})-f(\sigma)\big] \\
      \  &  +\; \: \sum_{x \in \Gamma_N^\pm}
      r_x^{(0,1),+}\big(\widehat b,\sigma\big) \exp\left\{ {\mathscr H}_N^{\hat H} \left(\sigma^{x,+}\right) - {\mathscr H}_N^{\hat H} \left(\sigma\right) \right\} \big[ f(\sigma^{x,+})-f(\sigma)\big]\\
      \  &  +\; \: \sum_{x \in \Gamma_N\pm}r_x^{(-1,0),+}\big(\widehat b,\sigma\big) \exp\left\{{\mathscr H}_N^{\hat H} \left(\sigma^{x,+}\right) - {\mathscr H}_N^{\hat H} \left(\sigma\right)\right\} \big[ f(\sigma^{x,+})-f(\sigma)\big] \\ 
      \  &  +\; \:  \sum_{x \in \Gamma_N\pm}   r_x^{(0,-1),-}\big(\widehat b,\sigma\big) 
      \exp\left\{ {\mathscr H}_N^{\hat H} \left(\sigma^{x,-}\right) - {\mathscr H}_N^{\hat H} \left(\sigma\right)\right\} \big[ f(\sigma^{x,-})-f(\sigma)\big]\\
      \  &  +\; \: \sum_{x \in \Gamma_N\pm} r_x^{(-1,1)}\big(\widehat b,\sigma\big) \exp\left\{ {\mathscr H}_N^{\hat H} \left(T^x\sigma\right) - {\mathscr H}_N^{\hat H} \left(\sigma\right)\right\} \big[ f(T^x\sigma)-f(\sigma)\big] ,
      \end{aligned}
      \end{equation}
      where for $x,y\in \La_N$,  
      \[ C_N^{\what E,\what H}(x,y;\sigma)=
      C_N^{\what E}(x,y;\sigma) \exp\left\{ - \nabla_{x,y}{\mathscr H}_N^{\hat H} (\s)  \right\},
      \]
      and the rates $r_x^{(0,\pm1),\pm}, r_x^{(\pm1,0),\mp}, r_x^{(-1,1)}$ are defined in \eqref{rate-b}.
      Denote by $\mathbb P_{\mu}^{N,\what H}$ the probability measure on $\DM$ induced by the Markov process $(\sigma_t)_{t\in [0,T]}$ whose generator is $\mathfrak L^{\what H}_N$ starting from $\mu$ and we denote by $\mathbb Q^{N,\what H}_\mu$ the law of the process $(\what \pi^N(\sigma_t))_{t\in [0,T]}$. Computing the same martingale as the original process we get the following hydrodynamic result. 

      \paragraph{\sl{Hydrodynamic limit}}

Let $\widehat \gamma = (\gamma_1,\gamma_2) : \overline \Lambda \to {\mb I}$ 
be an initial profile, and denote by 
$\widehat \rho = (m,\phi) : [0,T] \times \overline \Lambda \to \bar {\mb I}$ 
a typical macroscopic trajectory.  We shall prove in Theorem 
 \ref{th-hy-perturbe} below that the macroscopic evolution of  the  local 
particle  density $\widehat \pi^N$ under this process, is described by  the following system of 
non-linear coupled equations
\begin{equation} \label{Eq:perturbeR+D-hydrodynamic}
  \left\{
  \begin{array}{rl}
    \displaystyle \partial_t \what\rho &= \displaystyle \grad \cdot \left[\grad \what\rho -\frac{1}{2}  \left(\what E + 2 \grad \what H \right)\Sigma(\what\rho) \right],\\  
    \displaystyle \what\rho|_{\Ga^-} &= \displaystyle \what b ,  \\
    \displaystyle \partial_{e_1}\what\rho|_{\Ga^+} &= \displaystyle \frac{1}{2}  \left(\what E_1 + 2 \partial_{e_1} \what H \right)\Sigma(\what\rho) , \\[0.2cm]
    \what\rho_0 &= \what \gamma.
  \end{array}
\right.
\end{equation}

For   $\widehat G \in ( \mathcal C^{1,2}_{0,-}([0,T]\times \L))^2$, $\wrho =(\rho_1,\rho_2) \in \mathcal{D}([0,T], \wm)$ denote 
\begin{equation}\label{weakAlternativeperturbe}
  \begin{split}
   &\ell_{\wg}^{\what H, \what E}(\wrho) :=  \ell^{\what E}_{\what G}(\what \rho\, | \, \what \rho_0) - \sum_{k=1}^d \int_0^T
  \scal{}{\partial_{e_k} \what G_t}{\Sigma(\what\rho_t)  \partial_{e_k} \what H } dt, 
  \end{split}
  \end{equation}
where $\ell^{\what E}_{\cdot}(\what \rho\, | \, \what \gamma)$ is the linear function defined in \eqref{weakAlternative}.

  \begin{definition}
      Denote by  $\mc A _{[0, T]}^{\what H}$  the set of all weak  solutions of the  
 hydrodynamic limit:
\begin{equation*}
\mc A^{\what H}_{[0, T]} = \Big\{  \wrho \in  \big(L^2\big([0,T], H^1(\La)\big)\big)^2 \; :\; 
\quad \forall  \wg \in \big(  \mathcal C^{1,2}_{0,-} ( [0,T]
 \times \L)\big)^2 \, ,\; {\ell}_{\wg}^{\what H, \what E}(\wrho)=0
   \Big\}.
\end{equation*}
A weak solution to the  boundary value problem \eqref{Eq:perturbeR+D-hydrodynamic} is a function 
$\widehat \rho =(\rho_1,\rho_2):[0,T]\times \Lambda\to [-1,1]\times [0,1] $ 
satisfying (IB1') and (IB2) below:
\begin{itemize} 
\item[(IB1')] $\displaystyle \widehat \rho\in \mc A _{[0, T]}^{\what H}$,
\item[(IB2)]  $\what \rho_0 = \widehat \gamma $ a.e..
\end{itemize} 
  \end{definition}

\begin{proposition}
    \label{th-hy-perturbe}
    For any sequence of initial probability measures $(\mu_N)_{N\ge 1}$, the sequence of probability measures 
    $(\mathbb{Q}_{\mu_N}^{N,\what H})_{N\geq 1}$  
    is weakly relatively compact in $\DM$ and  all its converging
    subsequences converge to   some limit  $\mathbb{Q}^{*,\wb,\what H}$ that is concentrated on the set of paths $\what\pi = (\rho_1\, du,\rho_2\, du)\in \mathcal{D}([0,T],\what{\mathcal{M}}_0)$ such that $\what \rho\coloneqq (\rho_1,\rho_2)$ is a weak solution to the hydrodynamic equation in the sense that $\what \rho \in \mc A^{\what H}_{[0,T]}$. Furthermore, $\what \rho$ satisfies the following inequality   
    $$ \sum_{k=1}^d \scall{}{\partial_{e_k} \what \rho}{\Sigma(\what\rho)^{-1}\partial_{e_k} \what \rho} < + \infty.$$  Moreover, if the sequence of initial measures $(\mu_N)_{N\geq 1}$
is associated to some initial profile $\widehat\gamma=(\gamma_1,\gamma_2) : \overline \Lambda \to {\mb I}$, in the sense of \eqref{pfl}, then the sequence of probability measures $(\mathbb{Q}_{\mu_N}^{N,\what H})_{N\geq 1}$ 
    converges to the Dirac measure  concentrated on the unique   weak  solution $\wrho(\cdot,\cdot)$ of the boundary value 
    problem \eqref{Eq:perturbeR+D-hydrodynamic}. Accordingly, for any $t\in [0,T]$, 
    any $\delta>0$ and any function $\wg\in  \left(\mc C(\Lambda)\right)^2$ 
    $$
    \lim_{N\to \infty}\PP_{\mu_N}^{N,\what H}\Big\{ \Big| \langle  \widehat\pi_N(\sigma_t), \wg\rangle \, -\, 
    \langle\wrho(t,\cdot),\wg \rangle  
     \Big| \ge \delta \Big\}=0.
    $$
    \end{proposition}
    The proof of this Proposition follows the same lines as the one of Proposition \ref{th-hy} and is therefore omitted.

\subsection{Exponential martingales}

In this section we introduce exponential martingales needed in both the proof of the upper bound and lower bound. For any $\what G \in ( \mathcal C^{1,2}_{0,-}([0,T]\times \overline\La))^2$, the mean one exponential martingale defined as 
\begin{equation}\label{MartExp}
        \mathbb M_t^{\what G} \coloneqq \exp \left\{ N^d \left[ \scal{}{\what \pi^N_t}{\what G_t} - \scal{}{\what \pi^N_0}{\what G_0} - \frac{1}{N^d}\int_0^T e^{-N^d\scal{}{\what\pi^N_s}{\what G_s}}\left(\partial_s + \mathfrak{L}_N\right) e^{N^d\scal{}{\what \pi^N_s }{\what G_s}} \, ds  \right] \right\}.
    \end{equation}
    Let us denote for $\delta,\vare >0$, the sets 
    \begin{align*}
      B^{\what G, \what{\mathfrak{h}}}_{N,\vare,\delta} &\coloneqq \bigcap_{i=1}^2 \bigcap_{j=1}^3 \left\{ \sigma \in \mathcal{D}([0,T], \mathbb X_N) ~\middle| ~ \abs{\int_0^T  \sum_{j=1}^3 V_{N,\vare}^{G_i, {\mathfrak{h}}_j}(t,\sigma_t)\, dt} \leq \delta  \right\}, \\
      E^{l,\what G}_{N,\delta} & \coloneqq \bigcap_{i=1}^2 \left\{  \sigma \in \mathcal{D}([0,T],\mathbb X_N) ~\middle|~  \abs{\int_0^T \frac{1}{N^{d-1}} \sum_{x\in \Ga^-_N} G_{i,s}\left(\frac{x}{N}\right) \left(\sigma_s(x)^i - b_i\left(\frac{x}{N}\right)\right) \, ds} \leq \delta  \right\},\\
      E^{r,\what G }_{N,\vare, \delta} & \coloneqq \bigcap_{i=1}^2 \left\{  \sigma \in \mathcal{D}([0,T],\mathbb X_N) ~\middle|~  \abs{\int_0^T \frac{1}{N^{d-1}} \sum_{x\in \Ga^+_N} G_{i,s}\left(\frac{x}{N}\right) \left(\sigma_s(x)^i - \sigma^{\ent{\vare N},i}(x)\right) \, ds} \leq \delta \right\},
    \end{align*}
    where $\what{\mathfrak{h}} =({\mathfrak{h}}_1,{\mathfrak{h}}_2,{\mathfrak{h}}_3)$ has already been defined in \eqref{definitionDesPsi}.  Then if we define the set 
  \begin{equation}\label{definitionR}
    \mathcal{R}_{N,\vare,\delta}^{\what G} \coloneqq B^{\what G, \what{\mathfrak{h}}}_{N,\vare,\delta} \cap E^{l,\what G}_{N,\delta} \cap E^{r,\what G }_{N,\vare, \delta},
  \end{equation}
  we can rewrite the martingale on the set  $\mathcal{R}_{N,\vare,\delta}^{\what G}$, as
  \begin{equation}\label{approxExpMart}
    \mathbb M_t^{\what G} \1_{\mathcal{R}_{N,\vare,\delta}^{\what G}}(\sigma) = \exp \left\{ N^d \left[ J^{\what E}_{\what G}(\what \pi^{\vare N}(\sigma)) + O_{\what G}\left(\frac{1}{N}\right) + O_{\what G}(\vare) + O(\delta) \right] \right\}.
  \end{equation}
    
\subsection{Lower bound}

    \subsubsection{\texorpdfstring{$\I{\cdot}$}{I(.|gamma)}-Density}
    In this section, we show that any trajectory $\what \pi \in \DM$ with finite rate function $\I{\what \pi}$ can be approximated by a sequence of smooth trajectories $(\what \pi^n)_{n\in\N}$ such that 
    \[
        \what \pi^n \xrightarrow[n\to + \infty]{\DM} \what \pi \qquad \text{and} \qquad \I{\what \pi^n} \xrightarrow[n\to + \infty]{} \I{\what \pi}.
    \]
    
    \begin{definition}
      A subset $\mathcal{A}$ is said to be $\I{\cdot}$-dense if for every $\what \pi\in\DM$ verifying $\I{\what \pi}< + \infty$, there exists a sequence $(\what \pi^n)_{n \in \N}\in\mathcal A ^{\N}$ such that $\what \pi^n$ converges to $\what \pi$ in $\DM$ and $\I{\what \pi^n}$ converges to $\I{\what \pi}$. 
    \end{definition}
    \begin{definition}
      Let $\mathcal{A}_1$ be the subset of $\DM$ consisting of trajectories $\what \pi$ verifying $\I{\what \pi}< + \infty$ and for which there exists $\delta >0$ such that $\what \pi$ is a weak solution of the heat equation \eqref{Eq:heat} in the time interval $[0,\delta]$. 
    \end{definition}
    \begin{lemma}
      The set $\mathcal{A}_1$ is $\I{\cdot}$-dense.
    \end{lemma}
    \begin{proof}
      The proof is very similar to the proof of Lemma 5.1 in \cite{flm} and is therefore omitted.\qedhere
    \end{proof}

    For $\vare >0$, we define the space 
    \[
  \mathbf{I}_{ \vare} \coloneqq \left\{ (m,\phi) \in[-1,1]\times[0,1] ~\middle|~ 
  \abs{m} + \vare \leq \phi   \leq 1 - \vare  \right\} .
\]
    \begin{definition}
      Denote by $\mathcal{A}_2$ the subset of $\mathcal{A}_1$ of all trajectories $\what\pi$ verifying that for all $\delta\in (0,T]$, there exists $\vare>0$ such that $\what\pi_t(u) \in  \mathbf{I}_{\vare}$ for almost all $(u,t)\in \overline\La\times [\delta,T]$. 
    \end{definition}
    \begin{lemma}
      The set $\mathcal{A}_2$ is $\I{\cdot}-$dense.
    \end{lemma}
    \begin{proof}
      From the previous density result, it suffices to show that for every $\what \pi\in \mathcal{A}_1$ there exists a sequence $(\what \pi^n)_{n\in\N} \in (\mathcal{A}_2)^\N$ such that $\what \pi^n$ converges to $\what \pi$ in $\DM$ and $\I{\what \pi^n}$ converges to $\I{\what \pi}$.  For this purpose, let $\what \pi \in \mathcal{A}_1$, and let $\what \rho^0\eqqcolon (m^0,\phi^0)$ be the solution of the heat equation \eqref{Eq:heat}. For $\vare \in (0,1)$, we define 
      \[
        (m^\vare,\phi^\vare )\coloneqq\what \pi^\vare\coloneqq (1-\vare)\what \pi + \vare \what \rho^0.
      \]
      It is clear that $\what \pi^\vare \xrightarrow[\vare\to 0]{\DM} \what \pi$. Let us first show that $\what \pi^\vare$ belongs to $\mathcal{A}_2$ for all $\vare>0$. From Lemma \ref{rho0Iepsilon}, for all $\delta>0$, there exists $\la\in(0,1)$ such that for almost all $(u,t)\in \overline\La\times [\delta,T]$, 
          \[
            \abs{m^0_t(u)} + \lambda \leq \phi^0_t(u) \leq 1 - \la.
          \]
          We can check that for the same $\delta$, we have for almost all $(u,t)\in \overline\La\times [\delta,T]$ that $\what \pi^\vare_t(u)\in \mathbf{I}_{\vare\la}$. We now show that $\Q_{I_2}(\what \pi^\vare) < + \infty$. By convexity of $\Q_{I_2}$,
          \[
            \Q_{I_2}(\what \pi^\vare) \leq (1-\vare) \Q_{I_2}(\what \pi) + \vare \Q_{I_2}(\what \rho^0) \leq \Q_{I_2}(\what \pi) + \Q_{I_2}(\what \rho^0), 
          \]
          which is finite because $\I{\what \pi}$ is also finite and $Q_{I_2}(\what \rho^0)$ is finite from Proposition \ref{th-hy}. Next, because $\what \pi\in\mathcal{A}_1$, by definition, there exists $\delta > 0$ such that $\what\pi^\vare$ is the solution to the heat equation \eqref{Eq:heat} in $[0,\delta]$. To conclude that $\what \pi^\vare \in \mathcal A_2$ is is enough to show that $\I{\what \pi^\vare}< + \infty$.  Splitting the supremum we can use the following estimation 
          \begin{align*}
            \I{\what \pi^\vare} \leq A_1 + A_2,
          \end{align*}
          where 
          \begin{align*}
            A_1 (\what \pi^\vare) \coloneqq \sup_{\what G\in \left( \mathcal C^{1,2}_{0,-}([0,T]\times \La)\right)^2} \bigg\{  \int_0^T \scal{}{\partial_s \what \pi^\vare_s}{\what G_s} \, ds &+ \sum_{k=1}^d \int_0^T \scal{}{\partial_{e_k}\what \pi^\vare_s}{\partial_{e_k}\what G_s}\, ds \\
            & - \frac{1}{4} \sum_{k=1}^d\int_0^T \scal{}{ \partial_{e_k}\what  G_s}{\Sigma(\what\pi^\vare_s) \partial_{e_k}\what  G_s} \, ds \bigg\} , 
          \end{align*}
          and 
          \begin{align*}
            A_2(\what\pi^\vare) \coloneqq \sup_{\what G\in \left( \mathcal C^{1,2}_{0,-}([0,T]\times \La)\right)^2} \bigg\{ \frac12 \sum_{k=1}^d \int_0^T \scal{}{\partial_{e_k}\what G_s }{\Sigma(\what\pi^\vare_s) \what E_k} \, ds - \frac{1}{4} \sum_{k=1}^d \int_0^T \scal{}{ \partial_{e_k}\what  G_s}{\Sigma(\what\pi^\vare_s) \partial_{e_k}\what  G_s} \, ds \bigg\} .
          \end{align*}
          Using the classical inequality \eqref{inegaliteCarreScalaire} with $S=I_2$, we easily bound $A_2(\what \pi^\vare)$ by some constant $C_{T,\what E,\Sigma}$ depending on $T$, $\what E$, $\Sigma$. For $A_1(\what\pi^\vare)$,  by the concavity of $\Sigma(\cdot)$ (see the proof of Lemma \ref{l-energy1}), using that $\what \rho^0$ is the solution to the heat equation \eqref{Eq:heat} in $[0,T]$ and because $\Sigma(\what \rho^0)$ is a positive matrix,  
      \begin{align*}
        A_1(\what\pi^\vare) \leq & (1-\vare) A_1(\what\pi).
      \end{align*}
      Replacing $\what G$ by $ 2 \what G$ in the supremum, we obtain 
      \begin{align*}
        A_1(\what\pi^\vare) \leq 2 (1-\vare)I^{\what 0}_T(\what \pi | \what \gamma ) \leq 2 I^{\what 0}_T(\what \pi|\what \gamma) \leq 4 \I{\what \pi} +  \frac{1}{2}\norm{\Sigma(\what \pi)}{\what E}^2,
      \end{align*}from the estimation \eqref{eq:itzero} of Lemma \ref{lemma:comparison}. 
These estimates prove that $\I{\what\pi^\vare}$ is finite and $\hat\pi^\vare$ belongs to $\mathcal A_2$ for all $\vare >0$.
      
      The only thing left to show is that $\I{\what \pi^\vare}$ converges to $\I{\what \pi}$. From the lower semi-continuity of $\I{\what \cdot}$, it is enough to show that $\varlimsup_{\vare\to 0}\I{\what \pi^\vare} \leq \I{\what \pi}$. Let $\delta >0$ such that $\what\pi^\vare$ is the solution to the heat equation \eqref{Eq:heat} in $[0,\delta]$.
        From Lemma \ref{lemma:problemePerturbe}, there exists $\what H\in \H$ such that $\what \pi$ is the solution to the problem \eqref{Eq:perturbeR+Dlemma}, thus  
    \begin{align*}
      \int_0^T \scal{}{\partial_s \what \pi^\vare_s}{\what G_s} \, ds  =   - \sum_{k=1}^d \int_0^T \scal{}{\partial_{e_k}\what \pi^\vare_s}{ \partial_{e_k}\what G_s} \, ds + \frac{(1-\vare)}{2} \sum_{k=1}^d \int_0^T \scal{}{\what E_k + 2 \partial_{e_k}\what H_s}{\Sigma(\what \pi_s)\partial_{e_k}\what G_s} \, ds.
    \end{align*}
    Recall the definition of $\ell^{\what E}$ in \eqref{weakAlternative}, performing an integration by parts allows us to rewrite  
    \begin{align*}
      \ell_{\what G}^{\what E}(\what \pi^\vare\, | \, \what \gamma) =  \frac{(1-\vare)}{2}\sum_{k=1}^d  \int_0^T \scal{}{\Sigma(\what \pi_s)\left[\what E_k + 2 \partial_{e_k}\what H_s \right]}{\partial_{e_k} \what  G_s} \, ds - \frac{1}{2} \sum_{k=1}^d \int_0^T \scal{}{\Sigma(\what \pi^\vare_s) \what E_k}{\partial_{e_k}\what  G_s} \, ds ,
    \end{align*}
    and the rate functional  as  
    \begin{align*}
      \I{\what \pi^\vare} =& \sup_{\what G\in \left( \mathcal C^{1,2}_{0,-}([0,T]\times \La)\right)^2} \bigg\{ \sum_{k=1}^d \int_0^T \scal{}{\what{\mathbb{F}}_{k,s}(\what \pi^\vare)  }{\partial_{e_k}\what  G_s} \, ds  - \frac{1}{2} \sum_{k=1}^d  \int_0^T \scal{}{\partial_{e_k}\what G_s}{\Sigma(\what \pi^\vare_s)\partial_{e_k}\what  G_s} \, ds \bigg\} \\
      \leq & \frac12 \sum_{k=1}^d  \int_0^T \scal{}{\what{\mathbb{F}}_{k,s}(\what \pi^\vare)}{\Sigma(\what \pi^\vare_s)^{-1}  \what{\mathbb{F}}_{k,s}(\what \pi^\vare)} \, ds,
    \end{align*}
    where the inequality \eqref{inegaliteCarreScalaire} is used and the functional $\what{\mathbb F}_{k,s}$ is given by 
    \[
      \what{\mathbb{F}}_{k,s}(\what \pi^\vare) \coloneqq \frac{(1-\vare)}{2} \Sigma(\what\pi_s) \left[\what E_k+2\partial_{e_k}\what  H_s\right] - \frac{1}{2} \Sigma(\what \pi_s^\vare) \what E_k.
    \]
    On the other hand, from Proposition \ref{lemma:problemePerturbe}, we also have 
    \[
      \I{\what \pi}  = \frac{1}{2} \sum_{k=1}^d  \int_0^T \scal{}{\partial_{e_k} \what H_s}{\Sigma(\what \pi_s)\partial_{e_k}\what  H_s} \,ds .
    \] 
    Therefore, to conclude the proof, it suffices to show that for all $k\in \{1,\dots,d\}$, 
    \[
      \varlimsup_{\vare\to 0} \frac12  \int_0^T \scal{}{\mathbb{F}_{k,s}(\what \pi^\vare)}{\Sigma(\what \pi^\vare_s)^{-1}  \mathbb{F}_{k,s}(\what \pi^\vare)} \, ds \leq  \frac{1}{2}   \int_0^T \scal{}{\partial_{e_k}\what H_s}{\Sigma(\what \pi_s)\partial_{e_k}\what  H_s} \,ds .
    \]
    It is clear that 
    \begin{align*}
        \lim_{\vare\to 0}& \left( \what{\mathbb F}_{k,s}(\what\pi^\vare)(u), \Sigma(\what \pi^\vare_s(u))^{-1}\what{\mathbb F}_{k,s}(\what\pi^\vare)(u)\right)_{\R^2}  = \left(\partial_{e_k} \what H_s(u),\Sigma(\what\pi_s(u))\partial_{e_k} \what H_s(u)\right)_{\R^2} \quad \text{a.e. in $[0,T]\times \La$}, 
      \end{align*}
      where $(\cdot, \cdot )_{\R^2}$ stands for the usual inner product of $\R^2$. Next we need to dominate the left hand side uniformly in $\vare$ by an integrable function in order to use the dominated convergence theorem. From the triangular inequality and the following inequality $\left(a+b+c\right)^2 \leq 3\left( a^2 + b^2 +c^2 \right)$ for $a,b,c\in\R$, this term is bounded by
      \begin{align*}
        &\frac{(1-\vare)^2}{4} \left|{\sqrt{\Sigma(\what \pi^\vare_s(u))^{-1}} \Sigma(\what \pi_s(u)) \what E_k}\right|_{\R^2}^2 \\
        &+ (1-\vare)^2 \left|{\sqrt{\Sigma(\what \pi^\vare_s(u))^{-1}} \Sigma(\what \pi_s(u)) \partial_{e_k} \what H_s(u)}\right|_{\R^2}^2 + \frac{1}{4} \left|{\sqrt{\Sigma(\what\pi^\vare_s(u))}\what E_k}\right|_{\R^2}^2 ,
      \end{align*}
      where $|\cdot|_{\R^2}$ stands for the norm associated to the inner product $(\cdot,\cdot)_{\R^2}$.
      Now from the concavity and positivity of $\Sigma$ (cf. Lemma \ref{l-energy1} for more details), we have 
      \[
        (1-\vare) \Sigma(\what\pi_s(u)) \leq (1-\vare) \Sigma(\what\pi_s(u)) + \vare \Sigma(\what \rho^0_s(u)) \leq \Sigma(\what \pi^\vare_s(u)) ,
      \]
      thus, we can bound the previous expression further by 
      \begin{align*}
        \left|{\sqrt{\Sigma(\what\pi_s(u))} \what E_k}\right|_{\R^2}^2 + \left|{\sqrt{\Sigma(\what\pi_s(u))} \partial_{e_k}\what H_s(u)}\right|_{\R^2}^2 + C_\Sigma \left|{\what E_k}\right|_{\R^2}^2 .
      \end{align*}
      This completes the proof. \qedhere
    \end{proof} 

    \begin{definition}
      Denote by $\mathcal{A}_3$ the subset of trajectories $\what \pi\in \DM$ that are solutions of the boundary problem \eqref{Eq:perturbeR+Dlemma} for some $\what H \in \left( \mathcal C^{1,2}_{0,-}([0,T]\times \La)\right)^2$.
    \end{definition}
    \begin{lemma} \label{lemma:A3dense}
      The set $\mathcal{A}_3$ is $\I{\cdot}$-dense.
    \end{lemma}
    \begin{proof}
      From the previous density result, it suffices to show that for every $\what \pi\in \mathcal{A}_2$, there exists a sequence $(\what\pi^n)_{n\in\N} \in \left(\mathcal{A}_3\right)^{\N}$ such that $\what\pi^n$ converges to $\what\pi$ in $\DM$ and $\I{\what\pi^n}$ converges to $\I{\what\pi}$. For this purpose, let $\what\pi\in\mathcal{A}_2$, $\delta >0$, $\vare\in (0,1)$ such that $\what \rho|_{[0,\delta]}$ is the solution to the heat equation \eqref{Eq:heat} and $\what \pi \in \mathbf{I}_\vare$ almost everywhere in $[\delta , T] \times \La$. From Proposition \ref{lemma:problemePerturbe}, there exists $\what H^{\what\pi} \in \H$ such that $\what\pi$ is the weak solution of the boundary problem \eqref{Eq:perturbeR+Dlemma} for $\what H^{\what\pi}$, we claim that $\what H^{\what\pi}$ belongs to $L^2\left([0,T],H^1(\La)\right)$.
    Indeed, on the time interval $[0,\delta]$, $\what \pi$ solves the heat equation \eqref{Eq:heat}, by uniqueness of $\what H^{\what \pi}$, we have  $\grad \what H^{\what \pi} = - \frac{1}{2} \what  E$ a.e. in $[0,\delta] \times \La$. 
      While on the time interval $[\delta, T]$ we have 
      \begin{align*}
        \int_\delta^T \int_\La \abs{\partial_{e_k} \what H^{\what \pi,s}(u)}_{\R^2}^2 \, du \, ds \leq C_\vare \norm{\H}{\what H^{\what \pi}}^2 ,
      \end{align*}
      because $\what \pi \in \mathbf I_{\vare}$ a.e. in $[\delta,T]\times \La$, hence $\Sigma(\what \pi)^{-1} \leq C_{\vare } I_2$ a.e. in $[\delta , T] \times \La$. This implies that 
      \[
        \norm{\left( L^2([0,T],H^1(\La)\right)^2}{\what H^{\what \pi}}^2 \leq \delta \abs{\La} \sum_{k=1}^d\abs{\what E_k}_{\R^2}^2 + d C_\vare \norm{\H}{\what H^{\what \pi}}^2,
      \]
      this proves the claim.
      
    Let us consider a sequence $(\what H^n)_{n\in\N}$ of functions in $\left( \mathcal C^{1,2}_{0,-}([0,T]\times \overline\La)\right)^2$ verifying 
\[
  \what H^n \xrightarrow[n\to + \infty]{\left(L^2([0,T],H^1(\La)\right)^2)} \what H^{\what \pi} ,
\]
and define the sequence $(\what \pi^n)_{n\in\N}$, where for all $n\in\N$, $\what \pi^n$  is the weak solution of the problem \eqref{Eq:perturbeR+Dlemma} for the function $\what H^n$. It is clear that $\what \pi^n$ belongs to $\mathcal{A}_3$ for all $n\in\N$. First, we want to show that $\what \pi^n$ converges to $\what \pi$ in $\DM$. 
From Proposition \ref{lemma:problemePerturbe}, 
\[
  \I{\what \pi^n} = \frac{1}{2} \norm{1,\Sigma(\what \pi^n)}{\what  H^n}^2 \leq C_\Sigma \norm{\left(L^2([0,T], H^1(\La))\right)^2}{\what H^n} ^2,
\]
which is uniformly bounded in $n$ because $\what H^n$ converges to $\what H^{\what\pi}$ in $\left(L^2([0,T],{H}^1(\La))\right)^2$, thus, there exists a constant $C > 0$ such that  $\I{\what \pi^n} \in [0,C]$ for all $n\in\N$. But the set 
\[
  E_C\coloneqq \left\{ \tilde \pi \in \DM ~ \middle| ~  \I{\tilde \pi} \in [0,C]\right\},
\]
is a compact set in $\DM$ from Theorem \ref{thm:sciI}. Thus, from Lemma \ref{lemma:deriveetemporelle}, Lemma \ref{lemma:majorerNormeParI} and Corollary \ref{lemma:convergenceFaibleINTOforte}, there exists a subsequence $(\what \pi^{n_j})_{j\in\N}$ converging to some trajectory $\tilde \pi$ in $\DM$ and $(L^2([0,T]\times\La))^2$. Moreover by the reflexivity of $\left( L^2([0,T],H^1(\La))\right)^2$, $\partial_{e_k} \what \pi^{n_j}$ converges weakly to $\partial_{e_k} \tilde \pi$  in $\left( L^2([0,T]\times \La)\right)^2$. From this sequence we can extract a subsequence that converges almost everywhere, by abuse of notation we continue to denote it by $(\what \pi^{n_j})_{j\in\N}$.

For every $\what G\in \left( \mathcal C^{1,2}_{0,-}([0,T]\times \La)\right)^2$, with an integration by parts and because $\what \pi^{n_j} \in \mathcal B^{\what b}_{\what \gamma}$ using Proposition \ref{lemma:problemePerturbe}, we have
\begin{align*}
  \scal{}{\what \pi_T^{n_j}}{\what G_T }   =&  \scal{}{\what \gamma}{\what G_0} +   \int_{0}^T\scal{}{\what   \pi_s^{n_j}}{\partial_{t}\what G_s} ds  - \sum_{k=1}^d  \int_{0}^T \scal{}{\partial_{e_k} \what \pi_s^{n_j}}{ \partial_{e_k}\what  G_s} ds \\
& +\frac12 \sum_{k=1}^d \int_{0}^T \scal{}{\Sigma(\what \pi^{n_j}_s)\left[ \what E_k+2\partial_{e_k}\what  H^{n_j}_s\right]}{\partial_{e_k} \what  G_s} ds.
  \end{align*}
  Hence, using the convergence of $\what \pi^{n_j}$ to $\tilde \pi$ in $\DM$, the term on the left hand side converges to the same term with $\tilde \pi$ in place of $\what \pi^{n_j}$, and the same applies for the second term of the right hand side. For the third term on the right hand side, we use the weak convergence of $\partial_{e_k}\what \pi^{n_j}$ to  $\partial_{e_k}\tilde \pi$ in $(L^2([0,T],\La))^2$ mentioned a few lines before. And finally for the last term,  we use the convergence of $\what \pi^{n_j}$ to $\tilde\pi$ and $\what H^{n_j}$ to $\what H^{\what \pi}$ in $(L^2([0,T],\La))^2$, in addition to the Dominated Convergence Theorem since $\what \pi^{n_j}$ converges almost everywhere to $\tilde\pi$ and $\Sigma$ is Lipschitz continuous. All these convergences give us  
\begin{align*}
  \scal{}{\tilde\pi_T}{\what G_T }   =&  \scal{}{\what \gamma}{\what G_0} +   \int_{0}^T\scal{}{\tilde   \pi_s}{\partial_{t}\what G_s} ds  - \sum_{k=1}^d  \int_{0}^T \scal{}{\partial_{e_k} \tilde\pi_s}{ \partial_{e_k}\what  G_s} ds \\
& +\frac12 \sum_{k=1}^d \int_{0}^T \scal{}{\Sigma(\tilde \pi_s)\left[ \what E_k+2\partial_{e_k}\what  H^{\what \pi}_s\right]}{\partial_{e_k} \what  G_s} ds.
  \end{align*}
    Using Proposition \ref{lemma:problemePerturbe}, knowing that $\I{\tilde \pi}<+ \infty$ because $\tilde \pi \in E_C$, we have $\tilde \pi \in \mathcal{B}_{\what \ga}^{\what b}$. Therefore, applying an integration by parts we get

\begin{align*}
  \scal{}{\tilde \pi_T}{\what G_T} = &\scal{}{\what \gamma}{\what G_0} + \int_0^T \scal{}{\tilde \pi_s}{\partial_t \what G_s} \, ds + \int_0^T \scal{}{\tilde \pi_s}{\Delta \what G_s} \, ds  \\
  & + \frac{1}{2} \sum_{k=1}^d \int_0^T \scal{}{\Sigma(\tilde \pi_s)\left[ \what E_k + 2 \partial_{e_k}\what H^{\what \pi}_s\right]}{\partial_{e_k} \what G_s} \, ds \\
  & -  \int_0^T \int_{\Ga^+} \scal{}{\tilde \pi_s(r)}{\partial_{e_1}\what G_s(r)} \, dS(r) \, ds  + \int_0^T \int_{\Ga^-} \scal{}{\what b(r)}{\partial_{e_1}\what G_s(r)} \, dS(r) \, ds .
\end{align*}
Thus, $\tilde \pi$ is the weak solution of \eqref{Eq:perturbeR+Dlemma} for the function $\what H^{\what \pi}$. By uniqueness of such a solution stated in Appendix \ref{Appendix:uniquenessPerturbe}, we have $\tilde \pi = \what 
\pi$ and by uniqueness of subsequential limit in a compact set, $\what \pi^n$ converges to $\what \pi$ in $\DM$ and $\left(L^2([0,T]\times \La)\right)^2$. Finally, from the dominated convergence theorem 
\begin{align*}
    \lim_{j\to+\infty} \I{\what\pi^{n_j}} =& \lim_{j\to + \infty} \scall{}{\grad \what H^{n_j}}{\Sigma(\what \pi^{n_j}) \grad \what H^{n_j}} \\
    = & \frac{1}{2} \scall{}{\grad \what H}{\Sigma(\what \pi) \grad \what H} = \I{\what \pi} ,
\end{align*}
because $ \grad \what H^{n_j} \xrightarrow[]{j\to + \infty} \grad \what H $ in $\left( L^2([0,T],H^1(\La)\right)^2$ and $\what \pi^{n_j} \xrightarrow[]{j\to +\infty} \what \pi$ a.e. in $[0,T]\times \La$.\qedhere
\end{proof}

\subsection{Proof of lower bound}

We need to show that for any open set $\mathcal{O}$ of $\DM$, 
    \[
      \varliminf_{N\to + \infty} \frac{1}{N^d}\log \mathbb Q_{\delta_{\sigma^N}}^{N,\what b}(\mathcal O) \geq - \inf_{\what \pi \in \mathcal{O}} \I{\what \pi},
    \]
    which is equivalent to show that for any open set $\mathcal{O}$ of $\DM$ and any $\what \pi \in \mathcal{O}$ 
    \begin{equation}\label{lower1}
      \varliminf_{N\to + \infty} \frac{1}{N^d}\log \mathbb Q_{\delta_{\sigma^N}}^{N,\what b}(\mathcal O) \geq -  \I{\what \pi}.
    \end{equation}
    To this end, let us consider $\mathcal{O}$ an open set of $\DM$ and let $\what \pi \in \mathcal{O}$. Then, if $\I{\what \pi} = + \infty$, the inequality is clear. Else, $\I{\what\pi} < + \infty$ and by the $\I{\cdot}$-density of $\mathcal{A}_3$ stated in Lemma \ref{lemma:A3dense}, there exists a sequence $(\what \pi^{\what H_n})_{n\in\N} \in (\mathcal{A}_3)^{\mathbb N}$ such that 
    \begin{align*}
      \what \pi^{\what H_n} \; &\xrightarrow[n \to + \infty]{\DM}  \;\what \pi , \\
      \I{\what \pi^{\what H_n}} \; & \xrightarrow[n\to + \infty]{\hphantom{\DM}} \; \I{\what \pi}.
    \end{align*}   
    Using classical arguments (cf. Theorem 10.5.4 in \cite{kl}),  to prove \eqref{lower1}, it is enough to show 
    \begin{equation}\label{lower2}
       \varlimsup_{N \to + \infty} \frac{1}{N^d} H\left( \mathbb P^{N,\what H_n}_{\delta_{\sigma^N}} \middle| \mathbb P^{N,\what b}_{\delta_{\sigma^N}} \right) = \I{\what \pi^{\what H_n}}, 
    \end{equation}
    where $H\left(  \mathbb P^{N,\what H_n}_{\delta_{\sigma^N}} \middle|  \mathbb P^{N,\what b }_{\delta_{\sigma^N}} \right)$ stands for the relative entropy given by 
    \[
        H\left(  \mathbb P^{N,\what H_n}_{\delta_{\sigma^N}} \middle|  \mathbb P^{N,\what b }_{\delta_{\sigma^N}} \right) = \int \log \left( \frac{d  \mathbb P^{N,\what H_n}_{\delta_{\sigma^N}}}{d   \mathbb P^{N,\what b }_{\delta_{\sigma^N}}}\right) \, d   \mathbb P^{N,\what H_n}_{\delta_{\sigma^N}} .
    \]
    Using the explicit formula \eqref{MartExp} of the exponential martingale, the replacement lemmata \ref{replacementBulk}, \ref{replacementGauche}, \ref{replacementDroite} and its expression as a function of the empirical density \eqref{approxExpMart}, for all $\delta >0$ and $\vare >0$,
    \begin{align*}
      \frac{1}{N^d} H\left(  \mathbb P^{N,\what H_n}_{\delta_{\sigma^N}} \middle|  \mathbb P^{N,\what b }_{\delta_{\sigma^N}} \right) =& \mathbb E^{ \mathbb P^{N,\what H_n}_{\delta_{\sigma^N}}}\left[ J_{\what H_n}^{\what E}(\what \pi^{\vare N}(\sigma)) \right] + o_{N}(1) + O_{\what H_n}(\vare) + O(\delta) .
    \end{align*}
    Taking the limit $N\to + \infty$, using the convergence of $\mathbb Q^{N,\what H_N}_{\delta_{\sigma^N}}$ to $\delta_{\what \pi^{\what H_n}}$ by Proposition \ref{th-hy-perturbe}, and taking $\vare \to 0$ and $\delta \to 0$, we get 
    that 
    \[
      \varlimsup_{N\to + \infty} \frac{1}{N^d} H\left(  \mathbb P^{N,\what H_n}_{\delta_{\sigma^N}} \middle| \mathbb P^{N,\what b }_{\delta_{\sigma^N}} \right) = J_{\what H_n}^{\what E}(\what \pi^{\what H_n}).
    \]
    Taking into account that $\what \pi^{\what H_n}$ is a weak solution to the problem \eqref{Eq:perturbeR+D-hydrodynamic}, we obtain 
    \[
      \varliminf_{N\to + \infty} \frac{1}{N^d}\log \mathbb Q_{\delta_{\sigma^N}}^{N,\what b}(\mathcal O) \geq - \I{\what \pi^{\what H_n}}.
    \]
    To conclude the proof of the lower bound, it remains to let $n\to + \infty$.

    \subsection{Upper bound}
    For $a\geq 0 $, denote by $\Ea:\DM \longrightarrow [0,+\infty]$  the functional defined by 
    \[
      \Ea(\what \pi) \coloneqq \left\{ \begin{array}{ll}
        \Itilde{\what\pi} + a(1+a) \tilde{\mathcal Q}(\what \pi),\qquad &\text{if $\what \pi \in \mathcal{D}([0,T],\what{\mathcal{M}}_0)$,} \\
        + \infty , \qquad & \text{else,}
      \end{array} \right.
    \] 
    where  
    \[
      \tilde{\mathcal{Q}}(\what\pi) \coloneqq \sup_{k = 1,\dots, d} \sup_{\what H \in \left( \mathcal C^\infty_c([0,T]\times \La)\right)^2}  \tilde{\mathcal{Q}}_{\what H,k} (\what \pi),
    \]
    and for $\what H \in \left( \mathcal C^\infty_c([0,T]\times \La)\right)^2$ and $k\in \{1,\dots,d\}$,
    \[
      \tilde{\mathcal{Q}}_{\what H,k} (\what \pi) \coloneqq \scall{}{\what \pi}{\partial_{e_k}  \what H} - 2 \scall{}{\what H}{\Sigma(\what \pi) \what H} .
    \]
Notice that, for all $\what\pi \in \DM$,
\[
\Ea(\what \pi) = \Itilde{\what\pi} + a(1+a) {\mathcal E}(\what \pi),
\]
where the functional ${\mathcal E}$ is defined in \eqref{1:Q}.

The proof of the upper bound relies on the following proposition.
\begin{proposition}\label{prop:upperBound}
  Let $\mathcal{K}$ be a compact set of $\DM$. There exists a positive constant $C$ such that for any $a\in (0,1]$, and  any sequence $(\sigma^N)_{N \geq 1}$ of configurations such that $(\delta_{\sigma^N})_{N \geq 1}$ is associated to $\what \ga$, in the sense of \eqref{pfl}, 
  \[
    \varlimsup_{N\to + \infty} \frac{1}{N^d} \log \mathbb Q_{\delta_{\sigma^N}}^{N,\what b}(\mathcal{K}) \leq - \frac{1}{1+a} \inf_{\what\pi \in \mathcal{K}} \Ea(\what \pi) + a C (T+1).
  \]
\end{proposition}

\begin{proof}
  Fix two functions $\what  G \in \left( \mathcal C^{1,2}_{0,-}([0,T] \times \La)\right)^2$, $\what H \in \left( \mathcal C^\infty_c([0,T]\times \La)\right)^2$ and positive functions $J_1,J_2 \in  \mathcal C^\infty(\La) $ .  
  For a local function ${\mathfrak{h}}:\{0,1\}^{\Z} \longrightarrow \R$,  $i\in \{1,2\}$, $c,\vare >0$, and a function $G \in  \mathcal C^{1,2}_{0,-}([0,T] \times \La)$ ,  let $\mathcal R^{G,H}_{N,\vare,c}$ be the set 
  \[
    \mathcal R^{\what G,\what H}_{N,\vare,c} \coloneqq \bigcap_{k=1}^d \Big(  \mathcal{R}^{\partial_{e_k} \what H}_{N,\vare,c} \cap \mathcal{R}^{\what G}_{N,\vare,c}\Big),
  \]
  where the set $\mathcal{R}^{\cdot}_{N,\vare,\delta}$ are defined in  \eqref{definitionR}. By the superexponential estimates stated in Propositions \ref{replacementBulk}, \ref{replacementGauche} and \ref{replacementDroite}, and the inequality \eqref{inegaliteLOG},  it is enough to prove that, for every $0<a \leq 1$ 
  \[
    \varlimsup_{c\to 0} \varlimsup_{\vare \to 0} \varlimsup_{N \to + \infty} \frac{1}{N^d} \log \mathbb Q_{\delta_{\sigma^N}}^{N,\what b}(\mathcal{K} \cap \mathcal R^{\what G,\what H}_{N,\vare,c}  ) \leq -\frac{1}{1+a} \inf_{\what \pi \in \mathcal{K}} \Ea(\what \pi) + a C(T+1). 
  \]
  Let us define the functional $A^{\what H,\what J}_{\vare}:\DM \longrightarrow \R$ by 
  \[
    A^{\what H,\what J}_{\vare}(\what \pi) \coloneqq \sup_{1\leq k \leq d}\tilde{\mathcal{Q}}_{\what H,k}(\what \pi * \what u_\vare)  + \mathfrak M_{\what J}(\what\pi) ,
  \] 
  where $\mathfrak M$ is defined in \eqref{definitionM} (cf. Appendix \ref{Miscellaneous})
    and $\what J \coloneqq (J_1,J_2)$. From Hölder's inequality with $p=1+a$ and $q=\frac{1+a}{a} < +\infty$ because $a\neq 0$ , we have  
  \begin{align*}
    \frac{1}{N^d} \log \mathbb Q_{\delta_{\sigma^N}}^{N,\what b}(\mathcal{K} \cap \mathcal R^{\what G,\what H}_{N,\vare,c} ) \leq& \frac{1}{(1+a)N^d} \log \mathbb E_{\delta_{\sigma^N}}^{N,\what b}\left[ \1_{\mathcal{K}\cap \mathcal R^{\what G,\what H}_{N,\vare,c}} e^{-a(1+a)N^d A^{\what H,\what J}_{\vare}(\what \pi)} \right] \\
    &+ \frac{a}{(1+a)N^d} \log \mathbb E_{\delta_{\sigma^N}}^{N,\what b}\left[ e^{(1+a)N^d A^{\what H,\what J}_{\vare}(\what \pi)} \right].
  \end{align*}
  From its definition, $\mathfrak M_{\what J}$ is uniformly bounded by $\frac{C_{\what J}}{N}$ for some positive constant $C_{\what J}$. Therefore using Lemma \ref{lemma:energyEstimate} and inequality \eqref{inegaliteLOG},  we get
  \[
    \varlimsup_{\vare \to 0} \varlimsup_{N\to + \infty} \frac{a}{(1+a)N^d} \log \mathbb E_{\delta_{\sigma^N}}^{N,\what b}\left[ e^{(1+a)N^d A^{\what H,\what J}_{\vare}(\what \pi)} \right] \leq a C_1 (T+1),
  \]
  because $(1+a)\leq 2 = \delta_0$ and  $1\leq 1+a$. 
  
  It now remains to show that 
  \[
    \varlimsup_{c\to 0 }\varlimsup_{\vare \to 0} \varlimsup_{N\to + \infty} \frac{1}{(1+a)N^d} \log \mathbb E_{\delta_{\sigma^N}}^{N,\what b}\left[ \1_{\mathcal{K}\cap \mathcal R^{\what G,\what H}_{N,\vare,c}} e^{-a(1+a)N^d A^{\what H,\what J}_{\vare}(\what \pi)} \right]  \leq - \frac{1}{1+a}\inf_{\what \pi \in \mathcal K} \Ea(\what \pi) .
  \]
   Considering the expressions of the mean one exponential martingale in \eqref{MartExp}, \eqref{approxExpMart} and applying in our setting the argument, by now standard, done in Chapter 10, Section 4 of \cite{kl}, we get  
  \begin{align*}
    &\frac{1}{(1+a)N^d} \log \mathbb E_{\delta_{\sigma^N}}^{N,\what b}\left[ \1_{\mathcal{K}\cap \mathcal R^{\what G,\what H}_{N,\vare,c}} e^{-a(1+a)N^d A^{\what H,\what J}_{\vare}(\what \pi)} \right] \\
    \leq & \frac{1}{(1+a)} \sup_{\what \pi \in \mathcal{K}}\left( -  \left[ J_{\what G}^{\what E} (\what \pi^\vare) + a(1+a)A^{\what H,\what J}_{\vare}(\what \pi)\right]  \right) + O_{\what G}\left(\frac{1}{N}\right) + O_{\what G}(\vare) + O_{\what J}\left(\frac{1}{N}\right) + O_{\what J}(\vare)  + O(c).
  \end{align*}
  Letting $N\to + \infty$,  $\vare \to 0$ and  $c\to 0$, we get that 
  \begin{align*}
    &\varlimsup_{c\to 0 }\varlimsup_{\vare \to 0}\varlimsup_{N\to + \infty}\frac{1}{(1+a)N^d} \log \mathbb E_{\delta_{\sigma^N}}^{N,\what b}\left[ \1_{\mathcal{K}\cap \mathcal R^{\what G,\what H}_{N,\vare,c}} e^{-a(1+a)N^d  A^{\what H,\what J}_{\vare}(\what \pi)} \right] \\
    \leq & \frac{1}{(1+a)} \varlimsup_{\vare \to 0} \sup_{\what \pi \in \mathcal{K}} \left( -  \left[ J_{\what G}^{\what E} (\what \pi^\vare) + a(1+a)  A^{\what H,\what J}_{\vare}(\what \pi)\right]  \right) \\
    \leq & \frac{1}{(1+a)}  \sup_{\what \pi \in \mathcal{K}} \left( - J_{\what G}^{\what E} (\what \pi) - a(1+a)\varliminf_{\vare \to 0} A^{\what H,\what J}_{\vare}(\what \pi)  \right).
  \end{align*}
    In the last inequality, we used the compactness of $\mathcal{K}$, the lower semicontinuity of $J^{\what E}_{\what G} (\cdot)$ and  $A^{\what H,\what J}_{\vare}(\cdot)$ for every $\what G$, $\what H$, $\what J$ and $\vare$ and Minmax lemma (cf. Lemma 3.2, Appendix 2 in \cite{kl}). Optimizing over $\what G, \what H, \what J$, and using again Minmax lemma, the last term is further bounded by
  \begin{align*}
    &  - \frac{1}{(1+a)} \inf_{\what \pi \in \mathcal{K}} \left(   \Itilde{\what \pi} + a(1+a) \left[ \sup_{\what H} \varliminf_{\vare \to 0} \sup_{k = 1,\dots, d} \tilde{\mathcal{Q}}_{\what H,k}(\what \pi * \what u_\vare) + \sup_{\what J} \mathfrak M_{\what J}(\what \pi)\right]    \right).  \end{align*}
  Furthermore, $\sup_{\what J} \mathfrak M_{\what J}(\what\pi)$ is finite if and only if $\what \pi \in \mathcal D([0,T],\what{\mathcal{M}}_0)$ (see Appendix \ref{supMfiniD0TM}). Thus, if $\what \pi \notin \mathcal{D}([0,T],\widehat{\mathcal{M}}_0)$
    \begin{align*}
      &\Itilde{\what \pi} + a(1+a) \left[ \sup_{\what H} \varliminf_{\vare \to 0} \sup_{k = 1,\dots, d}\tilde{\mathcal{Q}}_{\what H,k}(\what \pi * \what u_\vare) + \sup_{\what J} \mathfrak{M}_{\what J}(\what \pi)\right]  = + \infty = \Ea (\what \pi).
    \end{align*}
      Otherwise, for $\what \pi \in \mathcal{D}([0,T],\widehat{\mathcal{M}}_0)$, $\sup_{\what J} \mathfrak{M}_{\what J}(\what \pi) =0$ from Lemma \ref{supMfiniD0TM}. Moreover, 
    from the lower semi-continuity of $\sup_{k = 1,\dots, d} \tilde{\mathcal{Q}}_{\what H,k}(\cdot)$, for $\what \pi \in  \mathcal{D}([0,T],\widehat{\mathcal{M}}_0)$, we get
     
    \begin{align*}
      \Itilde{\what \pi} + a(1+a) \left[ \sup_{\what H} \varliminf_{\vare \to 0} \sup_{k = 1,\dots, d} \tilde{\mathcal{Q}}_{\what H,k}(\what \pi * \what u_\vare) + \sup_{\what J} \mathfrak{M}_{\what J}(\what \pi)\right] 
      & \geq \Ea (\what \pi) .
    \end{align*}
    This concludes the proof. \qedhere
\end{proof}

Let us define  for $\what \pi \in \DM$ and $a\geq 0$,  the functionals
\begin{align*}
    \mathcal{E}_{I_2}(\what \pi) \coloneqq & \left\{ \begin{array}{ll}
        \mathcal{Q}_{I_2}(\what \pi), & \qquad \text{if $\what \pi \in \mathcal{D}([0,T], \wm)$,}  \\
        + \infty,  & \qquad \text{else,} 
    \end{array}    \right. \\
    \Ea^{I_2}(\what \pi) \coloneqq & \left\{ \begin{array}{ll}
        \Itilde{\what\pi} + a(1+a) \mathcal{E}_{I_2}(\what \pi) ,\qquad &\text{if $\what \pi \in \mathcal{D}([0,T],\what{\mathcal{M}}_0)$,} \\
        + \infty , \qquad & \text{else.}
      \end{array} \right.
\end{align*}
\begin{corollary}\label{coro:upperBound}
  Let $\mathcal{K}$ be a compact set of $\DM$. There exists a positive constant $C \geq 1$ depending on $\Sigma$ such that for any $a\in (0,1]$, and  any sequence $(\sigma^N)_{N \geq 1}$ of configurations such that $(\delta_{\sigma^N})_{N \geq 1}$ is associated to $\what \ga$ in the sense of \eqref{pfl},  
  \[
    \varlimsup_{N\to + \infty} \frac{1}{N^d} \log \mathbb Q_{\delta_{\sigma^N}}^{N,\what b}(\mathcal{K}) \leq - \frac{1}{1+a} \inf_{\what\pi \in \mathcal{K}} \EaC^{I_2}(\what \pi) + a C (T+1).
  \]
\end{corollary}
\begin{proof}
    Without loss of generality, take $C_{\Sigma}\geq 1$ such that $\Sigma(\what \pi) \leq C_\Sigma I_2$ for all $\what \pi \in \mathcal{D}([0,T],\wm)$, thus for $a\geq 0$ and $\what \pi \in \mathcal{D}([0,T],\wm)$
\[
    a(1+a)\Q_{I_2}(\what \pi) \leq  aC_{\Sigma}(1+aC_{\Sigma})  \Q(\what \pi).
\]
Therefore, for all $\what \pi \in \DM$ and $a\geq 0$, $\Ea(\what \pi) \geq \mathfrak E_{aC_{\Sigma}^{-1}}^{I_2}(\what \pi)$ giving us the result using Proposition \ref{prop:upperBound} and taking the $\max$ between $C$ and $C_\Sigma$.
\end{proof}

\begin{proof}[Proof of the upper bound for compact sets] Let $\mathcal{K}$ be a compact set of $\DM$. First, if 
\linebreak\mbox{$\inf_{\what \pi \in \mathcal{K}} {\mathfrak E}_{C^{-1}}^{I_2} (\what \pi) = + \infty$}  then, for all $\what \pi \in \mathcal{K}$, $\Itilde{\what \pi} = + \infty$ or $\mathcal{E}_{I_2}(\what \pi) = + \infty$ and the upper bound is satisfied using Corollary \ref{coro:upperBound} because $\inf_{\what \pi \in \mathcal{K}} \I{\what \pi}= + \infty$. Suppose now that $\inf_{\what \pi \in \mathcal{K} }  {\mathfrak E}_{C^{-1}}^{I_2} (\what \pi) < + \infty$, then there exists $\what \pi \in \mathcal{K}$
such that $\Itilde{\what \pi} =\I{\what \pi}< +\infty$ and $\mathcal{E}_{I_2}(\what \pi)< + \infty$, in this case $\inf_{\what \pi \in \mathcal{K}} {\mathcal E}_{I_2}(\what \pi) < +\infty$ and for $a\in (0,1]$
    \[
      \inf_{\what \pi \in \mathcal{K}}\EaC^{I_2}(\what \pi) = \inf_{ \substack{\what \pi \in \mathcal{K} \\ {\mathfrak E}_{C^{-1}}^{I_2} (\what \pi) < + \infty}} \left[\I{\what \pi} + a C^{-1}(1+aC^{-1}) {\mathcal{E}_{I_2}}(\what \pi)\right].
    \]
In view of  Corollary  \ref{coro:upperBound}, we get
      \begin{align*}
        & \varlimsup_{N\to + \infty} \frac{1}{N^d} \log \mathbb Q_{\delta_{\sigma^N}}^{N,\what b}(\mathcal K) 
         \\
         \leq&  - \frac{1}{(1+a)} \inf_{ \substack{\what \pi \in \mathcal{K} \\ {\mathfrak E}_{C^{-1}}^{I_2} (\what \pi) < + \infty}}    \I{\what \pi} - \frac{aC^{-1}(1+aC^{-1})}{1+a} \inf_{ \substack{\what \pi \in \mathcal{K} \\ {\mathfrak E}_{C^{-1}}^{I_2} (\what \pi) < + \infty}} {\mathcal{E}_{^{I_2}}}(\what \pi) + a C(T+1).
      \end{align*}
 Letting $a\to 0 $, we finally get,
\[
 \varlimsup_{N\to + \infty} \frac{1}{N^d} \log \mathbb Q_{\delta_{\sigma^N}}^{N,\what b}(\mathcal K) \le - \inf_{ \substack{\what \pi \in \mathcal{K} \\ {\mathfrak E}_{C^{-1}}^{I_2} (\what \pi) < + \infty}}     \I{\what \pi}
 \le - \inf_{\what \pi \in \mathcal{K}}   \I{\what \pi},
\]
that concludes the proof of the upper bound for compact sets.\qedhere 
\end{proof}    

To pass from compact sets to closed sets, we use the exponential tightness for the sequence $(\mathbb Q_{\delta_{\sigma^N}}^{N,\what H})_{N\geq 1}$, which is now standard (cf. Chapter 10, Section 4 of \cite{kl}).

    \appendix

    \section{Basic tools}
    \subsection{Basic inequalities}
    For any $a,b\in\R$ and $\al>0$, 
    \begin{equation}
      (a+b)^2 \leq 2a^2 + 2b^2, \qquad ab \leq \frac{\al }{2}a^2 + \frac{1}{2\al}b ^2. \label{BasicSommeCarre}
    \end{equation}
  For any $A,B\in\R^2$ and $S\in S_2^{++}(\R)$, 
  \begin{equation}
    \scal{}{A}{B} \leq \frac{\norm{2}{S A}^2}{2} + \frac{\norm{2}{S^{-1}B}^2}{2} \qquad \text{and} \qquad \scal{}{A}{B} \leq \frac{\scal{}{A}{S A}}{2} + \frac{\scal{}{B}{S^{-1} B}}{2} ,\label{inegaliteCarreScalaire}
  \end{equation}
  moreover, there exists a unique matrix $\sqrt{S}$ in $S_2^{++}(\R)$ such that $S = \left(\sqrt{S}\right)^2$.
  
  Indeed, the first inequality can be proven easily by using the Cauchy-Schwarz inequality and using \eqref{BasicSommeCarre} with $\al=1$ after noticing that $\scal{}{A}{B} = \scal{}{SA}{S^{-1}B}$ for any $S\in S_2^{++}(\mathbb R)$. The second inequality is just another way to write the first one with $\sqrt{S}$ in place of $S$. 
    
    For any sequences $(a_N)_{N\in\N},(b_N)_{N\in\N}$ of positive real numbers,
    \begin{equation}
      \lim_{N\to + \infty} \frac{1}{N^d} \log(a_N + b_N) \leq \max\left( \varlimsup_{N\to + \infty} \frac{1}{N^d} \log(a_N), \varlimsup_{N\to + \infty} \frac{1}{N^d} \log(b_N) \right). \label{inegaliteLOG}    
    \end{equation}
  
\subsection{Sobolev spaces}\label{sobolev}
\subsubsection{Classical Sobolev spaces and Trace operator}

For $0\leq l \leq + \infty$, let us remind that  $\mathcal C^{l}_{0,-} (\overline \Lambda)$ stands for
the subspace  of functions in $\mathcal C^{l} (\overline \Lambda)$ vanishing at the left boundary $\Gamma^-$ of $\L$.
Denote by  ${\ch}_{0,-}(\Lambda)$ the closure of $\mathcal C^{l}_{0,-} (\overline \Lambda)$ in $H^1(\L)$. This is the subspace of $H^1(\L)$
consisting of all functions with zero trace on $\Gamma^-$:
\begin{equation}
 {\ch}_{0,-}(\Lambda)\, \coloneqq\,\Big\{F\in  H^1(\L)  \; :\; \text{Tr}(F) = 0 \ \text{on} \ \Gamma^-\Big\}.
\end{equation}
A simple computation shows that, for any $F\in {\ch}_{0,-}(\Lambda)$,
\begin{equation}\label{equiv-norm}
 \|F\|_{L^2(\La)}\le \sqrt{2}\,  \|\partial_1 F\|_{L^2(\La)}\, \le\, \sqrt{2}\,  \|\nabla F\|_{L^2(\La)}:=\sqrt{2}\sqrt{\sum_{k=1}^d \|\partial_{e_k} F\|_{L^2(\La)}^2}\, 
\end{equation} 
and therefore, $\|F\|_{{\ch}_{0,-} (\Lambda)}\coloneqq  \|\nabla F\|_{L^2(\La)}$ defines a norm on ${\ch}_{0,-}(\Lambda)$, which is equivalent to that of $H^1(\L)$. Moreover, ${\ch}_{0,-}(\Lambda)$ is a Hilbert space for the inner product
$$
\big\< F,G{\big\>}_{{\ch}_{0,-}(\Lambda)} =\sum_{k=1}^d \big\< \partial_{e_k} F,\partial_{e_k}G\big\>.
$$
  
Denote by ${\ch}_{0,-}^*(\Lambda)$ the dual of ${\ch}_{0,-}(\Lambda)$, which is a Hilbert
space equipped with its  norm that we denote by $\|{\cdot}\|_{0,-,*} $. Since $\mathcal C^{1}_{0,-} (\overline \Lambda)$ is dense in ${\ch}_{0,-}(\Lambda)$, 
\begin{equation}
\label{eq:norm-1}
\|  H \|^2_{0,-,*} \;=\; \sup_{ G \in \mathcal C^{1}_{0,-} (\overline \Lambda)}
\Big\{ 2\<H ,G \>- \int_{\Lambda}  \nabla G (u) \cdot \grad   G (u)   du \Big \} .
\end{equation}
In this formula, $\<H,G\>\coloneqq \<H,G\>_{{\ch}_{0,-}^*,{\ch}_{0,-}}$ stands for the
value of the linear form $H$ at $G$.

\begin{remark} {\ }\\
    \begin{itemize}
     \item[(i)] $\mathcal{H}_{0,-}$ is
        continuously embedded in  $L^2(\Lambda)$, that is the mapping $i_{\mathcal{H}_{0,-}}: H\in\mathcal{H}_{0,-} \mapsto H\in L^2(\Lambda)$ is continuous. 
        \item[(ii)] $\mathcal{H}_{0,-}$ is compactly embedded
        in  $L^2(\Lambda)$, that is,  any bounded sequence of ${\mathcal{H}_{0,-}}$ has a subsequence which converges in $L^2(\Lambda)$, (cf. Theorem 1.21 in  \cite{Roubi}).
        \item[(iii)] $L^2(\Lambda)$ is continuously embedded in  $\mathcal{H}_{0,-}^*$. Moreover, for $H\in L^2(\Lambda)$
        \begin{equation}
            \scal{}{H}{G} \coloneqq \<H,G\>_{{\ch}_{0,-}^*,{\ch}_{0,-}}=\int_\Lambda H(u)G(u) du ,\quad \text {for any} \; G\in \mathcal{H}_{0,-}, \label{crochetDual}
        \end{equation}
        where the indices in the scalar product indicate the spaces paired by the duality.
    \end{itemize}
\end{remark}

 \subsubsection{Weighted Sobolev spaces}\label{wss}
We denote for a bounded positive semi-definite matrix  function $S : [0,T]\times \La \longrightarrow S^+_2$ where we recall that $S^+_2$ has been defined in Section \ref{SectionHydro}, the semi-inner product defined on $(L^2([0,T]\times\Lambda))^2$,
\begin{equation}\label{inner-weit0}
   \scall{S}{\what H}{\what G} \coloneqq \int_0^T  \scal{}{\what H_t}{S_t \what G_t}  \, dt,
\end{equation}
where $\scal{}{\cdot}{\cdot}$ is the inner product defined in \eqref{produitScalaireChapo}.
Denote by $\LL^2(S)$  the Hilbert space induced by the set
$(L^2([0,T]\times\Lambda))^2$
endowed with semi-inner product \eqref{inner-weit0}. “Induced” means that we first declare two functions $\what G , \what H\in (L^2([0,T]\times\Lambda))^2$ to be equivalent if $\scall{S}{\what H -\what G}{\what H-\what G}=0$ and then we complete the quotient space with
respect to the inner product. Denote by $\norm{S}{\cdot}$ the norm associated to the inner product $\scall{S}{\cdot}{\cdot}$. For the matrix identity $S=I_2$, this inner product is already defined in \eqref{doubleProduitScall}: $\scall{I_2}{\cdot}{\cdot} = \scall{}{\cdot}{\cdot}$. For simplicity, when $S=I_2$, we denote $\LL^2\coloneqq \LL^2(I_2)$.\\
Notice that, since $S$ is bounded, $( \mathcal C^{1,2}_{c}((0,T)\times \La))^2$ and $( \mathcal C^{1,2}_{0,-}([0,T]\times \La))^2$ are dense in $\LL^2(S)$. 

Similarly, we consider the semi-inner product defined on $\left(L^2([0,T],{\ch}_{0,-}(\Lambda))\right)^2$ by
\begin{equation}\label{norme1S}
    \scall{1,S}{\what G}{\what  H} \coloneqq   \sum_{k=1}^d \scall{S}{\partial_{e_k} \what  G}{\partial_{e_k} \what H},
\end{equation}
and denote by ${\mathbb H}_{0,-}(S)$ the Hilbert space induced by the set
$\left(L^2([0,T],{\ch}_{0,-}(\Lambda))\right)^2$
endowed with the inner product $\scall{1,S}{\cdot}{\cdot}$.  We denote by $\norm{1,S}{\cdot}$ the associated norm. 

Let ${\mathbb H}_{0,-}^*(S)$ be the dual of ${\mathbb H}_{0,-}(S)$, which is
a Hilbert space equipped with its norm $\Vert \cdot \Vert_{1,S,*}$. Notice that, since $S$ is bounded, $( \mathcal C^{1,2}_{0,-}([0,T]\times \overline \La))^2$ is dense in ${\mathbb H}_{0,-}(S)$ and
\begin{equation*}
\|\what{L}\|^2_{1,S,*} \;=\; \sup_{\what{G}\in \left(  \mathcal C_{0,-}^{1,2}([0,T]\times \overline{\La})\right)^2}
\Big\{ 2 \<\!\<\what{L}, \what{G} \>\!\> \;-\;
\Vert \what{G} \Vert^2_{1, S} \Big \}\;.
\end{equation*}
In this formula, $\<\!\<\what{L}, \what{G} \>\!\>$ stands for the value of the
linear form $\what{L}$ at $\what{G}$. When $S=I_2$, we denote simply ${\mathbb H}_{0,-}(S)$ and ${\mathbb H}_{0,-}^*(S)$ respectively by ${\mathbb H}_{0,-}$ and ${\mathbb H}_{0,-}^*$.

Denote by $\W_{[0,T]}^{1,2}\coloneqq\W^{1,2}([0,T]; H^1(\La),\mathcal H_{0,-}^*)$, the Sobolev space of functions $H\in H^1(\La)$ with time derivative $\partial_t H\in {\mathbb H}_{0,-}^*$:
\begin{equation}\label{W1,2}
   \W^{1,2}_{[0,T]} =\left\{ H\in L^2([0,T]; H^1(\La))\quad \text{such that}\quad   \partial_t H\in L^2([0,T]; \mathcal H_{0,-}^*)\right\}. 
\end{equation}
Similarly, we denote $\W_{[0,T],0,-}^{1,2}\coloneqq\W^{1,2}([0,T]; \mathcal H_{0,-},\mathcal H_{0,-}^*)$. 
\begin{proposition}\label{compact-embeding}
    ${\W}^{1,2}_{[0,T]}$ is compactly embedded
        in  $L^2([0,T],L^2(\La))$. 
\end{proposition}
\begin{proof} The proof follows from 
Aubin and Lions Lemma (see Lemma 7.7 in \cite{Roubi}),
    since $H^1(\La)\subset L^2(\L)\subset \mathcal H_{0,-}^*$ with continuous embedding $L^2(\L)\to \mathcal H_{0,-}^*$ and
compact embedding $H^1(\La)\to L^2(\L)$ (cf. Theorem 1.21 in \cite{Roubi}).
\end{proof}

The following result is a direct consequence of this Proposition.
\begin{corollary}\label{lemma:convergenceFaibleINTOforte}
      Let $(\what \pi^n)_{n\in\N}$ be a sequence of functions in $\left(L^2([0,T]\times \La)\right)^2$ converging weakly in $\left(L^2([0,T]\times \La)\right)^2$ to $\what \pi$. If there exists a constant $C_0$ such that for all $n\in\N$,
      \[
  \norm{\left(L^2([0,T],H^1(\La))\right)^2}{\what \pi^n}^2 + \norm{1,*}{\partial_t \what \pi^n}^2  \leq C_0,
\] 
      then $\what \pi^n$ converges strongly to $\what \pi$ in $\left(L^2([0,T]\times \La)\right)^2$.
\end{corollary}  

\begin{lemma}\label{lemma:inegrationByPartsRoubicek}
  $\W_{[0,T],0,-}^{1,2}$ is continuously embedded in $ \mathcal C([0,T],L^2(\La))$ and for all $H,G\in \W_{[0,T],0,-}^{1,2}$ and $0\leq t_1\leq t_2 \leq T$, 
  \[
    \scal{}{ H_{t_2}}{G_{t_2}} - \scal{}{ H_{t_1}}{G_{t_1}} =\int_{t_1}^{t_2}\scal{}{\partial_s H_s}{ G_s}\, ds + \int_{t_1}^{t_2}\scal{}{H_s}{\partial_s G_s}\, ds ,
  \]
  where $\scal{}{\cdot}{\cdot}$ in the time integrals stands for the duality pairing defined in \eqref{crochetDual}.
\end{lemma}
\begin{proof}
  The proof follows from Lemma 7.3 of \cite{Roubi} because $\mathcal{H}_{0,-} \subset L^2(\La) \subset \mathcal{H}_{0,-}^*$.
\end{proof}

    \begin{remark}
    With the notations above, for all $\what \pi\in \mathcal{D}([0,T], \wm) $ such that  $\mathcal{Q}(\what \pi)$ is finite,  
\[
    \mathcal{Q}(\what\pi) = \frac12 \scall{1, \Sigma(\what \pi)^{-1}}{\what \pi}{\what \pi}, \qquad \text{and} \qquad \Q_{I_2}(\what \pi) = \frac{1}{2} \scall{}{\grad \what \pi}{\grad \what \pi}.
\]  
 Similarly if $\I{\what \pi}$ is finite, then 
 \[
    \I{\what \pi} =  \frac12\norm{1,\Sigma(\what\pi),*}{\ell^{\what E}(\what \pi \, | \, \what \gamma)}^2 ,
 \] 
 where $\ell^{\what E}(\what \pi\, | \, \what \gamma) : \what G \longmapsto \ell^{\what E}_{\what G}(\what \pi\, | \, \what \gamma) $ is the linear functional defined in \eqref{weakAlternative}.       
    \end{remark}

We conclude this subsection with the next useful result concerning the convergence of the Trace:

\begin{lemma}\label{Lemma:ConvergenceFaibleBord}
    Let $(\what \pi^n)_{n\in \N}$ be a sequence in $(L^2([0,T],H^1(\La)))^2$. If the sequence $(\what \pi^n)_{n\in\N}$ converges weakly to $\what \pi$ in $(L^2([0,T],H^1(\La))^2$ then $(\Tr(\what\pi^n))_{n\in\N}$ converges weakly to  $\Tr(\what \pi)$ in $(L^2([0,T]\times \Ga))^2$. 
\end{lemma}
\begin{proof}
    For all $\what H\in \left( \mathcal C^\infty(\overline\La)\right)^2$ such that $\what H|_{\Ga^+} = \what 0 $,  
      \[
        \int_0^T \int_{\Ga^-} \what \pi_t^n(r) \what H_t(r) \, dr \, dt  = - \sum_{k=1}^d\left\{ \scall{}{\what \pi^n}{\partial_{e_k}\what H} + \scall{}{\partial_{e_k} \what \pi^n}{\what H} \right\} . 
      \] 
      From the weak convergence of $\what \pi^n$ to $\what \pi$ in $\left(L^2([0,T],H^1(\La))\right)^2$, this converges to  
      \begin{align*}
        \lim_{n\to +\infty} \int_0^T \int_{\Ga^-} \what \pi_t^n(r) \what H_t(r) \, dr \, dt  =  - \sum_{k=1}^d\left\{ \scall{}{\what \pi}{\partial_{e_k}\what H} + \scall{}{\partial_{e_k} \what \pi}{\what H} \right\} . 
      \end{align*}
        But we also have
\begin{align*}
        \int_0^T \int_{\Ga^-} \what \pi_t(r) \what H_t(r) \, dr \, dt  =  - \sum_{k=1}^d\left\{ \scall{}{\what \pi}{\partial_{e_k}\what H} + \scall{}{\partial_{e_k} \what \pi}{\what H} \right\},
      \end{align*}
      which concludes the proof. 
\end{proof}
    
    \subsection{Superexponential estimates (Replacement lemmas)} 
    \subsubsection{Dirichlet form estimates}
      To prove the hydrodynamic behavior of the system we follow the entropy and Dirichlet forms method
introduced in \cite{gpv}. Since the invariant measure is
not explicitly known, we compute the Dirichlet form of the state of the process with respect to a
product measure with slowly varying profile $\theta$. We prove that, provided
$\theta$ is smooth enough and takes the prescribed value $\what b$ at the right boundary, the rate to which
the entropy increases is of the order of the volume $N^d$, and for finite time $T$ this implies
only a modification of the constant multiplying $N^d$.
The Dirichlet form and entropy estimates, in the bulk rely on easeful summations and integrations by parts,  and on the following inequality (see for example Section 3.1 of \cite{mo1}).
\begin{equation}
  R_N(x,x\pm e_k;\sigma) \coloneqq   1 - \frac{\nu_{\what \theta}^N(\sigma^{x,x\pm e_k})}{\nu_{\what\theta}^N(\sigma)} \leq  \frac{C_2}{N} , \label{RNinegaliteN}
\end{equation}
for some positive constant $C_2$ and where the measure $\nu^N_{\what \theta}$ have been defined in \eqref{meas}.
    
    Denote for any probability measure $\nu$ and positive function $h\in L^2(\nu)$, the functional 
    \[
      \mathcal{D}_{0,N}(h,\nu) \coloneqq \frac 12 \sum_{k=1}^d \sum_{x,x+e_k\in\La_N} \int \left[h(\sigma^{x,x+e_k}) - h(\sigma)\right]^2 \, d \nu(\sigma) ,
    \]
    and the functional 
    \begin{align*}
      \mathcal{D}^\pm_{b,N}(h,\nu) \coloneqq& \frac{1}{2} \sum_{x\in \Ga_N^\pm} \int \left( r_x^{(1,0),-}(\what b , \sigma) + r_x^{(0,-1),-}(\what b , \sigma)\right)\left[ h(\sigma^{x,-}) - h(\sigma) \right]^2 \, d\nu(\sigma) \\
      & + \frac{1}{2} \sum_{x\in \Ga_N^\pm} \int \left( r_x^{(0,1),+}(\what b , \sigma) + r_x^{(-1,0),+}(\what b , \sigma)\right)\left[ h(\sigma^{x,+}) - h(\sigma) \right]^2 \, d\nu(\sigma) \\
      & + \frac{1}{2} \sum_{x\in \Ga_N^\pm} \int r_x^{(-1,1)}(\what b , \sigma) \left[ h(\sigma^{x,-}) - h(\sigma) \right]^2 \, d\nu(\sigma) .
    \end{align*}
  The proof of the next Lemma is similar to the one of \cite{mo1} (see Lemma 3.3) and use the following change of variable formulas :

    For any functions $f,g \in L^2(\nu_{\what \theta}^N)$, for any $k\in\{1,\dots,d\}$ and any sites $x,x+e_k \in \La_N$, 
        \begin{align*}
            \int  f(\sigma) g(\sigma^{x,x+e_k}) \, d \nu_{\what \theta}^N(\sigma)  & = \int  f(\sigma^{x,x+e_k})   g(\sigma)  [1-R_N(x,x+e_k;\sigma)] \, d \nu_{\what \theta}^N  (\sigma),  \\
            \int  r_x^{(1,0),-}(\what b, \sigma)  f(\sigma) g(\sigma^{x,-}) \, d \nu_{\what \theta}^N(\sigma)  & = \int   r_x^{(0,1),+}(\what b, \sigma) f(\sigma^{x,+})   g(\sigma)  \, d \nu_{\what \theta}^N  (\sigma),  \\
            \int r_x^{(-1,0),+}(\what b, \sigma) f(\sigma) g(\sigma^{x,+})  \, d \nu_{\what \theta}^N(\sigma)  & = \int  r_x^{(0,-1),-}(\what b, \sigma) f(\sigma^{x,-})   g(\sigma)  \, d \nu_{\what \theta}^N  (\sigma),  \\
            \int r_x^{(-1,1)}(\what b, \sigma) f(\sigma) g(T^x\sigma)  \, d \nu_{\what \theta}^N(\sigma)  & = \int  r_x^{(-1,1)}(\what b, \sigma)  f(T^x\sigma)   g(\sigma)  \, d \nu_{\what \theta}^N  (\sigma).
        \end{align*}
  
    \begin{lemma}\label{DirichelEstimatesInterieur}
      There exists a positive constant $C$ depending on $\what\theta$  such that for any $a>0$ and positive function $f\in L^2(\nu_{\what \theta}^N)$ 
      \begin{align*}
        \scal{\nu_{\what \theta}^N}{\mathcal L_{\what E,N}f}{f} & \leq - (1-a)\mathcal{D}_{0,N}(f,\nu_{\what \theta}^N) + \frac{C}{a}N^{d-2} \norm{L^2(\nu_{\what \theta}^N)}{f}^2 ,\\ 
        \scal{\nu_{\what \theta}^N}{L^\pm_{b,N}f}{f} & = - \mathcal{D}^\pm_{b,N}(f,\nu_{\what \theta}^N),
      \end{align*}
      where the function $\what \theta$ is defined in \eqref{eq:cC}.
    \end{lemma}    

    \subsubsection{Replacement lemma in the bulk} 
    
    For a positive integer $\ell$ and $x\in\La_N$, denote
\[
  \La_\ell(x) \coloneqq \La_{N,\ell}(x) \coloneqq \left\{ y\in \La_N ~ \middle|~ \norm{\infty}{y-x} \leq \ell\right\} ,
\]
where $\norm{\infty}{\cdot}$ represents the sup norm in $\R^d$: for $x\in\R^d$, $\norm{\infty}{x} = \max_{i}\abs{x_i}$.
When $x=0$, we shall denote $\La_\ell(0)$ simply by $\La_\ell$.
Fix a configuration $\sigma \in \mathbb X_N$, denote the empirical mean density on the box $\La_\ell(x)$ by $\what \sigma^\ell(x)$:
\[
  \what\sigma^\ell(x) \coloneqq ( \sigma^{\ell,1}(x) ,  \sigma^{\ell,2}(x)) \coloneqq \left(\frac{1}{\abs{\La_\ell(x)}} \sum_{y\in\La_\ell(x)} \sigma(y),\frac{1}{\abs{\La_\ell(x)}} \sum_{y\in\La_\ell(x)} \sigma(y)^2 \right).  
\]
For a cylinder function $\mathfrak{h}$, that is a function on $\{0,1\}^\Z$ depending on $\sigma(x)$, $x\in\Z$ only through finitely many $x$, denote by $\tilde{\mathfrak{h}}(m,\phi)$ the expectation of $\mathfrak{h}$ with respect to $\nu_{(m,\phi)}^N$, the  product measure with density $(m,\phi)$:
\[
  \tilde{\mathfrak h}(m,\phi) \coloneqq \E^{\nu_{(m,\phi)}^N}(\mathfrak{h}(\sigma)).  
\]
Further, denote for $G\in  \mathcal C([0,T]\times \overline\La)$ and $\vare > 0$, 
\[
  V_{N,\vare}^{G,\mathfrak{h}}(s,\sigma) \coloneqq \frac{1}{N^d} \sum_{x\in\La_N} G_s\left(\frac{x}{N}\right) \left[ \tau_x \mathfrak{h}(\sigma) - \tilde{\mathfrak{h}}(\what\sigma^{\ent{\vare N}}(x)) \right]  ,
\]
where the sum is carried over all $x$ such that the support of $\tau_x \mathfrak{h}$ belongs to $\La_N$ and $\ent{\cdot}$ denotes the lower integer part.

The following propositions are essential, but we omit their proofs, since they are similar to those of \cite{flm}(see Proposition 6.2), \cite{mo1}(see Proposition 3.6),\cite{fgn1} (see Lemma 5.2 and Lemma 5.4),\cite{msv}(see Proposition 2,3 and 4 of Subsection 3.2).

\begin{proposition} \label{replacementBulk}
  Let $\left(\mu_N\right)_{N\geq 1}$ be a sequence of probabilities on $\mathbb X_N$ and $G\in  \mathcal C^{1,2}([0,T]\times \overline\La)$. For every $\delta >0$, 
  \[
    \varlimsup_{\vare \to 0 } \varlimsup_{N\to + \infty} \frac{1}{N^d} \log \mathbb P_{\mu_N}^{N,\what b} \left[ 
      \abs{\int_0^T V^{G,\mathfrak{h}}_{N,\vare }(s,\sigma_s) \, ds } > \delta
    \right] = -\infty.  
  \]
\end{proposition}
\subsubsection{Replacement lemma at the left-hand side boundary}
\begin{proposition}\label{replacementGauche}
  Let $\left(\mu_N\right)_{N\geq 1}$ be a sequence of probabilities on $\mathbb X_N$ and $ G\in  \mathcal C^{1,2}([0,T]\times \overline\La)$. For every $\delta >0$, $i\in\{1,2\}$ and $\mathfrak{a}_l\in(-\infty,1)$,
  \[
    \varlimsup_{N\to + \infty} \frac{1}{N^d}\log\mathbb P_{\mu_N}^{N,\what b}\left[ \abs{\int_0^T \frac{1}{N^{d-1}} \sum_{x\in\Ga_N^-} G_{s}\left(\frac{x}{N}\right) \left(\sigma_s(x)^i - b_i\left(\frac{x}{N}\right)\right) \, ds } > \delta \right] = - \infty .   
  \]
\end{proposition}
\subsubsection{Replacement lemma at the right-hand side boundary}
\begin{proposition}\label{replacementDroite}
  Let $\left(\mu_N\right)_{N\geq 1}$ be a sequence of probabilities on $\mathbb X_N$ and $ G\in  \mathcal C^{1,2}([0,T]\times \overline\La)$. For every $\delta >0$, $i\in \{1,2\}$ and $\mathfrak{a}_r>1$, 
  \[
    \varlimsup_{\vare \to 0 }\varlimsup_{N\to + \infty} \frac{1}{N^d}\log\mathbb P_{\mu_N}^{N,\what b}\left[ \abs{\int_0^T \frac{1}{N^{d-1}} \sum_{x\in\Ga_N^+} G_{s}\left(\frac{x}{N}\right) \left(\sigma_s(x)^i - \what\sigma_s^{\ent{\vare N},i}(x)\right) \, ds } > \delta \right] = - \infty .   
  \]
\end{proposition}  
    
    \subsection{Energy estimates}\label{energyEstimates}
  We prove in this subsection an energy estimate which is one of the main ingredients in the proof of the hydrodynamic limit and large deviations. It allows to prove for the hydrodynamic limit, that the macroscopic trajectories are in  $\left(L^2([0,T],H^1(\Lambda))\right)^2$ (cf Theorem \ref{th-hy}) and to exclude paths with infinite energy in the large deviation regime.
For $\delta >0$, define
     \[
        \tilde{\mathcal{Q}}^\delta(\what\pi) \coloneqq \sum_{k=1}^d  \tilde{\mathcal{Q}^\delta_{k}}(\what \pi),
    \]
    where
    \[
      \tilde{\mathcal{Q}^\delta_{k}}(\what \pi) \coloneqq
        \sup_{\what G \in \left( \mathcal C^{\infty}_c([0,T]\times \La)\right)^2}
      \left\{ {\tilde{\mathcal{Q}}^\delta_{{\what G},k}}\right\},
    \]
    and for $\what G \in \left( \mathcal C^{\infty}_c([0,T]\times \La)\right)^2$, $k\in\{1,\dots,d\}$
    \[
    {\tilde{\mathcal{Q}}^\delta_{{\what G},k}} \coloneqq
    \scall{}{\what \pi}{\partial_{e_k} \what G} - \delta \scall{}{\what G}{\Sigma(\what \pi)\what G} .
    \]
  Remind the definition of $\Q$ given by \eqref{defQenergy} and notice that, for any $\delta>0$
    \[
        \mathcal{Q}(\what \pi) = \sum_{k=1}^d {\tilde{\mathcal{Q}}^{\frac 12}_{k}}(\what \pi) \qquad \text{and } \qquad {\tilde{\mathcal{Q}}}^\delta(\what\pi) \coloneqq \sum_{k=1}^d {\tilde{\mathcal{Q}}^\delta_{k}}(\what \pi) = \frac{1}{2\delta} \mathcal{Q}(\what \pi),
    \]
    We shall denote $\mathcal{Q}_k \coloneqq \tilde{\mathcal{Q}}_k^{\frac 12} $ so that $\mathcal{Q} = \sum_{k=1}^d \mathcal{Q}_k$.

    \begin{lemma}\label{l-energy1}
        Let $\what\pi\in\mathcal{D}([0,T],\what{\mathcal{M}}_0)$. Then,  $\Q(\what\pi)$ is finite if and only if $\what \pi \in \left(L^2([0,T],H^1(\La))\right)^2$ and 
         
        \[
          \sum_{k=1}^d \scall{}{\partial_{e_k}\what \pi}{ \Sigma(\what \pi)^{-1}\partial_{e_k}\what \pi} < + \infty .
        \]
      In that case 
      \[
        \Q(\what \pi) = \frac 12 \sum_{k=1}^d \scall{}{\partial_{e_k}\what \pi}{ \Sigma(\what \pi)^{-1}\partial_{e_k}\what \pi} .
      \] 
      Moreover, the functional $\Q(\cdot)$ is convex on $\mathcal D([0,T],\what{\mathcal{M}}_0)$. 
    \end{lemma}

    \begin{proof} 
      First, consider $\what \pi \in \left(L^2([0,T], H^1(\La))\right)^2$ such that 
      \[
          \sum_{k=1}^d \scall{}{\partial_{e_k}\what \pi}{ \Sigma(\what \pi)^{-1}\partial_{e_k}\what \pi} < + \infty .
        \]
        Let $k$ be an integer in $\{1,\dots,d\}$.
      Then for $\what G \in \left( \mathcal C^\infty_c ([0,T]\times \La)\right)^2$, we can apply an integration by parts to get, using inequality \eqref{inegaliteCarreScalaire} with $S = \Sigma(\what\pi)$,
      \begin{align*}
          \scall{}{\what \pi}{\partial_{e_k}  \what G} = - \scall{}{\partial_{e_k}  \what \pi}{ \what G} \leq \frac{1}{2} \scall{}{\partial_{e_k} \what \pi}{\Sigma(\what\pi)^{-1} \partial_{e_k} \what \pi} + \frac12 \scall{}{\what G}{\Sigma(\what \pi) \what G} .
      \end{align*}
      Hence we directly get 
      \[
          \mathcal{Q}(\what \pi ) \leq \frac 12 \sum_{k=1}^d \scall{}{\partial_{e_k}\what \pi}{ \Sigma(\what \pi)^{-1}\partial_{e_k}\what \pi} < + \infty .
        \]
      Conversely, assume that $\Q(\what \pi)$ is finite. Because $\what \pi$ is bounded and $\Sigma$ is smooth, there exists a constant $C_\Sigma$ depending only on $\Sigma$ such that 
      \[
        \Sigma(\what \pi) \leq C_\Sigma I_2,
      \]
      where we recall that the partial order on $S^+_2$ has been defined in \eqref{orderMatrices}.
      Thus for $k\in \{1,\dots,d\}$, taking $G\longrightarrow C_\Sigma G$, we have
      \begin{align*}
          \sup \left\{ \scall{}{\what \pi}{\partial_{e_k} \what G} - \scall{}{\what G}{ \what G} \right\}
         & = C_\Sigma \sup \left\{ \scall{}{\what \pi}{\partial_{e_k} \what G} - \scall{}{\what G}{C_\Sigma I_2 \what G} \right\} \\
   \ &     \leq   C_\Sigma \sup \left\{ \scall{}{\what \pi}{\partial_{e_k} \what G} - \scall{}{\what G}{\Sigma(\what \pi) \what G} \right\} \\
        \ & = \frac{1}{2}C_\Sigma(\what \pi) \Q(\what \pi) < + \infty ,
      \end{align*}
      where the supremum is taken over $\what G \in \left( \mathcal C_c^\infty([0,T]\times \La)\right)^2$.
      Therefore $\partial_{e_k}\what \pi$ belongs to $\left(L^2([0,T]\times \La)\right)^2$ for all $k\in \{1,\dots,d\}$, thus   $\what \pi \in \left(L^2([0,T],H^1(\La))\right)^2$. We can now perform an integration by parts to get
      \[
        \Q(\what \pi) = \sum_{k=1}^d \sup_{\what G \in \left( \mathcal C_c^\infty([0,T]\times \La)\right)^2} \left\{ \scall{}{\partial_{e_k} \what \pi}{ \what G} - \frac{1}{2}\scall{}{\what G}{\Sigma(\what \pi) \what G} \right\},
      \]
      but because $\left( \mathcal C_c^\infty([0,T]\times \La)\right)^2$ is dense in $\left(L^2([0,T]\times \La)\right)^2$, the supremum of the previous expression can be taken over all functions $\what G \in \left(L^2([0,T]\times \La)\right)^2$. Let $k\in \{1,\dots,d\}$, we now need to take the right function $\what G$. Consider $0<\delta <1$ and denote
      \[
        \Sigma_\delta(m,\phi) \coloneqq 2 \begin{pmatrix}
          \phi - m^2 + \delta & m(1-\phi) \\ m(1-\phi) &  (\phi+\delta)(1-\phi + \delta) 
        \end{pmatrix},
      \]  
      then we have $\Sigma_\delta(\what \pi) = \Sigma(\what \pi) + M_\delta(\what \pi)$ where 
      \[
        M_\delta(\what \pi) \coloneqq 2 \delta \begin{pmatrix}
          1 & 0 \\ 0 & 1+\delta
        \end{pmatrix}
      \]
      is a positive-definite matrix bounded below by $2\delta I_2$. Therefore, because $\Sigma(\what\pi)$ is a positive semi-definite matrix, we have $\Sigma_\delta(\what \pi) \geq \Sigma(\what \pi)$, in particular if we denote 
      \[
        \Q_{\delta,k}(\what \pi) \coloneqq \sup_{\what G \in \left(L^2([0,T]\times \La)\right)^2} \left\{ \scall{}{\partial_{e_k} \what \pi}{ \what G} - \frac{1}{2}\scall{}{\what G}{\Sigma_\delta(\what \pi) \what G} \right\},
      \] 
      we have $$\Q_{\delta,k}(\what \pi) \leq \sup_{\what G \in \left( \mathcal C_c^\infty([0,T]\times \La)\right)^2} \left\{ \scall{}{\partial_{e_k} \what \pi}{ \what G} -\frac{1}{2} \scall{}{\what G}{\Sigma(\what \pi) \what G} \right\}. $$ 
      Moreover, a simple computation shows that
      \begin{align*}
        \scall{}{\Sigma_\delta(\what \pi)^{-1} \partial_{e_k} \what \pi}{\Sigma_\delta(\what\pi)^{-1}\partial_{e_k} \what \pi} \leq \frac{1}{4\delta^2} \scall{}{ \partial_{e_k} \what  \pi}{\partial_{e_k} \what \pi},
      \end{align*}
      which is finite because $\what \pi \in \left(L^2([0,T],H^1(\La))\right)^2$. Thus we can take as a test function
     $\what G \coloneqq \Sigma_\delta(\what \pi)^{-1}\partial_{e_k} \what \pi \in \left(L^2([0,T]\times \La)\right)^2$, that yealds,
      \[
        \frac{1}{2} \scall{}{\partial_{e_k} \what \pi}{\Sigma_\delta(\what\pi)^{-1} \partial_{e_k} \what \pi} \leq \Q_{\delta,k}(\what \pi) \leq\sup_{\what G \in \left( \mathcal C_c^\infty([0,T]\times \La)\right)^2} \left\{ \scall{}{\partial_{e_k} \what \pi}{ \what G} - \frac{1}{2}\scall{}{\what G}{\Sigma(\what \pi) \what G} \right\} . 
      \]
      We conclude the proof by letting $\delta \xrightarrow{} 0 $ using Fatou's lemma which can be applied because $\Sigma_\delta(\what \pi)$ is positive-definite and summing in $k$.

      The convexity of the functional $\Q$ follows from the concavity of the function 
      \[
          (m,\phi)\in \overline{\mathbf  I}\mapsto \scal{}{(x,y)^t}{\Sigma(m,\phi) (x,y)^t},
    \] 
    for any $(x,y)\in \R^2$. Indeed, we have
      \[
        \scal{}{(x,y)^t}{\Sigma(m,\phi) (x,y)^t} = 2 (\phi - m^2) x^2  + 4 m(1-\phi) xy +  2 \phi(1-\phi) y^2 .
      \]
      Computing the Hessian matrix we get 
      \[
        H(\Sigma(m,\phi)) = \begin{pmatrix}
          -4 x^2 & -4xy \\ -4xy & -4 y^2
        \end{pmatrix},
      \]
      which has a zero determinant and a negative trace, therefore it has a zero eigenvalue and a negative eigenvalue, implying that the Hessian is negative semi-definite, thus the function is concave and $\Q$ is convex.
    \end{proof}
    
    Now, introduce the following approximation of the identity on $\La$:
    
    \[
        u_\vare(x) \coloneqq \frac{1}{(2\vare)^d}\1_{ [-\vare,\vare]^{d}}(x).
    \]
    Note that for $\vare>0$,  $x=(x_1,\dots,x_d) \in \La_N$ with $x_1  \in \{ \ent{-N(1-\vare)}, \dots, \ent{N(1-\vare)}\}$, $y\in \Ga^-_N$ and $z\in \Ga^+_N$, 
    \[
      \what \sigma^{\vare N}(x)= \frac{(2\vare N)^d }{(2\vare N+1)^d } \left( \what \pi * \what u_\vare   \right)\left(\frac{x}{N}\right),
    \] 
    where $\what u_\vare \coloneqq (u_\vare,u_\vare)$  and $\what   \pi * \what u_\vare \coloneqq (\pi_1 * u_\vare, \pi_2*u_\vare)$. 

    Notice that for all $\what\pi \in \DM$, $\what \pi * \what u_\vare$ belongs to $\mathcal D ([0,T], \what {\mathcal M}_0)$.
    
    \begin{lemma}\label{lemma:energyEstimate}
      There exists a constant $C_1$ depending only on $\what b$ such that for all $\delta_0 > 0$, for all $\delta \in [0,\delta_0]$, for all $(\mu_N )_{N\in\N}$ sequence of probabilities on $\mathbb X_N$, for all $k\in \{1,\dots,d\}$  and for all $\what G \in \left( \mathcal C^\infty_c([0,T]\times \La)\right)^2$,
      \[
        \varlimsup_{\vare \to 0 } \varlimsup_{N\to + \infty} \frac{1}{N^d} \log  \E_{\mu_N}^{N,\what b}\left[\exp\left( \delta N^d \tilde\Q^{\delta_0}_{\what G,k}(\what \pi_N * \what u_\vare)  \right) \right] \leq C_1 (T+1) .
      \] 
    \end{lemma}

    \begin{proof}
      Assume without loss of generality that $\vare$ is small enough so that the support of $\what G$ is contained in $[0,T]\times [-1+\vare,1-\vare] \times \mathbb T^{d-1}$. Recall the definition of $\what \theta$, since for all $\sigma_N\in \mathbb X_N$, $\nu^N_{\what\theta}(\sigma_N) \geq \exp\left(-C_1' N^d\right)$ for some finite constant $C_1'$ only depending on $\what\theta$, it is enough to prove the Lemma with $\E_{\nu^N_{\what\theta}}^{N,\what b}$ in place of $\E_{\mu_N}^{N,\what b}$.  

      Set the functions 
      \begin{align*} 
        &\mathfrak h_1(\sigma) \coloneqq \left[ \sigma(-e_k  ) - \sigma(0) \right]^2, \\
        &\mathfrak h_2(\sigma) \coloneqq \left[ \sigma(-e_k  ) - \sigma(0) \right]\left[ \sigma(-e_k )^2 - \sigma(0)^2 \right], \numberthis \label{definitionDesPsi}\\
        & \mathfrak{h}_3(\sigma) \coloneqq \left[ \sigma(-e_k  )^2 - \sigma(0)^2 \right]^2,
      \end{align*}
      and the matrix 
      \[
        \Sigma^{\mathfrak{h}}(\sigma) \coloneqq \begin{pmatrix}
          {\mathfrak{h}}_1(\sigma) & {\mathfrak{h}}_2(\sigma) \\ {\mathfrak{h}}_2(\sigma) & {\mathfrak{h}}_3(\sigma)
        \end{pmatrix}.
      \]
      Note that 
      \[
        \tilde\Sigma^{\mathfrak{h}}(a_1,a_2) \coloneqq \E^{\nu^N_{(a_1,a_2)}}\left(\Sigma^{\mathfrak{h}}\right) = \Sigma(a_1,a_2) .
      \]
      Denote $\what{\mathfrak{h}} \coloneqq \left({\mathfrak{h}}_1,{\mathfrak{h}}_2,{\mathfrak{h}}_3\right)$.
      For fixed $\delta_0 >0$,  define 
      \[
        B \coloneqq B^{\what G, \what{\mathfrak{h}}}_{N,\vare,\delta_0} \coloneqq \left\{ \sigma \in \mathcal{D}([0,T], \mathbb X_N) ~\middle| ~ \abs{\int_0^T  \sum_{j=1}^3 V_{N,\vare}^{\what G, {\mathfrak{h}}_j}(t,\sigma_t)\, dt} \leq \frac{1}{2\delta_0^2}  \right\}, 
      \]
      where 
      \[
        V^{\what G , {\mathfrak{h}}}_{N,\vare}(t,\sigma) \coloneqq \frac{1}{N^d}\sum_{i=1}^2 \sum_{x\in\La_N} G_{i,t}\left(\frac{x}{N}\right) \left[ \tau_x{\mathfrak{h}}(\sigma)- \tilde {\mathfrak{h}}(\sigma^{\vare N}(x)) \right]. 
      \]
      For all $\delta >0$, we have,
      \begin{align*}
        &\E^{N,\what b}_{\nu^N_{\what\theta}} \left[ \exp\left( \delta N^d \tilde \Q_{\what G , k}^{\delta_0}(\what \pi_N * \what u_\vare) \right)\right] \\
        =& \E^{N,\what b}_{\nu^N_{\what\theta}} \left[ \exp\left( \delta N^d \tilde \Q_{\what G , k}^{\delta_0}(\what \pi_N * \hat u_\vare) \right) \1_{B}\right] + \E^{N,\what b}_{\nu^N_{\what\theta}} \left[ \exp\left( \delta N^d \tilde \Q_{\what G , k}^{\delta_0}(\what\pi_N * \what u_\vare) \right) \1_{B^\complement}\right] .
      \end{align*}
      The second term diverges to $-\infty$ using Proposition \ref{replacementBulk}, thus using inequality \eqref{inegaliteLOG}, to prove the lemma, it suffices to show that, for any $0\le\delta\le \delta_0$,
      \[
        \varlimsup_{\vare \to 0 } \varlimsup_{N\to + \infty} \frac{1}{N^d} \log  \E^{N,\what b}_{\nu^N_{\what\theta}}\left[\exp\left( \delta N^d  \tilde\Q^{\delta_0}_{\what G,k}(\what \pi_N * \what u_\vare)  \right) \1_{B} \right] \leq C_1 (T+1) ,
      \] 
      for some positive constant $C_1$ which does not depend on $G$ and $\delta$.
      Now recalling the definition of $\tilde\Q^{\delta_0}_{\what G,k}$, we find from a discrete integration by parts 
      \begin{align*}
        &\int_0^T \scal{}{\what \pi_N * \what u_\vare}{\partial_k \what G_t}\, dt  =  \int_0^T \frac{1}{N^{d-1}} \sum_{x\in\La_N} \sum_{i=1}^2 G_{i,t}\left(\frac{x}{N}\right) \left[ \sigma_t(x-e_k)^i - \sigma_t(x)^i \right] \, dt + o_\vare(1) + o_N(1),  
      \end{align*}
      where $\lim_{\vare \to 0 } o_\vare(1) =0$ and $\lim_{N\to + \infty} o_N(1) = 0$. Moreover for the second term of $\tilde \Q$, using the fact that $\sigma_\cdot$ belongs to $B$, 
      \[
        \delta_0 \int_0^T \scal{}{\what G_t}{\Sigma(\what\pi^N_t * \what u_\vare) \what G_t} \, dt \geq  \delta_0 \int_0^T \frac{1}{N^d}\sum_{x\in\La_N} \scal{}{\what G_t\left(\frac{x}{N}\right) }{ \tau_x \Sigma^{{\mathfrak{h}}}(\sigma_t) \what G_t\left(\frac{x}{N}\right)} \, dt - \delta_0 O_{\what G^2}(N,\vare) - \frac{1}{\delta_0},
      \]
      where $\lim_{N \to +\infty } O_{\what G^2}(N,\vare) =0$.
      Therefore, to conclude the proof, it is enough to show that 
      \[
  \varlimsup_{N\to + \infty} \frac{1}{N^d} \log \E^{N,\what b}_{\nu^N_{\what\theta}}\left[ \exp\left( N^d \int_0^T V_{\hat G} ^\delta (t,\sigma_t) \, dt  \right) \right] \leq C_1 T  , 
\]
for any $\delta \leq \delta_0$, where 
\[
      V_{\what G}^\delta(t,\sigma) = \delta \frac{1}{N^{d-1}} \sum_{i=1}^2 \sum_{x\in\La_N} G_{i,t}\left(\frac{x}{N}\right) \left[   \sigma(x-e_k)^i - \sigma(x)^i  \right]  - \delta^2 \frac{1}{N^d}\sum_{x\in\La_N} \scal{}{\what G_t\left(\frac{x}{N}\right) }{ \tau_x \Sigma^{{\mathfrak{h}}}(\sigma) \what G_t\left(\frac{x}{N}\right)}.
\]
      But $V_{\what G}^\delta = V_{\delta \what G}^1$, therefore, to prove the Lemma, it suffices to show that, for any $\widehat G\in ({\mathcal C}_c^{\infty}([0,T]\times\Lambda))^2$,
      \[
  \varlimsup_{N\to + \infty} \frac{1}{N^d} \log \E^{N,\what b}_{\nu^N_{\what\theta}}\left[ \exp\left( N^d \int_0^T V_{\hat G} ^1 (t,\sigma_t) \, dt  \right) \right] \leq C_1 T  ,
\]
for some constant $C_1$ independent of $\what G$.
 Using the Feynman-Kac formula, 
    \begin{align*}
      &\frac{1}{N^d} \log \E^{N,\what b}_{\nu^N_{\what\theta}}\left[ \exp\left( N^d \int_0^T V_{\hat G} ^1 (t,\sigma_t) \, dt  \right) \right] \\
      \leq & \int_0^T \sup_{\norm{L^2(\nu_{\what\theta}^N)}{f}^2=1} \left\{ \scal{\nu_{\what\theta}^N}{V_{\what G}^1(t,\sigma)}{f^2} + \scal{\nu_{\what\theta}^N}{N^{2-d}\cl_{\what E,N}f+N^{2-d-\mathfrak{a}_l} L^-_{b,N}f+N^{2-d-\mathfrak{a}_r} L^+_{b,N}f}{f} \right\} \, dt .
    \end{align*}
    Thus, from the Dirichlet estimates stated in Lemma \ref{DirichelEstimatesInterieur}, there exists a positive constant $c$ only depending on $\what\theta$ such that for any $\beta>0$, 
    \begin{align*}
      &\scal{\nu_{\what\theta}^N}{N^{2-d}\cl_{\what E,N}f+N^{2-d-\mathfrak{a}_l} L^-_{b,N}f+N^{2-d-\mathfrak{a}_r} L^+_{b,N}f}{f} \leq  - N^{2-d}(1-\beta) \mathcal{D}_{0,N}(f,\nu_{\what\theta}^N) + \frac{c}{\beta}\norm{L^2(\nu_{\what\theta}^N)}{f}^2 .    
    \end{align*} 
    It now remains to show that for some fixed $\beta$
\[
  \varlimsup_{N\to + \infty} \int_0^T \sup_{\norm{L^2(\nu_{\what\theta}^N)}{f}^2=1} \left\{ \scal{\nu_{\what\theta}^N}{V_{\what G}^1(t,\sigma)}{f^2}  - N^{2-d}(1-\beta) \mathcal{D}_{0,N}(f,\nu_{\what\theta}^N)  \right\} \leq C_1 T.
\]
From the definition of $V$, we can split 
\[
  \scal{\nu_{\what\theta}^N}{V_{\what G}^1(t,\sigma)}{f^2} = I_1 - I_2,
\]
where 
\[
    I_1 \coloneqq  \frac{1}{N^{d-1}} \sum_{i=1}^2 \sum_{x\in\La_N} G_{i,t}\left(\frac{x}{N}\right) \int  \left[  \sigma(x-e_k)^i - \sigma(x)^i  \right]f^2(\sigma) \, d\nu_{\what\theta}^N(\sigma),
\]
and 
\[
    I_2 \coloneqq \frac{1}{N^d}\sum_{x\in\La_N}  \int \scal{}{\what G_t\left(\frac{x}{N}\right) }{ \tau_x \Sigma^{{\mathfrak{h}}}(\sigma) \what G_t\left(\frac{x}{N}\right)} f^2(\sigma) \, d\nu_{\what\theta}^N(\sigma) .
\]
Estimating $I_1$ in terms of $I_2$ and $\mathcal{D}_{0,N}(f,\nu_{\what\theta}^N)$, using the change of variable $\sigma \mapsto \sigma^{x,x-e_k}$  
\begin{align}
  \hspace{-0.15cm} I_1 &= \frac{1}{2} \frac{1}{N^{d-1}} \sum_{i=1}^2\sum_{x\in\La_N}  G_{i,t}\left(\frac{x}{N}\right) \int \left[  \sigma(x-e_k)^i - \sigma(x)^i \right] \left[ f^2(\sigma) - f^2(\sigma^{x,x-e_k}) \right] \, d\nu^N_{\what\theta}(\sigma) \nonumber  \\
     & \ \ +  \frac{1}{2} \frac{1}{N^{d-1}}\sum_{i=1}^2 \sum_{x\in\La_N}  G_{i,t}\left(\frac{x}{N}\right) \int \left[  \sigma(x-e_k)^i - \sigma(x)^i \right]f^2(\sigma^{x,x-e_k}) R_N(x,x-e_k;\sigma)  \, d\nu^N_{\what\theta}(\sigma) \label{eq:estimeeEnergiePreuve},
\end{align}
where $R_N(x,x-e_k;\sigma)$ are defined in \eqref{RNinegaliteN}. Using inequality \eqref{BasicSommeCarre}, the first sum of \eqref{eq:estimeeEnergiePreuve} is bounded, for any $a>0$ by 
  \begin{align*}
    &\frac{aN^{2-d}}{4} \mathcal{D}_{0,N}(f,\nu_{\what\theta}^N) + \frac{1}{4a} \sum_{x\in\La_N} \int \left[ f(\sigma) + f(\sigma^{x,x-e_k})\right]^2 \scal{}{\what G_t\left(\frac{x}{N}\right)}{\tau_x \Sigma(\sigma)\what G_t \left(\frac{x}{N}\right)} \, d\nu_{\what\theta}^N(\sigma) \\ 
    \leq &      \frac{aN^{2-d}}{4} \mathcal{D}_{0,N}(f,\nu_{\what\theta}^N) + \frac{I_2}{a} +  \frac{C_{\what G}}{aN} , 
  \end{align*}
  where the last expression is obtained using the inequality \eqref{BasicSommeCarre}, the fact that $\what G$ is smooth with a compact support and the inequality \eqref{RNinegaliteN}. For the second sum, using \eqref{BasicSommeCarre} and \eqref{RNinegaliteN} again, the sum is bounded by 
  \[
    a C_3 + \frac{I_2}{a} . 
  \]
  Taking into account all our inequalities, we showed that 
  \[
    I_1 \leq \frac{aN^{2-d}}{4} \mathcal{D}_{0,N}(f,\nu_{\what\theta}^N) + \frac{2 I_2}{a} + \frac{C_{\what G}}{aN}  + a C_3 . 
  \]
  With $a=2$  and $\beta=1/2$, we get the desired result.
    \end{proof}

    \begin{corollary}\label{energieEstimee}
      Let $\left(\what G_i\right)_{i\in\N}$ be a sequence of functions in $\left( \mathcal C^{\infty}_0([0,T]\times \La)\right)^2$ and $(\mu_N)_{N\geq 1}$ a sequence of probabilities on $\mathbb X_N$. There exists a constant $C_1$ depending only on $\what b$ such that for all $j\in\N$, $k\in\{1,\dots,d\}$, $\delta_0 >0$ and $\delta \in [0,\delta_0]$, 
      \[
        \varlimsup_{\vare \to 0 } \varlimsup_{N\to + \infty} \frac{1}{N^d} \log  \E_{\mu_N}^{N,\what b}\left[\exp\left( \delta N^d  \max_{0\leq i \leq j} \tilde\Q^{\delta_0}_{\what G_i,k}(\what \pi_N * \what u_\vare)  \right) \right] \leq C_1 (T+1) .
      \] 
    \end{corollary}

    \begin{proof}
 The proof is a direct consequence of Lemma \ref{lemma:energyEstimate} and inequality \eqref{inegaliteLOG}. 
    \end{proof}

    \section{Hilbert basis of eigensolutions}
    \label{AppendixBasisEigen}

    The family of functions $\left(V_{n}\right)_{n\in \N^d}$ defined for $n\coloneqq (n_1,\dots,n_d)\in \N^d$  by   
    \begin{align}\label{A-Vn}
      V_{(n_1,\dots,n_d)}(x_1,\dots,x_d) \coloneqq \sin\left(\left(n_1 + \frac{1}{2}\right) \frac{\pi}{2}(x_1+1)\right) \prod_{i=2}^d \sqrt{2} \sin\left(\pi (n_i+1) x_i \right),
    \end{align}
    is a Hilbert basis of $L^2(\La)$. Furthermore it is a family of solutions of the problem 
    \begin{align*}
      \left\{ \begin{array}{rl}
        - \Delta \phi &= \la \phi, \\
        \phi\vert_{\Ga^-} &= 0 , \\ 
        \partial_{e_1} \phi \vert_{\Ga^+} &= 0 ,
      \end{array}  \right.
    \end{align*}
    with 
    \[
    \la_{(n_1,\dots,n_d)} = \left( n_1 + \frac12\right)^2\frac{\pi^2}{4} + \sum_{i=2}^d \pi^2 (n_i+1)^2,
    \]
    and we can write $\la_{(n_1,\dots,n_d)} = \sum_{i=1}^d \la_{(n_1,\dots,n_d),i}$ with $\la_{(n_1,\dots,n_d),1} = \left( n_1 + \frac12\right)^2\frac{\pi^2}{4}$ and $\la_{(n_1,\dots,n_d),i} = \pi^2 (n_i+1)^2$ for $i\neq 1$.
    Moreover, the family of functions $\left( \frac{\partial_{e_i}V_{n}}{\la_{(n_1,\dots,n_d),i}}  \right)_{(n_1,\dots,n_d)\in \N^d}$ is a Hilbert basis of $L^2(\La)$. Furthermore, the family of functions $\left(U_{n_1,\dots,n_d}\right)_{(n_1,\dots,n_d)\in \N^d}\coloneqq \left( \frac{\partial_{e_1}V_{n_1,\dots,n_d}}{\la_{(n_1,\dots,n_d),1}}  \right)_{(n_1,\dots,n_d)\in \N^d}$ is a family of solutions to the problem  
    \begin{align*}
      \left\{ \begin{array}{rl}
        - \Delta \phi &= \la \phi, \\
        \partial_{e_1}\phi\vert_{\Ga^-} &= 0 , \\ 
         \phi \vert_{\Ga^+} &= 0 ,
      \end{array}  \right.
    \end{align*}
    and allows us to prove that for any $f,g \in \mathcal H_{0,-}$, 
    \[
    \sum_{i=1}^d \scal{}{\partial_{e_i} f}{\partial_{e_i} g} = \sum_{n\in\N^d} \la_{n} \scal{}{f}{V_{n}}\scal{}{g}{V_{n}}.
    \]
    Finally, the family $\left(V_{(n_2,\dots,n_d)}\right)_{(n_2,\dots,n_d) \in \N^{d-1}}$ defined by 
    \begin{align*}
      V_{(n_2,\dots,n_d)}(x_2,\dots,x_d) \coloneqq  \prod_{i=2}^d \sqrt{2} \sin\left(\pi (n_i+1) x_i \right), 
    \end{align*}
    is also a Hilbert basis on $L^2(\Ga)$.   

\section{Uniqueness}\label{Appendix:uniquenessPerturbe}

We show in this Appendix the uniqueness of the perturbed problem \eqref{Eq:perturbeR+D-hydrodynamic}. The uniqueness of the hydrodynamic limit \eqref{eq:1} is then an immediate corollary.

\begin{proposition}\label{Uniqueness}
      For any $\what E\in \left(\R^d\right)^2$ and $\what H \in \left(  \mathcal C^{1,2}_{0,-}([0,T]\times \overline \La)\right)^2$, there exists a unique weak solution of  \eqref{Eq:perturbeR+D-hydrodynamic}.
    \end{proposition}

    \begin{proof}
      From the Appendix \ref{AppendixBasisEigen}, there exists a family of functions $\left(V_n\right)_{n\in\N^d}$ that are eigensolutions to the problem
      \[
        \left\{ \begin{array}{rl}
          - \Delta \phi  =& \la \phi ,  \\
          \phi|_{\Ga_-} =& 0 , \\
          \partial_{e_1}\phi|_{\Ga^+} =& 0 ,
        \end{array}\right.
      \]
      and this family forms a complete, orthonormal system in the Hilbert space $L^2(\La)$.
    Since $\N^d$ is countable, we may index this family by $\N$, with this notation, for any $U,W\in L^2(\La)$,
      \[
        \scal{}{U}{W} = \sum_{i=0}^{+\infty} \scal{}{U}{V_i} \scal{}{W}{V_i}.
      \]
      Furthermore, if $U$ and $W$ belong to $\mathcal H^1_{0,-}(\La)$,
      \[
        \sum_{k=1}^d \scal{}{\partial_{e_k} U}{\partial_{e_k} W} = \sum_{i=0}^{+\infty} \la_i \scal{}{U}{V_i} \scal{}{W}{V_i} .
      \]
      From the explicit formula \eqref{A-Vn}, for all $i\in\N$, the function $V_i$ is a product function and we can write the eigenvalue as $\la_i = \sum_{k=1}^d \la_{i,k}$, such that for all $k\in\{1,\dots,d\}$, the family of functions $\left(\frac{\partial_{e_k}V_i}{\la_{i,k}}\right)_{i\in\N}$ is a Hilbert basis of $L^2(\La)$.

      Let us consider $\what{\rho}_1 \coloneqq (m_1,\phi_1)$ and $\what{\rho}_2 \coloneqq (m_2,\phi_2)$, 2 weak solutions of the problem \eqref{eq:1}. Denote $\what{\overline\rho} \coloneqq \what\rho_1 - \what\rho_2$, $\overline{m}\coloneqq m_1 - m_2$ and $\overline\phi \coloneqq \phi_1 - \phi_2 $, observe that $\overline m$ and $ \overline \phi$ belongs to $H^1_{0,-}(\La)$. We introduce the sequence of functions
      \[
        G_n(t) \coloneqq \underbrace{\sum_{i=0}^n \scal{}{\overline m_t}{V_i}^2}_{\eqqcolon G_n^{\overline m}(t)}  + \underbrace{\sum_{i=0}^n \scal{}{\overline \phi_t}{V_i}^2}_{\eqqcolon G_n^{\overline\phi}(t)}  ,
      \]
        thus we have $G_n(t) \xrightarrow[n\to + \infty]{} G(t) \coloneqq \norm{L^2(\La)}{\overline m_t}^2 + \norm{L^2(\La)}{\overline \phi_t}^2 $. The goal is to show  that $G(t) = 0$ for all $t$ in $[0,T]$.

        We treat only the term  $G_n^{\overline m}$ corresponding to $\overline m$, the second term $ G_n^{\overline\phi}$ can be handled in the same way.
        Applying the formula \eqref{weakAlternative} we have, from the boundary conditions of $V_i$,
        \begin{align*}
          \scal{}{\overline m_t}{V_i} =  \int_0^t \scal{}{\Delta V_i}{\overline m_s} \, ds
          + \frac12 \sum_{\ell =1}^2\sum_{k=1}^d \int_0^t \scal{}{\partial_{e_k} V_i}{\left[ \Sigma_{1,\ell}(\what\rho_{1,s})-\Sigma_{1,\ell}(\what\rho_{2,s})\right] \big( E^\ell_k +2\partial_{e_k} H_{\ell,s}\big)} \, ds ,
        \end{align*}
        which is differentiable in time, and by the fundamental theorem of calculus, we get
        \begin{align*}
          \left(G_n^{\overline m}\right)(t) & =  -2 \int_0^t \Big\{\sum_{i=0}^n \la_i \scal{}{ V_i}{\overline m_s}\scal{}{ V_i}{\overline m_s}\Big\} ds \\
          \ & + \int_0^t \Big\{
          \sum_{\ell =1}^2\sum_{k=1}^d \sum_{i=0}^n\scal{}{\partial_{e_k} V_i}{\left[ \Sigma_{1,\ell}(\what\rho_{1,s})-\Sigma_{1,\ell}(\what\rho_{2,s})\right] \big( E^\ell_k +2\partial_{e_k} H_{\ell,s}\big)} \scal{}{ V_i}{\overline m_s} \Big\}\, ds.
        \end{align*}
The first sum in the right hand side converges to $-2 \sum_{k=1}^d\norm{L^2(\La)}{\partial_{e_k}\overline{m}_t}^2$, on the other hand,
using $d$ times the inequality \eqref{BasicSommeCarre},  and the Lipschitz continuity of $\Sigma_{1,\ell}$ for $\ell=1,2$, the second line in
the previous formula is bounded above, for all $a>0$ by
\begin{equation*}
\begin{aligned}
\ &\frac12 \sum_{\ell =1}^2\sum_{k=1}^d \sum_{i=0}^n
\int_0^t \Big(\Big\{
  {a\la_{i,k}} \scal{}{ V_i}{\overline m_s}^2\Big\} 
  +
  \Big<
 \frac{\partial_{e_k} V_i}
 {\sqrt{a \la_{i,k}}}
 ,\left[ \Sigma_{1,\ell}(\what\rho_{1,s}) - \Sigma_{1,\ell}(\what\rho_{2,s})\right] \big( E^\ell_k +2\partial_{e_k} H_{\ell,s}\big) \Big>^2\, \Big) ds
 \\
\ & \le a \sum_{k=1}^d \int_0^t \norm{L^2(\La)}{\partial_{e_k} \overline m_s}^2 ds +
 \frac{C}{a}\int_0^t \norm{L^2(\La)}{\what{\overline{\rho}}_s}^2\, ds ,
 \end{aligned}
\end{equation*}
where we have used  the fact that $\left(\frac{\partial_{e_k}V_i}{\sqrt{\la_{i,k}}}\right)_{i\in\N}$ is an Hilbert basis of $L^2(\La)$,
and $C=C_{\what E,\Sigma, \what H}$ is a constant depending on $\what E$,$\Sigma$ and $\what H$.
Taking $a=2$ and using the same arguments for $G_n^{\overline \phi}$,
we finally obtain,
 \[
          G(t) \leq  C_{\what E,\Sigma,\what H} \int_0^t G(s) \, ds,
        \]
for some positive constant $C_{\what E,\Sigma, \what H}$.  We conclude the proof of the uniqueness by the Grönwall's Lemma.
    \end{proof}

\section{Heat solution in \texorpdfstring{$I_\vare$}{Iepislon} a.e. in \texorpdfstring{$[\delta,T]\times \La$}{[delta,T] x Lambda}}

When looking up the heat equation \eqref{Eq:heat} for $\what E = \what 0$, the problem is no longer coupled and we can consider the equation corresponding to the second coordinate 

\begin{equation}\label{eq:heatPhi}
  \left\{ \begin{array}{rl}
    \partial_t \phi &= \; \Delta \phi, \\
    \phi_0 &= \; \gamma_2 \in [0,1], \\
    \phi\vert_{\Ga^-} &= \;b_2 \in [0,1] , \\
    \partial_{e_1} \phi\vert_{\Ga^+} &=\; 0 ,\\
    \phi(t,u) &\in [0,1] \text{ a.e. in $[0,T]\times \La$.}
  \end{array} \right.
\end{equation}

\begin{lemma}\label{lemma:heatBORNEE}

  Let $\phi$ be the unique solution of the heat equation \eqref{eq:heatPhi}.        
    Then, $\phi \in  \mathcal C^2([0,T]\times \La) \cap  \mathcal C([0,T]\times \overline\La)$ and for all $0<\delta< T$, there exists $\vare > 0 $ such that $\vare \leq \phi \leq 1 - \vare$ in $[\delta ,T] \times \La$.    
  \end{lemma} 
  \begin{proof}
    The existence and uniqueness of a solution follows from the Theorem \ref{th-hy} and Proposition \ref{Uniqueness} with $\what E =\what 0$. 

    Remind the definition of $\what\theta$ in the description of the model, without loss of generality we can assume that $\partial_{e_1} \theta_2 = 0$ in $\overline \La$.  Consider the following nonhomogeneous boundary value problem 
    \begin{equation}\label{eq:vHeatAppendix}
      \left\{ \begin{array}{rl}
        \partial_t v & = \; \Delta v + \Delta \theta_2,\\
        v_0 & = \; \gamma_2 - \theta_2 , \quad\text{a.e. in $\La$,} \\
        v\vert_{\Ga^-} & = \; 0 , \\ 
        \partial_{e_1} v \vert_{\Ga^+} &= \; 0 , \\
        v(t,u) & \in [0,1] \; \text{a.e. in $[0,T]\times \La$.}
      \end{array} \right.
    \end{equation}
    
    A simple computation shows that if $v$ is a solution of \eqref{eq:vHeatAppendix} then $\phi \coloneqq v + \theta_2$ is solution of the heat equation \eqref{eq:heatPhi}, in particular the problem \eqref{eq:vHeatAppendix} has  at most one solution. From the regularity of $\theta_2$, to prove the first statement of the lemma it suffices to show that \eqref{eq:vHeatAppendix} has a solution in $ \mathcal C^2([0,T]\times \La) \cap  \mathcal C([0,T]\times \overline\La)$.  

    Recall the definition of $(V_n)_{n\in\N^d}$ in \eqref{A-Vn}, we saw in the appendix \ref{AppendixBasisEigen} that these functions are a complete family of eigensolutions of the problem 
    \begin{align*}
      \left\{ \begin{array}{rl}
        - \Delta \phi &= \la \phi, \\
        \phi\vert_{\Ga^-} &= 0 , \\ 
        \partial_{e_1} \phi \vert_{\Ga^+} &= 0 .
      \end{array}  \right.
    \end{align*}
    Because $\N^d$ is a countable set, we can denote this family of functions by $(V_n)_{n\in\N}$ with eigenvalues $(\la_n)_{n\in\N}$. We can now define the heat kernels on the time interval $(0,T]$, on the sets $\overline\La$  by the following expression 
    \begin{align*}
      p(t,u,v) &\coloneqq \sum_{n\in\N} e^{-\la_n t} V_n(u) V_n(v),
    \end{align*}
    for $t\in [0,T]$ and $u,v\in \overline\La$.
    One can check that on $(0,+\infty)$, the function $p$ solves $\partial_t p = \Delta_u p = \Delta_v p$ and that $p(0,u,v) = \delta_u(v)$. For $t\in (0,T]$, let  the function $\mathcal{V} : (0,T] \times \La \longrightarrow\R$ be defined by 
    \[
      \mathcal{V}(t,u) \coloneqq (p * (\gamma_2 - \theta_2))(t,u) + \int_0^t (p* \Delta \theta_2)(s,u) \, ds .
    \]
    
    A simple computation shows that $\mathcal{V} \in  \mathcal C^\infty((0,T]\times \La) $ and solves the equation  \eqref{eq:vHeatAppendix}. As a consequence, $\mathcal{V} + \theta_2$ solves the heat equation \eqref{eq:heatPhi} and by uniqueness, the solution of \eqref{eq:heatPhi} belongs to $ \mathcal C^2([0,T]\times \La) \cap  \mathcal C([0,T]\times \overline\La)$.

    We now prove the second property of the lemma. Let $\phi$ be the solution of \eqref{eq:heatPhi}, then by definition and by the previous result just proved, $0 \leq \phi \leq 1$ in $[0,T]\times \overline \La$. It follows from Theorem 3.3.5 of \cite{pw} that for all $(t,u) \in (0,T] \times \overline\La$, $\phi(t,u) < 1$. Indeed, suppose that there exists $(t,u) \in (0,T] \times \La$ such that $\phi(t,u) = 1$, then $\phi = 1$ on $[t,T]\times \overline\La$, which is in contradiction with boundary conditions \eqref{eq:cC}.

    Fix $T>\delta > 0$. Since the set $K_\delta \coloneqq [\delta,T]\times \overline\La$  is compact and $\phi$ is continuous on this compact, we have that $\displaystyle\max_{(t,u) \in K_\delta} \phi(t,u) < 1$, this concludes the proof of the fact that $\phi$ is bounded above by $1-\vare$ for some $\vare >0$. We prove similarly that the solution is bounded below by $\vare$ considering $\tilde \phi \coloneqq 1 - \phi$.
  \end{proof}
  
  \begin{lemma}\label{rho0Iepsilon}
    The unique solution $\what \rho^0\coloneqq (m^0,\phi^0)$ of the heat equation \eqref{Eq:heat} belongs to \linebreak
    $\left( \mathcal C^2([0,T]\times \La) \cap  \mathcal C([0,T]\times \overline\La)\right)^2$ and for all $0< \delta \leq T$, there exists $\vare > 0$, such that $(m^0,\phi^0)\in \mathbf I_\vare$ in $[\delta,T]\times \La$. 
  \end{lemma}
  \begin{proof}
    The unique solution to the heat equation \eqref{Eq:heat} is warranted by Theorem \ref{th-hy} and Proposition \ref{Uniqueness} with $\what E= \what 0$. Let $(m,\phi)$ be such a solution and fix $0< \delta \leq T$. Then, from the Lemma \ref{lemma:heatBORNEE}, $\phi \in  \mathcal C^2([0,T]\times \La) \cap  \mathcal C([0,T]\times \overline\La)$ and there exists $\vare >0$ such that $\vare \leq \phi \leq 1 - \vare$ in $[\delta,T]\times \overline \La$. Let us define $\phi_\pm \coloneqq \frac{\phi \pm m}{2}$, then both functions $\phi_-$ and $\phi_+$ take their values in $[0,T]$ and, by linearity, both functions are solution to the heat equations  \eqref{eq:heatPhi}. It follows from Lemma \ref{lemma:heatBORNEE}, that $\phi_-, \phi_+ \in  \mathcal C((0,T]\times \overline\La)$  and there exists $\vare_-,\vare_+>0$ such that $\vare_\pm \leq \phi_\pm \leq 1 - \vare_\pm$ in $[\delta , T]\times \La$. To conclude the proof, it is enough to take $\vare \coloneqq \min (\vare_-,\vare_+)$.  
  \end{proof}
  \section{Miscellaneous}\label{Miscellaneous}
  \begin{lemma}\label{supMfiniD0TM}
    Define for positive functions $J_1,J_2 \in  \mathcal C^\infty(\Lambda)$, the function on $\DM$, 
    \begin{equation}\label{definitionM}
      \mathfrak M_{\what J}(\what \pi) \coloneqq \sup_{0\leq s \leq T} \left\{ \scal{}{\pi^2_s}{J_1} - \scal{}{1}{J_1} + \scal{}{\abs{\pi^1_s}}{J_2} - \scal{}{\pi^2_s}{J_2} \right\}, 
    \end{equation}
    where $\what J \coloneqq (J_1,J_2)$. Then for all $\what \pi \in \DM$, 
    \[
      \sup_{\what J} \mathfrak M_{\what J}(\what \pi) < + \infty \qquad \iff \qquad \what \pi \in \mathcal{D}([0,T],\wm) ,
    \]
    where the supremum is taken over all positive functions $J_1,J_2 \in  \mathcal C^\infty (\La)$. In that case $\sup_{\what J} \mathfrak M_{\what J}(\what \pi) = 0$.
  \end{lemma}
  \begin{proof}
    If $\what \pi \in \mathcal{D}([0,T],\wm)$, then from the definition of $\wm$, $\what \pi$ belongs to $\overline{\mathbf I}$ defined in \eqref{cC0b}, thus $\sup_{\what J} \mathfrak{M}_{\what J}(\what \pi) = 0$.  Conversely, for $\what \pi \in \DM$  such that $\sup_{\what J} \mathfrak{M}_{\what J}(\what \pi) $ is finite, taking $J_2 = 0$,  there exists $C>0$ verifying  
    \[
        \sup_{0\leq s \leq T }\scal{}{\pi^2_s}{J} - \scal{}{1}{J} \leq C ,
    \]
    for all positive function $J\in  \mathcal C^\infty(\La)$. In particular for all positive function $J\in  \mathcal C^\infty(\La)$ and $A>0$ 
    \[
      \sup_{0\leq s \leq T }\scal{}{\pi^2_s}{ A J} - \scal{}{1}{ A J} \leq C ,
    \]
    dividing all this inequality by $A$, we get 
    \[
        \sup_{0\leq s \leq T }\scal{}{\pi^2_s}{J} - \scal{}{1}{J} \leq \frac{C}{A} .
    \]
    Taking $A\to + \infty$, we actually have that for all positive function $J\in  \mathcal C^\infty(\La)$, 
    \[
      \forall~ s \in [0,T], \qquad \scal{}{\pi^2_s}{J}   \leq \scal{}{1}{J} .
    \] 
    Additionally, applying the same arguments taking $J_1 = 0$, we have for all positive function $J\in  \mathcal C^\infty(\La)$ that 
    \[
      \forall ~ s \in [0,T], \qquad \scal{}{\abs{\pi^1_s}}{J} \leq \scal{}{\pi^2_s}{J}.
    \]
    In particular, for all $s\in [0,T]$, we have that $\pi^2_s$ is a positive measure and approximating any indicator function by positive smooth function, we get that $\pi^2_s$ is absolutely continuous with respect to the Lebesgue measure. Moreover, from the last inequality we also get that for all $s\in [0,T]$,  $\abs{\pi^1_s}$ is absolutely continuous with respect to the Lebesgue measure, same applies for $\pi^1_s$ and the inequalities give us that $\what \pi \in \mathcal{D}([0,T],\wm)$. 
  \end{proof}

\bigskip

\noindent \textbf{Acknowledgments} The authors thank the referees for their careful reading of
the manuscript and their comments which helped to improve presentation of the
article.

\medskip
\noindent \textbf{Funding} No funding was received for this manuscript.
\medskip

\noindent \textbf{Conflict of Interest} The authors declare no conflicts of interest.

\medskip

\noindent \textbf{Data availability} This article has no associated data.

\bigskip

\bibliographystyle{abbrvurl}
\bibliography{biblio.bib}

\begin{thebibliography}{10}

\bibitem{bdgjl3}
L.~Bertini, A.~De~Sole, D.~Gabrielli, G.~Jona-Lasinio, and C.~Landim.
\newblock Large {Deviations} for the {Boundary} {Driven} {Symmetric} {Simple} {Exclusion} {Process}.
\newblock {\em Mathematical Physics, Analysis and Geometry}, 6(3):231--267, aug 2003.
\newblock \href {https://doi.org/10.1023/A:1024967818899} {\path{doi:10.1023/A:1024967818899}}.

\bibitem{blm09}
L.~Bertini, C.~Landim, and M.~Mourragui.
\newblock Dynamical large deviations for the boundary driven weakly asymmetric exclusion process.
\newblock {\em Ann. Probab.}, 37(6):2357--2403, 2009.
\newblock \href {https://doi.org/10.1214/09-AOP472} {\path{doi:10.1214/09-AOP472}}.

\bibitem{b1}
M.~Blume.
\newblock Theory of the first-order magnetic phase change in u${\mathrm{o}}_{2}$.
\newblock {\em Phys. Rev.}, 141:517--524, Jan 1966.
\newblock \href {https://doi.org/10.1103/PhysRev.141.517} {\path{doi:10.1103/PhysRev.141.517}}.

\bibitem{beg}
M.~Blume, V.~J. Emery, and R.~B. Griffiths.
\newblock Ising model for the $\ensuremath{\lambda}$ transition and phase separation in ${\mathrm{he}}^{3}$-${\mathrm{he}}^{4}$ mixtures.
\newblock {\em Phys. Rev. A}, 4:1071--1077, Sep 1971.
\newblock \href {https://doi.org/10.1103/PhysRevA.4.1071} {\path{doi:10.1103/PhysRevA.4.1071}}.

\bibitem{c1}
H.~Capel.
\newblock Phase transitions in spin-one {Ising} systems.
\newblock {\em Physics Letters}, 23(5):327--328, Oct. 1966.
\newblock URL: \url{https://linkinghub.elsevier.com/retrieve/pii/0031916366900230}, \href {https://doi.org/10.1016/0031-9163(66)90023-0} {\path{doi:10.1016/0031-9163(66)90023-0}}.

\bibitem{cjlmm}
E.~N.~M. Cirillo, N.~Jävergård, R.~Lyons, A.~Muntean, and S.~A. Muntean.
\newblock 3d morphology formation in a mixture of three differently averse components.
\newblock {\em Modelling and Simulation in Materials Science and Engineering}, 33(5):055014, jun 2025.
\newblock \href {https://doi.org/10.1088/1361-651X/ade4e6} {\path{doi:10.1088/1361-651X/ade4e6}}.

\bibitem{CLMM}
E.~N.~M. Cirillo, R.~Lyons, A.~Muntean, and S.~A. Muntean.
\newblock {\em Pattern Formation in Three-State Systems: Towards Understanding Morphology Formation in the Presence of Evaporation}, pages 71--85.
\newblock Springer Nature Switzerland, Cham, 2025.
\newblock \href {https://doi.org/10.1007/978-3-031-84379-2_6} {\path{doi:10.1007/978-3-031-84379-2_6}}.

\bibitem{derrida_large_2002}
B.~Derrida, J.~L. Lebowitz, and E.~R. Speer.
\newblock Large {Deviation} of the {Density} {Profile} in the {Steady} {State} of the {Open} {Symmetric} {Simple} {Exclusion} {Process}.
\newblock {\em Journal of Statistical Physics}, 107(3-4):599--634, May 2002.
\newblock URL: \url{https://link.springer.com/10.1023/A:1014555927320}, \href {https://doi.org/10.1023/A:1014555927320} {\path{doi:10.1023/A:1014555927320}}.

\bibitem{erignoux_hydrodynamics_2020}
C.~Erignoux, P.~Gonçalves, and G.~Nahum.
\newblock Hydrodynamics for {SSEP} with non-reversible slow boundary dynamics: {Part} {I}, the critical regime and beyond.
\newblock {\em Journal of Statistical Physics}, 181(4):1433--1469, Nov. 2020.
\newblock arXiv:1912.09841 [math-ph].
\newblock URL: \url{http://arxiv.org/abs/1912.09841}, \href {https://doi.org/10.1007/s10955-020-02633-w} {\path{doi:10.1007/s10955-020-02633-w}}.

\bibitem{flm}
J.~Farfan, C.~Landim, and M.~Mourragui.
\newblock Hydrostatics and dynamical large deviations of boundary driven gradient symmetric exclusion processes.
\newblock {\em Stochastic Processes and their Applications}, 121(4):725--758, 2011.
\newblock \href {https://doi.org/10.1016/j.spa.2010.11.014} {\path{doi:10.1016/j.spa.2010.11.014}}.

\bibitem{fgln}
T.~Franco, P.~Gon{\c{c}}alves, C.~Landim, and A.~Neumann.
\newblock Dynamical large deviations for the boundary driven symmetric exclusion process with {Robin} boundary conditions.
\newblock {\em ALEA, Lat. Am. J. Probab. Math. Stat.}, 19(2):1497--1546, 2022.
\newblock \href {https://doi.org/10.48550/arXiv.2203.14417} {\path{doi:10.48550/arXiv.2203.14417}}.

\bibitem{fgn1}
T.~Franco, P.~Gon\c{c}alves, and A.~Neumann.
\newblock Hydrodynamical behavior of symmetric exclusion with slow bonds.
\newblock {\em Annales de l'I.H.P. Probabilit\'es et statistiques}, 49(2):402--427, 2013.
\newblock \href {https://doi.org/10.1214/11-AIHP445} {\path{doi:10.1214/11-AIHP445}}.

\bibitem{fgn3}
T.~Franco, P.~Gonçalves, and A.~Neumann.
\newblock Large deviations for the ssep with slow boundary: the non-critical case.
\newblock {\em Latin American Journal of Probability and Mathematical Statistics}, 20:359, 01 2023.
\newblock \href {https://doi.org/10.30757/ALEA.v20-13} {\path{doi:10.30757/ALEA.v20-13}}.

\bibitem{gpv}
M.~Z. Guo, G.~C. Papanicolaou, and S.~R.~S. Varadhan.
\newblock Nonlinear diffusion limit for a system with nearest neighbor interactions.
\newblock {\em Commun. Math. Phys.}, 118(1):31--59, 1988.
\newblock \href {https://doi.org/10.1007/BF01218476} {\path{doi:10.1007/BF01218476}}.

\bibitem{kuh3}
P.~C. Hemmer, M.~Kac, and G.~E. Uhlenbeck.
\newblock On the van der waals theory of the vapor‐liquid equilibrium. iii. discussion of the critical region.
\newblock {\em Journal of Mathematical Physics}, 5(1):60--74, 01 1964.
\newblock \href {https://doi.org/10.1063/1.1704065} {\path{doi:10.1063/1.1704065}}.

\bibitem{kuh}
M.~Kac, G.~E. Uhlenbeck, and P.~C. Hemmer.
\newblock On the van der waals theory of the vapor‐liquid equilibrium. i. discussion of a one‐dimensional model.
\newblock {\em Journal of Mathematical Physics}, 4(2):216--228, 02 1963.
\newblock \href {https://doi.org/10.1063/1.1703946} {\path{doi:10.1063/1.1703946}}.

\bibitem{kl}
C.~Kipnis and C.~Landim.
\newblock {\em Scaling limits of interacting particle systems}, volume 320 of {\em Grundlehren Math. Wiss.}
\newblock Berlin: Springer, 1999.
\newblock \href {https://doi.org/10.1007/978-3-662-03752-2} {\path{doi:10.1007/978-3-662-03752-2}}.

\bibitem{liggett}
T.~M. Liggett.
\newblock {\em Stochastic {Interacting} {Systems}: {Contact}, {Voter} and {Exclusion} {Processes}}, volume 324 of {\em Grundlehren der mathematischen {Wissenschaften}}.
\newblock Springer Berlin Heidelberg, Berlin, Heidelberg, 1999.
\newblock \href {https://doi.org/10.1007/978-3-662-03990-8} {\path{doi:10.1007/978-3-662-03990-8}}.

\bibitem{lcm}
R.~Lyons, E.~N. Cirillo, and A.~Muntean.
\newblock Phase separation and morphology formation in interacting ternary mixtures under evaporation: Well-posedness and numerical simulation of a non-local evolution system.
\newblock {\em Nonlinear Analysis: Real World Applications}, 77:104039, 2024.
\newblock \href {https://doi.org/10.1016/j.nonrwa.2023.104039} {\path{doi:10.1016/j.nonrwa.2023.104039}}.

\bibitem{LMN}
R.~Lyons, A.~Muntean, and G.~Nika.
\newblock A {Bound} {Preserving} {Energy} {Stable} {Scheme} for a {Nonlocal} {Cahn{\textendash}Hilliard} {Equation}.
\newblock {\em Comptes Rendus. M\'ecanique}, 352:239--250, 2024.
\newblock \href {https://doi.org/10.5802/crmeca.265} {\path{doi:10.5802/crmeca.265}}.

\bibitem{mm1}
R.~Marra and M.~Mourragui.
\newblock Phase segregation dynamics for the blume–capel model with kac interaction.
\newblock {\em Stochastic Processes and their Applications}, 88(1):79--124, 2000.
\newblock \href {https://doi.org/10.1016/S0304-4149(99)00120-9} {\path{doi:10.1016/S0304-4149(99)00120-9}}.

\bibitem{mo1}
M.~Mourragui and E.~Orlandi.
\newblock Boundary driven kawasaki process with long-range interaction: dynamical large deviations and steady states.
\newblock {\em Nonlinearity}, 26(1):141, nov 2012.
\newblock \href {https://doi.org/10.1088/0951-7715/26/1/141} {\path{doi:10.1088/0951-7715/26/1/141}}.

\bibitem{msv}
M.~Mourragui, E.~Saada, and S.~Velasco.
\newblock Hydrodynamic and hydrostatic limit for a generalized contact process with mixed boundary conditions.
\newblock {\em Electron. J. Probab.}, 28:44, 2023.
\newblock Id/No 155.
\newblock \href {https://doi.org/10.1214/23-EJP1025} {\path{doi:10.1214/23-EJP1025}}.

\bibitem{pw}
M.~H. Protter and H.~F. Weinberger.
\newblock {\em Maximum {Principles} in {Differential} {Equations}}.
\newblock Springer New York, New York, NY, 1984.
\newblock \href {https://doi.org/10.1007/978-1-4612-5282-5} {\path{doi:10.1007/978-1-4612-5282-5}}.

\bibitem{Roubi}
T.~Roubíček.
\newblock {\em Nonlinear {Partial} {Differential} {Equations} with {Applications}}, volume 153 of {\em {ISNM} {International} {Series} of {Numerical} {Mathematics}}.
\newblock Birkhäuser-Verlag, Basel, 2005.
\newblock \href {https://doi.org/10.1007/3-7643-7397-0} {\path{doi:10.1007/3-7643-7397-0}}.

\bibitem{kuh2}
G.~E. Uhlenbeck, P.~C. Hemmer, and M.~Kac.
\newblock On the van der waals theory of the vapor‐liquid equilibrium. ii. discussion of the distribution functions.
\newblock {\em Journal of Mathematical Physics}, 4(2):229--247, 02 1963.
\newblock \href {https://doi.org/10.1063/1.1703947} {\path{doi:10.1063/1.1703947}}.

\end{thebibliography}

\end{document}